\documentclass[11pt]{amsart}

 \usepackage[utf8]{inputenc}
 \usepackage[T1]{fontenc}

 \usepackage[english]{babel}

 \usepackage{amsmath,amsfonts,amssymb}
 \usepackage{mathrsfs}
 \usepackage{dsfont} 
 \usepackage{mathtools} 
 \usepackage{esint} 
 \usepackage{xfrac} 
 \usepackage{defs} 
 \usepackage{fonts} 

 \usepackage{amsthm}
 \usepackage[shortlabels]{enumitem} 
 \usepackage[colorlinks=true,urlcolor=blue, citecolor=red,linkcolor=blue,linktocpage,pdfpagelabels, bookmarksnumbered,bookmarksopen]{hyperref} 
 \usepackage[paper=a4paper,hmargin=1in,vmargin=1in, marginparwidth=0.8in]{geometry} 

 \usepackage[colorinlistoftodos,prependcaption,linecolor=blue,backgroundcolor=blue!25,bordercolor=blue,textsize=tiny]{todonotes}
 \usepackage{cite} 
 \usepackage{cleveref}

plus 6pt minus 12pt
\numberwithin{equation}{section}

\newtheorem{theorem}{Theorem}[section]
\newtheorem{proposition}[theorem]{Proposition}
\theoremstyle{definition}

\theoremstyle{plain}
\newtheorem{lemma}[theorem]{Lemma}
\newtheorem{corollary}{Corollary}[theorem]
\theoremstyle{remark}

\title[The Fractional Lam\'e-Navier Operator]{The Fractional Lam\'e-Navier Operator: Appearances, Properties and Applications}
 \author{James M. Scott}
 \address{Department of Applied Physics and Applied Mathematics,
 Columbia University,
 New York, NY 10027}
 \email{jms2555@columbia.edu}

\thanks{
	Support by the NSF DMS 1937254 is gratefully acknowledged. The author also thanks Tadele Mengesha and Pablo Ra\'ul Stinga for helpful and illuminating conversations.}



\begin{document}

\maketitle

\begin{abstract}

We introduce and analyze an explicit formulation of fractional powers of the Lam\'e-Navier system of partial differential operators.
We show that this \textit{fractional Lam\'e-Navier operator} is a nonlocal integro-differential operator that appears in several widely-used continuum mechanics models.
We demonstrate that the fractional Lam\'e-Navier operator can be obtained using compositions of nonlocal gradient operators.
Additionally, the effective form of the fractional Lam\'e-Navier operator is the same as the operator obtained as a particular choice of parameters in state-based peridynamics.
We further show that the Dirichlet-to-Neumann map associated to the local classical Lam\'e-Navier system posed in a half-space coincides with the square root power of the Lam\'e-Navier operator for a particular choice of elastic coefficients.

We establish basic analysis results for the fractional Lam\'e-Navier operator, including the calculus of positive and negative powers, and explore its interaction with the H\"older and Bessel classes of functions.
We also derive the fractional Lam\'e-Navier as the Dirichlet-to-Neumann map of a local degenerate elliptic system of equations in the upper half-space. We use an explicit formula for a Poisson kernel for the extension system to establish the well-posedness in weighted Sobolev spaces.
As an application, we derive the equivalence of two fractional seminorms using a purely local argument in the extension system, and then use this equivalence to obtain well-posedness for a variational Dirichlet problem associated to the fractional Lam\'e-Navier operator.

\end{abstract}

\section{Introduction}

The Lam\'e-Navier operator is a partial differential operator that appears across several prominent models employed in continuum mechanics.
For example, the operator appears as the divergence of the Cauchy stress tensor for a linear isotropic medium.
For hyperelastic solids, the steady-state constitutive equations representing displacement consist of the Lam\'e-Navier operator equated with the external force. For compressible fluids, the operator appears in the Navier-Stokes momentum equations.
In this work the Lam\'e-Navier operator $\bbL$ acting on the vector field $\bu : \bbR^d \to \bbR^d$ for $d \geq 2$ has the explicit form
\begin{equation}\label{eq:LameOperator}
	\bbL \bu(\bx) := -  \mu \Delta \bu(\bx) - (\mu + \lambda) \grad(\div \bu(\bx))\,.
\end{equation}
The constants $\mu$ and $\lambda$ are the Lam\'e constants (representing the bulk and shear modulus respectively in the context of solid materials), and they satisfy the conditions for rank-one / Legendre-Hadamard ellipticity of $\bbL$:
\begin{equation}\label{eq:EllipticityConditionsLame}
	\mu > 0\,, \qquad 2 \mu + \lambda > 0\,.
\end{equation}

We are interested in fractional powers of the Lam\'e-Navier operator, which we denote as $\bbL^s$ for $s \in (0,1)$.
Our chief motivation is to gain further understanding of nonlocal systems in continuum mechanics by viewing them through the lens of the analysis of fractional operators.
Fractional operators for scalar-valued functions have appeared in a wealth of settings, including obstacle problems \cite{athanasopoulos2006optimal, athanasopoulos2008structure}, image processing \cite{Gilboa-Osher}, flocking and emergent dynamics \cite{tadmor, tadmor2014critical}, quasi-geostrophic equations \cite{cordoba2004maximum}, and Markov processes \cite{barlow2009non}.
A widely-used operator is the fractional power of the Laplacian $-\Delta$. It is defined for sufficiently smooth functions $u : \bbR^d \to \bbR$ through the Fourier inversion formula
\begin{equation}\label{eq:FourierSymbol:FractionalLaplacian}
	\cF((-\Delta)^s u)(\bsxi) = \big( 2 \pi |\bsxi| \big)^{2s} \hat{u}(\bsxi)\,,
\end{equation}
where $\cF$ denotes the Fourier transform
\begin{equation*}
	\cF(u)(\bsxi) = \wh{u}(\bsxi) = \int_{\bbR^d} \rme^{-2 \pi \imath \bx \cdot \bsxi} u(\bx) \, \rmd \bx\,, \qquad u \in L^1(\bbR^d)\,,
\end{equation*}
and $\cF^{-1}$ denotes its inverse.
The fractional Laplacian $(-\Delta)^s$ also has an integral form as
\begin{equation}\label{eq:IntegralFormula:FractionalLaplacian}
	(-\Delta)^s u(\bx) := c_{d,s} \pv \intdm{\bbR^d}{\frac{u(\bx)-u(\by)}{|\bx-\by|^{d+2s}}}{\by}\,,
\end{equation}
where $s \in (0,1)$ and $c_{d,s}$ is a normalizing constant defined in \eqref{eq:DefnOfNormalizingConstants}. See \cite{kwasnicki2017ten} for a survey discussing the equivalence of various definitions of $(-\Delta)^s$.

The goal of this work is to introduce explicit formulae and give a firm analytical ground for fractional powers of the Lam\'e-Navier operator $\bbL$. Some of the results here have an analogue for second-order uniformly elliptic operators with constant coefficients, such as the Laplacian. However, in contrast to the Laplacian, the Lam\'e-Navier equations make up a strongly-coupled Legendre-Hadamard elliptic system acting on vector-valued functions. When stating a result for the operators $\bbL^s$, we remark when necessary on the similarity of the proof to its already-existing analogue for scalar operators.

Our process to define the fractional Lam\'e-Navier operator $\bbL^s$ starts in the same way as that of $(-\Delta)^s$; we find an explicit formula for the fractional power of the multiplier matrix associated to the Lam\'e-Navier operator. 
This symmetric matrix, which we denote by $\bM(\bsxi)$, is readily seen to be
\begin{equation}\label{eq:LameEquation:FourierSymbol}
	\bM(\bsxi) := \mu ( 2 \pi |\bsxi| )^2 \bI + (\mu + \lambda ) \big( 2 \pi \big)^2 \bsxi \otimes \bsxi\,, \qquad \bsxi \in \bbR^d\,,
\end{equation}
where $\bI$ denotes the $d \times d$ identity matrix. Thanks to the rank-one ellipticity conditions \eqref{eq:EllipticityConditionsLame} the matrix $\bM(\bsxi)$ is positive definite:
\begin{equation*}
	\Vint{ \bM(\bsxi) \bseta,\bseta } \geq \min \{ \mu, 2 \mu + \lambda \} |\bseta|^2 \qquad \text{ for all } \bseta \in \bbR^d \text { and for all } \bsxi \in \bbR^d\,.
\end{equation*}
Therefore $\bM(\bsxi)$ has a unique positive power of $s$ for each $s \in (0,1)$, which we calculate explicitly in \Cref{lma:FourierSymbolCalculation} below. Fourier inversion (see \Cref{sec:IntroductionToOperator}) gives the integro-differential operator denoted as $\bbL^s$:

\begin{theorem}\label{thm:IntegralDefinitionForFracLame}
	Let $s \in (0,1)$. For sufficiently smooth and integrable vector fields $\bu : \bbR^d \to \bbR^d$, the fractional $s$-power of the Lam\'e-Navier operator $\bbL$ is defined via the Fourier inversion formula
	\begin{equation*}
		\cF(\bbL^s \bu)(\bsxi) := \bM^s(\bsxi) \wh{\bu}(\bsxi)\,,
	\end{equation*}
	where the action of $\cF$ is componentwise. The operator $\bbL^s$ has the explicit integral form
	\begin{equation}\label{eq:FractionalPowerOfLame}
		\begin{split}
		\bbL^s \bu(\bx) &:= \left( \frac{(2s+1)\mu^s - (2\mu+\lambda)^s}{2s}\right) c_{d,s} \, \pv \intdm{\bbR^d}{\frac{\bu(\bx) - \bu(\by)}{|\bx-\by|^{d+2s}}}{\by}  \\
		& \qquad \qquad + \frac{(2 \mu + \lambda)^s - \mu^s}{2s} \kappa_{d,s} \, \pv \intdm{\bbR^d}{ \left( \frac{(\bx-\by) \otimes (\bx-\by)}{|\bx-\by|^2} \right) \frac{\bu(\bx) - \bu(\by)}{|\bx-\by|^{d+2s}}}{\by}\,,
		\end{split}
	\end{equation}
	where $\bx \in \bbR^d$ and $c_{d,s}$, $\kappa_{d,s}$ are normalizing constants defined in \eqref{eq:DefnOfNormalizingConstants} below.
\end{theorem}

The operator $\bbL^s$ exhibits strong coupling at its principal order. Indeed, the choice of constant $\lambda = \big( (2s+1)^{1/s}-2 \big) \mu$ for any $\mu >0$ still satisfies the ellipticity conditions \eqref{eq:EllipticityConditionsLame} and the resulting form of $\bbL^s$ is
\begin{equation*}
	\bbL^s \bu(\bx) = \mu^s \kappa_{d,s} \, \pv \intdm{\bbR^d}{ \left( \frac{(\bx-\by) \otimes (\bx-\by)}{|\bx-\by|^2} \right) \frac{\bu(\bx) - \bu(\by)}{|\bx-\by|^{d+2s}}}{\by}\,,
\end{equation*}
so each component of $\bbL^s \bu(\bx)$ involves all components of $\bu$ via the projected difference quotient $(\bx-\by) \cdot (\bu(\bx)-\bu(\by) )$.
Theoretical results for scalar integro-differential equations (like those involving the fractional Laplacian) may not hold for the coupled system $\bbL^s$. For example, maximum principle techniques cannot be used.
Nonlocal and fractional operators consisting exclusively of strong coupling have been studied in \cite{scott2020mathematical} and we will refer to those results as necessary throughout this work. 

Although we introduce this operator using purely mathematical formalism, we demonstrate that it appears in several different modeling applications: general nonlocal vector calculus, state-based peridynamics, and the Dirichlet-to-Neumann map for elastic half-spaces. In what follows we describe each of these in detail.

In the last several years notions of nonlocal vector calculus have been studied in various contexts; see \cite{alali2014generalized, Du-NonlocalCalculus, d2020towards, vsilhavy2020fractional}. 
We show that, in one of these contexts, the fractional Lam\'e-Navier equation can be recovered as a linear combination of compositions of the ``fractional divergence'' $\div^s$ and the ``fractional gradient'' $\grad^s$.
\begin{theorem}[See \Cref{sec:Applications}]\label{thm:Intro:NonlocalCalculus}
	For sufficiently smooth functions $\bu : \bbR^d \to \bbR^d$,
	\begin{equation*}
		\bbL^s \bu(\bx) = - \mu^s (-\Delta)^s \bu(\bx) - \big( (2\mu+\lambda)^s - \mu^s \big) \grad^s \div^s \bu(\bx)\,.
	\end{equation*}
\end{theorem}

Peridynamics is a nonlocal continuum theory that has been applied in a variety of modeling settings; for some recent developments see the book \cite{bobaru2016handbook} and the references therein.
A class known as \textit{state-based} peridynamic systems was introduced in \cite{Silling2007} in order to model linear elastic materials with arbitrary Poisson ratio in the context of peridynamics.
Mathematical theory for general state-based peridynamic operators has been developed in \cite{MengeshaDuElasticity, alali2015peridynamics, mengesha2015VariationalLimit}.
State-based peridynamic operators involve a double integration. However, in \cite{vsilhavy2017higher} it was noted that this double integration can, at least formally, be simplified so that any state-based peridynamic operator can be written as a sum of ``traditional'' integro-differential operators. In \Cref{thm:FractionalLame:StateBasedPeri} we show that in fact a state-based peridynamic operator coincides with the fractional Lam\'e-Navier operator for a particular choice of parameters. In other words, a special case of the peridynamic theory allows for the recovery of the exact fractional model introduced in this work.

The third application appears in a classical partial differential model, namely the Lam\'e-Navier system of equations which describe an infinite isotropic hyperelastic half-space.
Its Dirichlet-to-Neumann map (which is a nonlocal operator) has been used to great effect in recent works that analyze the Peierls-Nabarro model of dislocations \cite{gao2021existence, gao2021revisit}.
\begin{theorem}[See \Cref{sec:Applications}]
	Up to a choice of elastic constants, the Dirichlet-to-Neumann map associated to a Lam\'e-Navier-harmonic system in the upper-half space coincides with that of the operator $\bbL^{\frac{1}{2}}$.
\end{theorem}
It should be emphasized that, in contrast with the Laplacian, the Dirichlet-to-Neumann map for a general elliptic equation or system will \textit{not} be the square root of the same operator. This is also the case with a general Lam\'e-Navier system. Nevertheless, the effective form of the integral operators involved are the same; see \Cref{sec:Applications} for details.

With these applications in mind, we turn to analysis. The fundamental solution defined by an exact formula in \eqref{eq:Definition:LameNavierKernel} is used to establish basic calculus properties for $\bbL^s \bu$ and their resolvents $(\bbL^{s})^{-1} \bu = \bbL^{-s} \bu$; see \Cref{thm:CalcProperties}. These properties are then used to investigate the interaction of $\bbL^s$ with the classical H\"older spaces $C^{k,\alpha}$ in \Cref{sec:HolderSpaces}.
We establish basic mapping properties, and prove regularity results such as local Schauder and $C^{1,\alpha}$ estimates for solutions $\bu$ to $\bbL^s \bu = \bff$ using the explicit form of the fundamental solution.

In the range of Bessel potential spaces, we prove a Sobolev inequality for $\bbL^{-s}$, and demonstrate both a distributional and pointwise definition of $\bbL^s$. To be precise, we are able to prove an alternative characterization of Bessel potential vector fields in terms of the $L^p$-convergence of the principal-value integrals defining $\bbL^s$; see \Cref{thm:BesselSpaces:MainThm} for the statement.
Although the same result for $(-\Delta)^s$ is well-known to experts, we could not find a direct reference to a detailed proof in the literature\footnote{There is a sketch of the proof in \cite{Stein}, but we are unable to fill in the details.}. Therefore we have included a comprehensive proof of the alternative characterization for $\bbL^s$ in Sections \ref{sec:BesselSpaces} and \ref{sec:BesselSpacesProof}. The proof is inspired by the results in \cite{wheeden1968hypersingular} for higher-order hypersingular operators.
As a consequence of the pointwise definition we prove strong solvability in the scale of Bessel spaces of the fractional elliptic system 
\begin{equation*}
	\bbL^s \bu + q \bu = \bff\,, \qquad q > 0\,,
\end{equation*}
posed on all of Euclidean space.

The fact that $\bbL^s$ is derived as the fractional power of an operator allows us to bring powerful analysis tools to bear that are unavailable for general nonlocal operators.
One of the most celebrated techniques in the analysis of fractional operators is to analyze an extension problem for a local degenerate elliptic equation at the expense of another variable. The equation is chosen so that the fractional operator coincides with the Dirichlet-to-Neumann map. Results for the nonlocal operator are then recovered from the local analogues for the extension problem.
This was done in \cite{caffarelli2007extension} for the fractional Laplacian using modern analysis techniques; see \cite{gale2013extension, stinga2010extension, kwasnicki2018extension} for approaches to general elliptic operators. In \Cref{sec:ExtensionProblem} we show that the fractional system $\bbL^s$ can be realized as the Dirichlet-to-Neumann map of a linear degenerate elliptic system. The following theorem is the formal summary of the section's main results.

\begin{theorem}
	For any $\bu : \bbR^d \to \bbR^d$, 
	there exists a unique $\bU : \bbR^{d} \times [0,\infty) \to \bbR^d$ belonging to an appropriately weighted Sobolev space that solves the partial differential system
	\begin{equation}\label{eq:Intro:ExtProb}
		\begin{cases}
			\p_{tt} \bU(\bx,t) + \frac{1-2s}{t} \p_t \bU(\bx,t) + \mu \grad_{\bx}\bU(\bx,t) + (\mu + \lambda) \grad_{\bx} \div_{\bx} \bU(\bx,t) = {\bf 0}\,, \\
			\bU(\bx,0) = \bu(\bx)\,.
		\end{cases}
	\end{equation}
	The Dirichlet-to-Neumann map 
	$$
	- \lim\limits_{t \to 0} t^{1-2s} \p_t \bU(\bx,t)
	$$ 
	coincides with $\bbL^s \bu(\bx)$ up to a multiplicative constant depending only on $s$.
\end{theorem}
Note that this extension system is strongly-coupled, and so the techniques used in \cite{fabes1982local, cabre2014nonlinear} to analyze the analogous extension problem for $(-\Delta)^s$ are not available. For example, a Harnack inequality for the components of $\bu$ is not expected to hold. Therefore, we state the existence of a solution to the equations using a Poisson kernel, and uniqueness using a variational framework.

The last section of the paper is devoted to a nonlocal Dirichlet problem
\begin{equation*}
	\begin{cases}
		\bbL^s \bu = \bff\,, \quad \text{ in } \Omega\,, \\
		\bu = {\bf 0}\,, \quad \text{ in } \bbR^d \setminus \Omega\,, \\
	\end{cases}
\end{equation*}
where $\Omega \subset \bbR^d$ is a bounded domain. We use Hilbert space methods to prove the well posedness of this problem in the spirit of \cite{Felsinger, KassmannMengeshaScott}. Coercivity of the associated bilinear form will be obtained by
comparing the weighted Sobolev norms of the solution to \eqref{eq:Intro:ExtProb} with the solution to the extension problem for $(-\Delta)^s$. In other words, we obtain a result for a nonlocal fractional system using the local extension system.

In general we do not give proofs on the consistency of our results with their classical counterparts as $s \to 1^-$ or $s \to 0^+$. However, the reader can verify without difficulty that many of the formulae (for fundamental solutions, Poisson kernels, etc.) agree formally with the classical ones. In the case of inequalities we will sometimes state formal consistency results as a remark.

It is the author's hope that the basic analysis contained in this work introduces another common thread of understanding between the nonlocal and fractional communities in the context of continuum mechanics modeling.
For example, in the very recent preprint \cite{silhavy2022fractional} an operator of the form $I^{1-s} \circ \bbL$ is considered in the context of fractional elasticity. A fundamental solution and a Korn-type inequality involving the fractional gradient is derived. An application of the identities in \Cref{subsec:FracVectorCalc} reveal that such an operator is actually of the form $\bbL^s$ (although with different material constants).  Therefore, some of the results for $\bbL^s$ established in this work also apply to $I^{1-s} \circ \bbL$.

The manuscript is organized as follows: The next section introduces  the definition of the fractional Lam\'e-Navier operator for smooth functions. In \Cref{sec:Applications} we demonstrate the appearance fractional Lam\'e-Navier operator in the aforementioned continuum models. \Cref{sec:Properties} contains the derivation of the fundamental solution and calculus properties of the operator. \Cref{sec:HolderSpaces} contains mapping and regularity results in the context of H\"older spaces, and \Cref{sec:BesselSpaces} contains similar results in the context of Bessel potential spaces with proofs in \Cref{sec:BesselSpacesProof}. We introduce and analyze the degenerate elliptic extension system in the upper half-space in \Cref{sec:ExtensionProblem}, and then apply it to solve a variational problem in \Cref{sec:VariationalProblem}. Fourier transform formulae and classical identities for special functions are collected in the appendix.

\section{The Fractional Lam\'e-Navier Operator for Smooth Functions}\label{sec:IntroductionToOperator}

To obtain the integral form of $\bbL^s$ in \Cref{thm:IntegralDefinitionForFracLame} we build on previous mathematical analysis for fractional operators that exclusively consist of strong coupling; this work can be found in \cite{scott2020mathematical, mengesha2020solvability, KassmannMengeshaScott, MengeshaScott2018Korn, MengeshaScott2018Potential}. First, we find the algebraic expression for fractional powers of the matrix $\bM(\bsxi)$.

\begin{lemma}\label{lma:FourierSymbolCalculation}
	Let $s > 0$. Then for all $\bsxi \in \bbR^d$
	\begin{equation}\label{eq:FractionaLameFourierMatrix}
		\bM^s(\bsxi) := \big( \bM(\bsxi) \big)^s = \mu^s (2 \pi |\bsxi| )^{2s} \bI + \big( (2\mu + \lambda)^s - \mu^s \big) (2 \pi |\bsxi|)^{2s} \frac{\bsxi \otimes \bsxi}{|\bsxi|^2}\,.
	\end{equation}
\end{lemma}

\begin{proof}
	For any $d \times d$ symmetric positive definite matrix $\bA$ written in diagonal form,
	\begin{equation*}
		\bA = \bU^{-1} \diag(\lambda_1, \ldots, \lambda_d) \bU \quad \Rightarrow \quad \bA^s = \bU^{-1} \diag(\lambda_1^s, \ldots, \lambda_d^s) \bU\,.
	\end{equation*}
	We proceed to diagonalize $\bM$. For a fixed $\bsxi \in \bbR^d$ not equal to ${\bf 0}$, let $\bR(\bsxi)$ be a rotation (i.e. $\bR^T \bR = \bI$) such that $\bR(\bsxi) \bsxi = |\bsxi| \be_1$, where $\be_1$ is the unit vector $(1,0,\ldots,0)$. Then
	\begin{equation*}
		\begin{split}
			\bM(\bsxi) &= \mu  ( 2 \pi |\bsxi| )^2 \bI + (\mu + \lambda ) \big( 2 \pi \big)^2 \big( \bR^T(\bsxi) |\bsxi| \be_1 \big) \otimes \big( \bR^T(\bsxi) |\bsxi| \be_1 \big) \\
			&= \mu  ( 2 \pi |\bsxi| )^2 \bI + (\mu + \lambda ) (2 \pi |\bsxi|)^2 \bR^T(\bsxi) \big( \be_1 \otimes \be_1 \big) \bR(\bsxi)\,.
		\end{split}
	\end{equation*}
	Since $\bR$ is a rotation, we can write
	\begin{equation*}
		\begin{split}
			\bM(\bsxi) &=  ( 2 \pi |\bsxi| )^2  \bR^T(\bsxi) \Big( \mu \bI + (\mu + \lambda ) \big( \be_1 \otimes \be_1 \big) \Big) \bR(\bsxi) \\
			&= ( 2 \pi |\bsxi| )^2 \bR^T(\bsxi) \diag(2 \mu + \lambda, \mu, \ldots, \mu) \bR(\bsxi)\,.
		\end{split}
	\end{equation*}
	Therefore, we take the $s$ power, and then use that $\bR^T \be_1 = \frac{\bsxi}{|\bsxi|}$:
	\begin{equation*}
		\begin{split}
			\bM^s(\bsxi) &= ( 2 \pi |\bsxi| )^{2s} \bR^T(\bsxi) \diag((2 \mu + \lambda)^s, \mu^s, \ldots, \mu^s) \bR(\bsxi) \\
			&= ( 2 \pi |\bsxi| )^{2s} \bR^T(\bsxi) \Big( \mu^s \bI +\big( (2 \mu + \lambda)^s - \mu^s \big) (\be_1 \otimes \be_1) \Big) \bR(\bsxi) \\
			&= ( 2 \pi |\bsxi| )^{2s} \left( \mu^s \bI +\big( (2 \mu + \lambda)^s - \mu^s \big) \left( \frac{\bsxi}{|\bsxi|} \otimes \frac{\bsxi}{|\bsxi|} \right) \right)\,.
		\end{split}
	\end{equation*}
\end{proof}

Now, we define the positive normalizing constants $c_{d,s}$ and $\kappa_{d,s}$ as
\begin{equation}\label{eq:DefnOfNormalizingConstants}
	\begin{split}
		c_{d,s} &:= \frac{2^{2s}s \Gamma(\frac{d}{2}+s)}{\pi^{d/2}\Gamma(1-s)} = \left( \intdm{\bbR^d}{\frac{1-\cos(h_1)}{|\bh|^{d+2s}}}{\bh} \right)^{-1}\,, \\
		\kappa_{d,s} &:= (d+2s)c_{d,s} = (d+2s) \frac{2^{2s} s \Gamma(\frac{d}{2}+s)}{\pi^{d/2} \Gamma(1-s)}\,,
	\end{split}
\end{equation}
where $\Gamma(a)$ for $a > 0$ denotes the $\Gamma$-function.
We denote the class of scalar-valued Schwartz functions by $\scS(\bbR^d)$, and we denote the class of $\bbR^d$-valued Schwartz vector fields by $\scS(\bbR^d;\bbR^d)$. Its dual is denoted by $\scS'$.
Define the operator $\bbF^s$ as
\begin{equation}\label{eq:Localization:DefnOfL}
	\bbF^s \bu(\bx) := \kappa_{d,s} \, \pv \int\limits_{\bbR^d} \left( \frac{(\bx-\by) \otimes (\bx - \by)}{|\bx-\by|^2} \right) \frac{\bu(\bx)-\bu(\by)}{|\bx-\by|^{d+2s}} \, \mathrm{d}\by\,,
\end{equation}
for $d \geq 2$ and $s \in (0,1)$, and $\kappa_{d,s}$ given in \eqref{eq:DefnOfNormalizingConstants}.

For smooth functions (i.e. functions in $\scS$) we can dispense with the $\pv$ in front of the integral by identifying $\bbF^s$ with the following operator characterized by the second-order difference quotient.
The proof follows by adapting line-by-line the proof of the analogous statement for $(-\Delta)^s$; see \cite[Lemma 3.2]{DNPV12}.

\begin{lemma}\label{lma:SecondOrderDiff}
	Let $\bu \in \scS(\bbR^d;\bbR^d)$. Then the characterization
	\begin{equation*}
		\bbF^s \bu(\bx) = \frac{\kappa_{d,s}}{2} \intdm{\bbR^d}{ \left( \frac{\bh \otimes \bh}{|\bh|^2} \right) \frac{2 \bu(\bx) - \bu(\bx+\bh)-\bu(\bx-\bh)}{|\bh|^{d+2s}} }{\bh}
	\end{equation*}
	holds for all $\bx \in \bbR^d$.
\end{lemma}
The operator $\bbF^s$ is introduced in \cite{scott2020mathematical}, and properties of the constants $c_{d,s}$ and $\kappa_{d,s}$ are studied there as well. In particular we have the following formula for the Fourier symbol:

\begin{theorem}[{See \cite[Section 2.1]{scott2020mathematical}}]\label{thm:FourierTransformOfL}
	Let $\bu \in \scS(\bbR^d;\bbR^d)$. 
	Then for every $\bsxi \in \bbR^d$
	\begin{equation}\label{eq:FourierTransformOfL}
		\widehat{\bbF \bu}(\bsxi) = \big( 2 \pi |\bsxi| \big)^{2s} \left( 2s \frac{\bsxi \otimes \bsxi}{|\bsxi|^2} + \bI \right) \widehat{\bu}(\bsxi)\,.
	\end{equation}
\end{theorem}

Combining \Cref{lma:FourierSymbolCalculation}, the formulae \eqref{eq:FourierSymbol:FractionalLaplacian}-\eqref{eq:IntegralFormula:FractionalLaplacian} and Theorem \ref{thm:FourierTransformOfL}, we obtain the integral characterization of $\bbL^s$:
\begin{corollary}
	The integral operator $\bbL^s$ given in \eqref{eq:FractionalPowerOfLame}
	is defined through the formula
	\begin{equation*}
		\bbL^s \bu(\bx) = \cF^{-1} \left[ \bM^s(\bsxi) \wh{\bu} \right] (\bx)
	\end{equation*}
	in the following sense: For any $\bu$ and $\bv \in \scS(\bbR^d;\bbR^d)$
	\begin{equation*}
		\begin{split}
			\intdm{\bbR^d}{\Vint{\bM^s(\bsxi) \wh{\bu}(\bsxi), \wh{\bv}(\bsxi) }}{\bsxi} = \intdm{\bbR^d}{\Vint{\bbL^s \bu(\bx), \bv(\bx) }}{\bx}\,.
		\end{split}
	\end{equation*}
\end{corollary}
The function $\bbL^s \bu(\bx)$ is clearly defined pointwise for $\bu \in \scS(\bbR^d;\bbR^d)$, and by \Cref{lma:SecondOrderDiff} the following second-order difference characterization holds:
\begin{multline*}
	\bbL^s \bu(\bx) := \left( \frac{(2s+1)\mu^s - (2\mu+\lambda)^s}{2s}\right) \frac{c_{d,s}}{2} \intdm{\bbR^d}{\frac{2\bu(\bx) - \bu(\bx+\bh)-\bu(\bx-\bh)}{|\bh|^{d+2s}}}{\bh}  \\
	+ \frac{(2 \mu + \lambda)^s - \mu^s}{2s} \frac{\kappa_{d,s}}{2} \intdm{\bbR^d}{ \left( \frac{\bh \otimes\bh}{|\bh|^2} \right) \frac{2\bu(\bx) - \bu(\bx+\bh)-\bu(\bx-\bh)}{|\bh|^{d+2s}}}{\bh}\,.
\end{multline*}
$\bbL^s \bu$ is in general not in $\scS(\bbR^d;\bbR^d)$ due to the singularity of the derivatives of $|\bsxi|^{2s}$ at the origin, but remains in $C^{\infty}(\bbR^d;\bbR^d)$ since $\bM^s(\bsxi) \wh{\bu}$ is rapidly decreasing.

\section{The Fractional Lam\'e-Navier Operator in Applications}\label{sec:Applications}

\subsection{Fractional Vector Calculus}\label{subsec:FracVectorCalc}

It is obvious by definition that the operator $\bbL^s$ commutes with translations and rotations, and is homogeneous of degree $2s$.
Thus the natural invariances desired for nonlocal vector analysis are satisfied; see \cite{vsilhavy2020fractional}.
In fact, we can express $\bbL^s$ in terms of the nonlocal gradient operators $\grad^s$ and $\div^s$, defined as 
\begin{equation}\label{eq:GradDef}
	\grad^s \bu(\bx) := k_{d,s} \intdm{\bbR^d}{ \frac{\bu(\by)-\bu(\bx)}{|\by-\bx|^{d+s}} \otimes \frac{\by-\bx}{|\by-\bx|} }{\by} 
\end{equation}
and
\begin{equation}\label{eq:DivDef}
	\div^s \bu(\bx) := k_{d,s} \intdm{\bbR^d}{ \frac{\bu(\by)-\bu(\bx)}{|\by-\bx|^{d+s}} \cdot \frac{\by-\bx}{|\by-\bx|} }{\by}
\end{equation}
where the constant $k_{d,s}$ is given by
\begin{equation*}
	k_{d,s} = \frac{2^s \Gamma(\frac{d+s+1}{2})}{\pi^{d/2} \Gamma(\frac{1-s}{2})} = \left( \intdm{\bbR^d}{ \frac{1}{|\bh|^{d+s-1}} \frac{h_1 \sin(h_1)}{|\bh|^2} }{\bh} \right)^{-1}\,.
\end{equation*}
The constant $k_{d,s}$ is defined so that the Fourier transform is normalized:
\begin{equation}\label{eq:FourierTransformOfGradients}
	\begin{split}
		\cF( \grad^s \bu)(\bsxi) &= (2 \pi |\bsxi|)^{s-1} \big( \wh{\bu}(\bsxi) \otimes (2 \pi \imath \bsxi) \big) \,, \\
		\cF( \div^s \bu)(\bsxi) &= (2 \pi |\bsxi|)^{s-1}  \big( \wh{\bu}(\bsxi) \cdot  (2 \pi \imath \bsxi) \big)\,. \\
	\end{split}
\end{equation}
See \cite{shieh2015new, shieh2018new} for these identities and other properties. It is straightforward to obtain a proof of \Cref{thm:Intro:NonlocalCalculus} using these identities:

\begin{theorem}
	For any $\bu \in \scS(\bbR^d;\bbR^d)$,
	\begin{equation}\label{eq:NonlocalGradientCharacterization}
		\bbL^s \bu(\bx) = - \mu^s (-\Delta)^s \bu(\bx) - \big( (2\mu+\lambda)^s - \mu^s \big) \grad^s \div^s \bu(\bx)\,.
	\end{equation}
\end{theorem}

\begin{proof}
	We use the Fourier transform.
	By the formulae \eqref{eq:FourierTransformOfGradients} and by the identity $\div^s \circ \grad^s =(-\Delta)^s$ (see \cite{vsilhavy2020fractional, d2020towards})
	\begin{equation*}
		\begin{split}
			\cF( \div^s \circ \grad^s \bu)(\bsxi) &= -(2 \pi |\bsxi|)^{2s} \wh{\bu}(\bsxi) \,, \\
			\cF( \grad^s \circ \div^s \bu)(\bsxi) &= (2 \pi |\bsxi|)^{2s-2}  \big( \wh{\bu}(\bsxi) \cdot  (2 \pi \imath \bsxi) \big) (2 \pi \imath \bsxi) = - (2 \pi |\bsxi|)^{2s} \frac{\bsxi \otimes \bsxi}{|\bsxi|^2} \wh{\bu}(\bsxi)\,, \\
		\end{split}
	\end{equation*}
	and the equality follows from \eqref{eq:FractionaLameFourierMatrix}.
\end{proof}

We remark here that we can write $\bbL^s$ in a kind of divergence form in terms of the operators $\grad^s$ and $\div^s$. Define the \textit{fractional stress tensor} $\bssigma^s(\bu)$ by
\begin{equation*}
	\bssigma^s(\bu) := 2 \mu^s \grad^s_{sym}\bu + \big( (2 \mu+\lambda)^s - 2\mu^s \big) (\mathrm{tr} \grad^s_{sym}\bu) \bI\,,
\end{equation*}
where $\grad_{sym}^s \bu := \frac{1}{2} (\grad^s \bu + (\grad^s \bu)^T )$. Then for any $\bu \in \scS(\bbR^d;\bbR^d)$
\begin{equation*}
	\bbL^s \bu = -\div^s \bssigma^s(\bu)\,.
\end{equation*}

\subsection{A State-Based Peridynamic Operator}

The formulation for state-based peridynamic operators that we use here is based on those in \cite{Silling2010, MengeshaDuElasticity, alali2014generalized, alali2015peridynamics}. For homogeneous isotropic materials the operator takes the form
\begin{equation*}
	\begin{split}
		&\cL_{sb} \bu(\bx) \\
		&:= C_1(\mu,\lambda) \intdm{\bbR^d}{ \rho_1(|\bx-\by|) \frac{(\bx-\by) \otimes (\bx-\by)}{|\bx-\by|^2} \big( \bu(\bx)-\bu(\by) \big) }{\by} \\
		&\quad + C_2(\mu,\lambda) \iintdm{\bbR^d}{\bbR^d}{ \rho_2(|\bx-\by|) \rho_2(|\by-\bz|) \\
			&\qquad \times \left[ \left( \frac{\bx-\by}{|\bx-\by|} \otimes \frac{\by-\bz}{|\by-\bz|} \right) \big( \bu(\by) - \bu(\bz) \big) 
			- \left( \frac{\bx-\by}{|\bx-\by|} \otimes \frac{\bx-\bz}{|\bx-\bz|} \right) \big( \bu(\bx) - \bu(\bz) \big)  \right] }{\bz}{\by}
	\end{split}
\end{equation*}
for material constants $C_1(\mu,\lambda)$ and $C_2(\mu,\lambda)$ and for kernels $\rho_1$ and $\rho_2$. If we allow $\rho_1(|\bx|)$ and $\rho_2(|\bx|)$ to have a singularity at the origin, then we can cast the fractional Lam\'e-Navier operator as a special class of state-based peridynamic operators.

\begin{theorem}\label{thm:FractionalLame:StateBasedPeri}
	Let $C_1(\mu,\lambda)=\mu^s$ and $C_2(\mu,\lambda) = (2 \mu + \lambda)^s - (2s+1) \mu^s$, and set $\rho_1(|\bseta|) = \kappa_{d,s} |\bseta|^{-d-2s}$ and  $\rho_2(|\bseta|) = k_{d,s} |\bseta|^{-d-s}$. Then for all $\bu \in \scS(\bbR^d;\bbR^d)$
	\begin{equation*}
		\cL_{sb} \bu = \bbL^s \bu\,.
	\end{equation*}
\end{theorem}

\begin{proof}
	Using the definitions \eqref{eq:Localization:DefnOfL}, \eqref{eq:GradDef} and \eqref{eq:DivDef},
	\begin{equation*}
		\begin{split}
			\cL_{sb} \bu(\bx)
			&= \mu^s \bbF^s \bu(\bx) \\
				&\qquad + C_2(\mu,\lambda) \int_{\bbR^d} \frac{k_{d,s}}{|\bx-\by|^{d+s}} \Bigg[ \left( \int_{\bbR^d}{  \frac{k_{d,s}}{|\by-\bz|^{d+s}}    \big( \bu(\by) - \bu(\bz) \big) \cdot \frac{\by-\bz}{|\by-\bz|} } \, \rmd \bz \right) \\
				&\qquad - \left( \int_{\bbR^d}{  \frac{k_{d,s}}{|\bx-\bz|^{d+s}}    \big( \bu(\bx) - \bu(\bz) \big) \cdot \frac{\bx-\bz}{|\bx-\bz|} } \, \rmd \bz \right) \Bigg] \frac{\bx-\by}{|\bx-\by|} \, \rmd \by \\
			&= \mu^s \bbF^s \bu(\bx) \\
				&\qquad + C_2(\mu,\lambda) \int_{\bbR^d} \frac{k_{d,s}}{|\bx-\by|^{d+s}} \big( \div^s \bu(\by) - \div^s \bu(\bx) \big) \frac{\bx-\by}{|\bx-\by|} \, \rmd \by \\
			&= \mu^s \bbF^s \bu(\bx) 
				- \big( (2 \mu + \lambda)^s - (2s+1) \mu^s \big) \grad^s \div^s \bu(\bx)\,.
		\end{split}
	\end{equation*}
	The result then follows from the Fourier transform formulae \eqref{eq:FourierTransformOfL} and \eqref{eq:FourierTransformOfGradients}.
\end{proof}

Note that if we let $s=1$ in $C_1$ and $C_2$ in the above theorem, then we obtain the same material constants that appear in \cite[Equation 2.2]{alali2015peridynamics}. However, the model in that work is valid for a much more general class of kernels $\rho_1$ and $\rho_2$. So we can think of the fractional Lam\'e-Navier operator as a tradeoff; we use the clean-cut theory for fractional powers of operators at the expense of flexibility in the modeling parameters.
We also note here that this theorem provides a rigorous justification of the calculations made in \cite{vsilhavy2017higher} in the special case of fractional kernels.

\subsection{The Dirichlet-to-Neumann Map Associated to the Lam\'e-Navier System}

In this section we find an explicit form of the Dirichlet-to-Neumann map associated to the classical Lam\'e-Navier system of equations in the upper half-space.
Consider the system of equations
\begin{equation}\label{eq:LameInHalfSpace}
	\begin{cases}
		-\wt{\mu} \Delta \bv(\bx,x_{d+1}) - (\wt{\mu} + \wt{\lambda}) \grad ( \div \bv(\bx,x_{d+1})) = 0\,, & \quad (\bx,x_{d+1}) \in \bbR^{d+1}_+ = \bbR^d \times (0,\infty)\,, \\
		\bv(\bx,0) = \bg(\bx)\,, \quad \bx \in \bbR^d\,.
	\end{cases}
\end{equation}
where $\wt{\mu} > 0$, $2 \wt{\mu} + \wt{\lambda} > 0$, $\bv = (\bu,v_{d+1}) : \bbR^{d+1}_+ \to \bbR^{d+1}$ and $\bg = (\bff,g_{d+1}) : \bbR^d \to \bbR^{d+1}$. The partial differential operators are taken here with respect to the $(\bx,x_{d+1})$ variable, i.e. $\Delta = \sum_{j = 1}^{d+1} \p_{j j} $, $\grad = (\p_1, \ldots, \p_d, \p_{d+1})$, and $\div = \sum_{j=1}^{d+1} \p_j$.
The natural Neumann boundary condition for the associated variational problem is given by prescribing the values of $(\wt{\mu} \grad \bv + (\wt{\mu} + \wt{\lambda} ) \div \bv \bI) \bsnu$ on the hyperplane $\{ t = 0 \}$ where $\bsnu = -\be_{d+1}$ is the outward unit normal to $\bbR^{d+1}_+$. The Dirichlet-to-Neumann map is therefore defined by
\begin{equation*}
	\frak{D}_{\wt{\mu},\wt{\lambda}} \bg(\bx) := -\wt{\mu} \p_{d+1} \bv(\bx,0) - (\wt{\mu}+\wt{\lambda}) \div \bv(\bx,0) \be_{d+1}\,.
\end{equation*}
By the divergence theorem it is clear that $\frak{D}_{\wt{\mu},\wt{\lambda}}$ is a map from  $\scL^{1/2,2}(\bbR^d;\bbR^{d+1})$ to its dual $\scL^{-1/2,2}(\bbR^d;\bbR^{d+1})$, but we state the following theorem for Schwartz functions in order to emphasize the explicit calculations.

\begin{theorem}
	Let $\bff \in \scS(\bbR^d;\bbR^d)$, and let $g_{d+1} \in \scS(\bbR^d)$. Suppose $\bv$ is the unique solution of \eqref{eq:LameInHalfSpace} (see \cite{Martell-HalfSpace} for details). 
	Then the first $d$ components of the Dirichlet-to-Neumann map, denoted as $\bsLambda_{\wt{\mu},\wt{\lambda}} \bg$, have the expression
	\begin{multline}\label{eq:PeridynamicsAsDtNMap}
		\bsLambda_{\wt{\mu},\wt{\lambda}}\bg(\bx) = \frac{\wt{\mu}(\wt{\mu}+\wt{\lambda})}{3 \wt{\mu} + \wt{\lambda}} \grad g_{d+1}(\bx) + \frac{2 \wt{\mu}^2}{3 \wt{\mu} + \wt{\lambda}} \frac{2}{\omega_d} \intdm{\bbR^d}{\frac{\bff(\bx)-\bff(\by)}{|\by-\bx|^{d+1}}}{\by} \\
		+ \frac{\wt{\mu} (\wt{\mu} + \wt{\lambda})}{3 \wt{\mu} + \wt{\lambda}} \frac{2(d+1)}{\omega_d} \intdm{\bbR^d}{\frac{(\bx-\by) \otimes (\bx-\by)}{|\bx-\by|^2} \frac{\bff(\bx)-\bff(\by)}{|\by-\bx|^{d+1}}}{\by}\,,
	\end{multline}
	where $\omega_d := \frac{2 \pi^{\frac{d+1}{2}}}{\Gamma(\frac{d+1}{2})}$ is the surface measure of the unit sphere $\bbS^d \subset \bbR^{d+1}$.
	The last component $(\frak{D}_{\wt{\mu},\wt{\lambda}} \bg)_{d+1}$
	has the expression
	\begin{equation}\label{eq:PeridynamicsAsDtNMap2}
		(\frak{D}_{\wt{\mu},\wt{\lambda}} \bg)_{d+1}(\bx) = \frac{ 2 \wt{\mu} (2 \wt{\mu}+\wt{\lambda}) }{3 \wt{\mu}+\wt{\lambda} } (-\Delta)^{\frac{1}{2}} g_{d+1}(\bx) - \frac{ \wt{\mu} (\wt{\mu}+\wt{\lambda}) }{3 \wt{\mu}+\wt{\lambda} } \div_{\bx} \bff(\bx)\,.
	\end{equation}
\end{theorem}

\begin{proof}
	Denote the last variable $x_{d+1} = t$, so that $\bv(\bx,x_{d+1}) = \bv(\bx,t)$ and $\partial_{d+1} = \partial_t$. 
	We need to find 
	\begin{equation*}
		\bsLambda_{\wt{\mu},\wt{\lambda}} \bg(\bx) = - \wt{\mu} \p_{t} \bu(\bx,0)
	\end{equation*}
	and
	\begin{equation*}
		(\frak{D}_{\wt{\mu},\wt{\lambda}} \bg)_{d+1}(\bx)  = - (2\wt{\mu} + \wt{\lambda}) \p_{t} v_{d+1}(\bx,0) - (\wt{\mu} + \wt{\lambda}) \div_{\bx} \bu(\bx,0)\,.
	\end{equation*}
	
	We will use the Poisson kernel $\bP_{\wt{\mu},\wt{\lambda}} : \bbR^d \times (0,\infty) \to \bbR^{(d+1) \times (d+1)}$ associated to the Lam\'e-Navier equation; see \cite{Martell-HalfSpace} for definitions and properties. The unique solution to \eqref{eq:LameInHalfSpace} is the Poisson integral
	\begin{equation*}
		\bv(\bx,t) = \intdm{\bbR^d}{\bP_{\wt{\mu},\wt{\lambda}}(\bx-\by,t) \bg(\by)}{\by}\,,
	\end{equation*}
	that is,
	\begin{multline}
		\bv(\bx,t) = \frac{2 \wt{\mu}}{3 \wt{\mu} + \wt{\lambda}} \frac{2}{\omega_d} \intdm{\bbR^d}{\frac{t}{\big( |\bx-\by|^2 + t^2 \big)^{\frac{d+1}{2}}} \bg(\by) }{\by} \\
		+ \frac{\wt{\mu} + \wt{\lambda}}{3 \wt{\mu} + \wt{\lambda}} \frac{2(d+1)}{\omega_d} \intdm{\bbR^d}{ \begin{bmatrix} (\bx-\by) \otimes (\bx-\by) & t (\bx-\by) \\ t (\bx-\by) & t^2 \end{bmatrix} \frac{t}{\big( |\bx-\by|^2 + t^2 \big)^{\frac{d+3}{2}}} \bg(\by)}{\by}\,.
	\end{multline}
	Computing the Fourier transform in $\bx$ (see \cite{MengeshaScott2018Korn} for details) gives
	\begin{equation*}
		\begin{split}
			\wh{\bv}(\bsxi,t) &= \wh{\bP}_{\wt{\mu},\wt{\lambda}}(\bsxi,t) \wh{\bg}(\bsxi) \\
			&= \rme^{-2 \pi |\bsxi| t} \left( \bI_{d+1} - \frac{\wt{\mu}+\wt{\lambda}}{3 \wt{\mu} + \wt{\lambda}} (2 \pi |\bsxi| t)  \begin{bmatrix} \frac{\bsxi \otimes \bsxi}{|\bsxi|^2} & \imath \frac{\bsxi}{|\bsxi|} \\ \imath \frac{\bsxi}{|\bsxi|} & -1 \end{bmatrix} \right) \wh{\bg}(\bsxi)\,.
		\end{split}
	\end{equation*}
	where $\bI_{d+1}$ denotes the $d \times d$ identity matrix. From here we compute
	\begin{equation}\label{eq:DtNMap:Pf1}
		\begin{split}
			\cF (\bsLambda_{\wt{\mu},\wt{\lambda}} \bg)(\bsxi) &= - \wt{\mu} \p_{t} \wt{\bu}(\bsxi,0) \\
			&= \wt{\mu} \, 2 \pi |\bsxi| \left( \bI + \frac{\wt{\mu}+\wt{\lambda}}{3 \wt{\mu} + \wt{\lambda}} \frac{\bsxi \otimes \bsxi}{|\bsxi|^2}  \right) \wh{\bff}(\bsxi) +  \wt{\mu} \frac{\wt{\mu}+\wt{\lambda}}{3 \wt{\mu} + \wt{\lambda}} 2 \pi \imath \bsxi \; \wh{g_{d+1}}(\bsxi)
		\end{split}
	\end{equation}
	and
	\begin{equation}\label{eq:DtNMap:Pf2}
		\begin{split}
			\cF ((\frak{D}_{\wt{\mu},\wt{\lambda}} \bg)_{d+1})(\bsxi)  &= - (2\wt{\mu} + \wt{\lambda}) \p_{t} \wh{v}_{d+1}(\bsxi,0) - (\wt{\mu} + \wt{\lambda}) (2 \pi \imath \bsxi) \cdot \wh{\bu}(\bsxi,0) \\
			&= \frac{ 2 \wt{\mu} (2 \wt{\mu}+\wt{\lambda}) }{3 \wt{\mu}+\wt{\lambda} } 2 \pi |\bsxi|  \wh{g}_{d+1}(\bsxi) - \frac{ \wt{\mu} (\wt{\mu}+\wt{\lambda}) }{3 \wt{\mu}+\wt{\lambda} } 2 \pi \imath \bsxi \cdot \wh{\bff}(\bsxi)\,.
		\end{split}
	\end{equation}
%
%
%
	Note that when $s=\frac{1}{2}$ the constant $c_{d,s} = \frac{\Gamma(\frac{d+1}{2})}{\pi^{\frac{d+1}{2}}} = \frac{2}{\omega_d}$, and so $\kappa_{d,s} = \frac{2(d+1)}{\omega_d}$. Therefore
	\eqref{eq:PeridynamicsAsDtNMap} follows from \eqref{eq:DtNMap:Pf1} with the help of the formula \eqref{eq:FourierTransformOfL} with $s = \frac{1}{2}$. 
	The formula \eqref{eq:PeridynamicsAsDtNMap2} is clear from \eqref{eq:DtNMap:Pf2} and \eqref{eq:FourierSymbol:FractionalLaplacian}.
\end{proof}

When $g_{d+1} = 0$ the operator $\bsLambda_{\wt{\mu},\wt{\lambda}}$ will coincide with the half-power of the Lam\'e-Navier operator $\bbL^{\frac{1}{2}}$ when the elastic constants in both formulae satisfy a certain relation.
We record this in the following corollary:

\begin{corollary}
	Let $\mu$, $\lambda$ denote the elastic constants associated to the operator $\bbL$ in \eqref{eq:LameOperator}, and let $\wt{\mu}$, $\wt{\lambda}$ be as in \eqref{eq:LameInHalfSpace} with $\bsLambda_{\wt{\mu},\wt{\lambda}}$ as in \eqref{eq:PeridynamicsAsDtNMap}.
	Then for any function $\bu \in \scS(\bbR^d;\bbR^d)$ and defining $\bU = (\bu,0)$
	\begin{equation*}
		\bsLambda_{\wt{\mu},\wt{\lambda}} \bU(\bx) = \bbL^{\frac{1}{2}} \bu(\bx)
	\end{equation*}
	if and only if
	\begin{equation*}
		\mu = \wt{\mu}^2 \quad \text{ and } \quad 2 \mu + \lambda =  \left( \frac{2 \wt{\mu} (2 \wt{\mu} + \wt{\lambda})}{3 \wt{\mu} + \wt{\lambda}} \right)^2 \,.
	\end{equation*}
\end{corollary}

\section{Properties of the Fractional Lam\'e-Navier Operator}\label{sec:Properties}

\subsection{The Fundamental Solution}\label{subsec:FundSoln}

The fractional Laplacian has a fundamental solution that is well-studied in its own right, see \cite{Stein}. Formally, the fundamental solution $\Phi^s$ is given via the Fourier identity
\begin{equation}\label{eq:RieszKernel:Fourier}
	\wh{\Phi^s}(\bsxi) = (2 \pi |\bsxi|)^{-2s} \quad \text{ in } \scS'(\bbR^d)\,,
\end{equation}
i.e.
\begin{equation*}
	\int_{\bbR^d} (2 \pi |\bsxi|)^{-2s} v(\bsxi) \, \rmd \bsxi = \int_{\bbR^d}  \Phi^s(\bx) \wh{v}(\bx) \, \rmd \bx \qquad \text{ for all } v \in \scS(\bbR^d)\,. 
\end{equation*}
In fact, $\Phi^s$ is given by the Riesz kernel
\begin{equation*}
	\Phi^s(\bx) = \frac{g_{d,s}}{|\bx|^{d-2s}}\,, \quad \text{ where } \quad g_{d,s} := \frac{\Gamma \left( \frac{d}{2} -s \right) }{\pi^{d/2} 2^{2s} \Gamma(s)}\,.
\end{equation*}
Therefore, solutions to the Poisson equation associated to $(-\Delta)^s$ are formally given by the \textit{Riesz Potential} $I^s$ defined as
\begin{equation*}
	I^s u(\bx) = \Phi^s \ast u(\bx)\,,
\end{equation*}
$I^s u(\bx)$ is a well-defined functional on $\scS(\bbR^d)$ for all $s \in (0,\frac{d}{2})$, and it satisfies the Fourier identity
\begin{equation}\label{eq:RieszTransform:Fourier}
	\cF (I^s u)(\bsxi) = (2 \pi |\bsxi|)^{-2s} \wh{u}(\bsxi) \quad \text{ in } \scS'(\bbR^d)
\end{equation}
for all $u \in \scS(\bbR^d)$.

We formally define negative fractional powers of the Lam\'e-Navier operator in the same way, beginning with Fourier inversion:
\begin{equation*}
	\cF(\bbL^{-s} \bu )(\bsxi) := \big[ \bM^{s}(\bsxi) \big]^{-1} \wh{\bu}(\bx)\,.
\end{equation*}
Henceforth use the notation $\bbL^{-s} = \bbE^s$. In the range $s \in (0,\frac{d}{2})$,  $\bbE^s$ has an explicit expression as an integral operator;
\begin{equation}\label{eq:Definition:LameNavierPotential}
	\bbE^s \bu(\bx) := \bsPsi^s \ast \bu(\bx)\,,
\end{equation}
where the matrix field $\bsPsi^s : \bbR^d \to \bbR^{d \times d}$ is defined as
\begin{equation}\label{eq:Definition:LameNavierKernel}
	\begin{split}
		\bsPsi^s(\bx) &:= \frac{\gamma_{d,s}}{\mu^s (2 \mu + \lambda)^s} \left[ \frac{(2s-1) (2 \mu + \lambda)^s + \mu^s }{d-2s} \frac{1}{|\bx|^{d-2s}} \bI \right.\\ 
		&\left. \qquad \qquad + \big( (2 \mu +\lambda)^s - \mu^s \big) \frac{1}{|\bx|^{d-2s}} \frac{\bx \otimes \bx}{|\bx|^2}  \right]
	\end{split}
\end{equation}
and the constant $\gamma_{d,s}$ is defined as
\begin{equation}\label{eq:Definition:LameNavierPotentialConstant}
	\gamma_{d,s} := \frac{\Gamma( \frac{d+2-2s}{2} ) }{2^{2s} \pi^{d/2} \Gamma(1+s) }\,.
\end{equation}
This formal identification is made rigorous in Theorem \ref{thm:NegativePowers:FourierTransform} and Proposition \ref{prop:NegativePowers:LameNavierPotential}.

\begin{theorem}\label{thm:NegativePowers:FourierTransform}
	Let $d \geq 2$ and $s \in (0,\frac{d}{2})$.
	Then $\bsPsi^s \in \scS'(\bbR^d;\bbR^d)$, and 
	\begin{equation*}
		\cF(\bsPsi^s)(\bsxi) = \big[ \bM^s(\bsxi) \big]^{-1} \quad \text{ in } \scS'(\bbR^d;\bbR^d)\,.
	\end{equation*}
\end{theorem}

\begin{proposition}\label{prop:NegativePowers:LameNavierPotential}
	The operator $\bbE^s$ is a well-defined and continuous operator on $\scS(\bbR^d;\bbR^d)$. For any $\bu \in \scS(\bbR^d;\bbR^d)$ we have
	\begin{equation*}
		\cF(\bbE^s \bu)(\bsxi) = \big[ \bM^s(\bsxi) \big]^{-1} \wh{\bu}(\bsxi) \quad \text{ in } \scS'(\bbR^d;\bbR^d)\,.
	\end{equation*}
\end{proposition}

We call $\bbE^s$ the \textit{Lam\'e-Navier-Riesz potential} of order $s$. For $s \in (0,1)$ the Lam\'e-Navier-Riesz potential is a type of ``fundamental solution'' for the fractional Lam\'e-Navier equations; this will be discussed further in \Cref{sec:HolderSpaces}.

\begin{proof}[proof of Theorem \ref{thm:NegativePowers:FourierTransform}]
	It is easy to see that $\bsPsi^s \in \scS'(\bbR^d;\bbR^d)$. Next, note that for $a$, $b \in \bbR$ the matrix $a \bI + b \frac{\bsxi \otimes \bsxi}{|\bsxi|^2}$ is invertible so long as $a \neq 0$ and $a \neq - b$, with inverse given by $a^{-1} (\bI - \frac{b}{a+b} \frac{\bsxi \otimes \bsxi}{|\bsxi|^2} )$. Thus we have
	\begin{equation}\label{eq:InverseOfFractionalFourierMatrix}
		\big[ \bM^s(\bsxi) \big]^{-1} = \frac{1}{\mu^s} \frac{1}{(2 \pi |\bsxi|)^{2s}} \bI + \left( \frac{1}{(2 \mu + \lambda)^s} - \frac{1}{\mu^s} \right) \frac{1}{(2 \pi |\bsxi|)^{2s}} \frac{\bsxi \otimes \bsxi}{|\bsxi|^2}
	\end{equation}
	for any $s > 0$.
	Therefore $\big[ \bM^s(\bsxi) \big]^{-1} \in \scS'(\bbR^d;\bbR^{d \times d})$, and so we can compute its inverse Fourier transform. Using \eqref{eq:RieszKernel:Fourier} and \eqref{eq:NegativePowers:FourierTransformPiece1}
	\begin{equation*}
	\begin{split}
		\cF^{-1} &\left( \big[ \bM^s(\cdot) \big]^{-1} \right)(\bx) \\
			&= \frac{1}{\mu^s} \cF^{-1} \left[ \frac{1}{(2 \pi |\bsxi|)^{2s}} \right] \bI + \left( \frac{1}{(2 \mu + \lambda)^s} - \frac{1}{\mu^s} \right) \cF^{-1} \left[ \frac{1}{(2 \pi |\bsxi|)^{2s}} \frac{\bsxi \otimes \bsxi}{|\bsxi|^2} \right] \\
			&= \frac{g_{d,s}}{\mu^s} \frac{1}{|\bx|^{d-2s}} \bI + \left( \frac{1}{(2 \mu + \lambda)^s} - \frac{1}{\mu^s} \right) \frac{\gamma_{d,s}}{|\bx|^{d-2s}} \left( \frac{1}{d-2s} \bI -  \frac{\bx \otimes \bx}{|\bx|^2} \right)\,.
	\end{split}
	\end{equation*}
	Using the identity $a \Gamma(a) = \Gamma(a+1)$ for $a > 0$, we see that $\gamma_{d,s} = \frac{d-2s}{2s} g_{d,s}$.  Therefore, 
	\begin{equation*}
	\begin{split}
		\cF^{-1} &\left( \big[ \bM^s(\cdot) \big]^{-1} \right)(\bx) \\
		&= \frac{2s}{\mu^s} \frac{\gamma_{d,s}}{d-2s} \frac{1}{|\bx|^{d-2s}} \bI + \left( \frac{1}{(2 \mu + \lambda)^s} - \frac{1}{\mu^s} \right) \frac{\gamma_{d,s}}{|\bx|^{d-2s}} \left( \frac{1}{d-2s} \bI -  \frac{\bx \otimes \bx}{|\bx|^2} \right)\,.
	\end{split}
	\end{equation*}
	We obtain the expression on the right-hand side of \eqref{eq:Definition:LameNavierKernel} after straightforward algebraic manipulations.
\end{proof}

\subsection{Functional Calculus and Distributional Forms}

The following theorem is now apparent using the Fourier transform formulae:

\begin{theorem}\label{thm:CalcProperties}
	Let $\bu$ $\bv \in \scS(\bbR^d;\bbR^d)$. Let $d \geq 2$, and let $s$, $t \in (-\frac{d}{2},1]$. Then 
	$$
	\bbL^1 \bu = \lim\limits_{s \to 1^-} \bbL^s \bu = \bbL \bu
	$$ 
	and
	$$
	\bbL^0 \bu  = \lim\limits_{s \to 0^+} \bbL^s \bu = \bu\,.
	$$
	For $d \geq 3$, $\bbL^{-1}\bu = \bbE \bu$ is the fundamental solution associated to $\bbL$ (See \cite[Section 10.3]{mitrea}).
	Moreover,
	\begin{equation*}
		(\bbL^s \circ \bbL^t) \bu = \bbL^{s+t} \bu
	\end{equation*}
	whenever $s+t \in (-\frac{d}{2},1]$, and
	\begin{equation*}
		\intdm{\bbR^d}{ \Vint{\bbL^s \bu(\bx), \bv(\bx)} }{\bx} = \intdm{\bbR^d}{ \Vint{\bu(\bx), \bbL^s  \bv(\bx)} }{\bx}\,.
	\end{equation*}
\end{theorem}

Although $\bbL^s \bu$ is not a Schwartz function for $\bu \in \scS$, it still satisfies a decay property. This decay, proved in the theorem below, is a straightforward adaptation of \cite[Proposition 2.9]{garofalo2017fractional}.

\begin{theorem}
	Let $\bu \in \scS(\bbR^d;\bbR^d)$. For every $|\bx| \geq 1$, we have
	\begin{equation}\label{eq:DecayRate}
		|\bbL^s \bu(\bx)| \leq C|\bx|^{-d-2s}\,,
	\end{equation}
	where $C = C(\mu,\lambda,d,s,\bu)$.
\end{theorem}

In the subsequent sections, we will demonstrate how far calculations involving $\bbL^s$ can be extended. The operator is so far defined only for very smooth functions, but thanks to the decay rate \eqref{eq:DecayRate} we can extend the definition by duality following the strategy in \cite{silvestre2007regularity, bucur2016some, garofalo2017fractional}.
Let $\scS_s(\bbR^d;\bbR^d)$ be the locally convex topological space defined by
\begin{equation*}
	\scS_s(\bbR^d;\bbR^d) := \left\{ \bu \in C^{\infty}(\bbR^d;\bbR^d) \, : \, \sup_{\bx \in \bbR^d} (1+|\bx|^{d+2s}) |D^{\bsgamma} \bu(\bx)| < \infty\,, \quad \bsgamma \in \bbN^{d}_0  \right\}
\end{equation*}
equipped with the family of seminorms
\begin{equation*}
	[\bu]_{\scS_{s,\bsgamma}(\bbR^d)} := \sup_{\bx \in \bbR^d} (1+|\bx|^{d+2s}) |D^{\bsgamma} \bu(\bx)|\,, \quad \bsgamma \in \bbN^{d}_0\,.
\end{equation*}
Let $\scS_s'(\bbR^d;\bbR^d)$ denote the topological dual of $\scS_s(\bbR^d;\bbR^d)$.
It is straightforward to check using \eqref{eq:DecayRate} that $\bbL^s \bu \in \scS_s(\bbR^d;\bbR^d)$ whenever $\bu \in \scS(\bbR^d;\bbR^d)$.
We can then define $\bbL^s$ on the space $\scS_s'(\bbR^d;\bbR^d)$ by duality since $\bbL^s$ is of a symmetric form; for $\bu \in \scS_s'(\bbR^d;\bbR^d)$ and $\bv \in \scS_s(\bbR^d;\bbR^d)$
\begin{equation*}
	\Vint{\bbL^s \bu, \bv} := \Vint{\bu, \bbL^s \bv}\,.
\end{equation*}

The spaces described above are very specialized. In what follows we will be able to apply $\bbL^s$ to functions belonging to a subset of a weighted Lebesgue space. Define the space $L^1_{s}(\bbR^d;\bbR^d)$ by
\begin{equation}\label{eq:DefnOfWeightedLebSpace}
	L^1_{s}(\bbR^d;\bbR^d) := \left\{ \bu \in L^1_{loc}(\bbR^d;\bbR^d) \, : \,  \Vnorm{\bu}_{L^1_{s}(\bbR^d)} := \intdm{\bbR^d}{ \frac{|\bu(\bx)|}{1+|\bx|^{d+2s}} }{\bx} < \infty  \right\}\,.
\end{equation}
We note that $L^1_{s}(\bbR^d;\bbR^d) \subset \scS_s'(\bbR^d;\bbR^d)$, since for $\bu \in L^1_{s}(\bbR^d;\bbR^d)$ and $\bsvarphi \in \scS_s(\bbR^d;\bbR^d)$
\begin{equation*}
	\intdm{\bbR^d}{\Vint{\bu(\bx), \bsvarphi(\bx)}}{\bx} = \intdm{\bbR^d}{\Vint{\frac{\bu(\bx)}{1+|\bx|^{d+2s}}, (1+|\bx|^{d+2s})\bsvarphi(\bx)}}{\bx} \leq \Vnorm{\bu}_{L^1_{s}(\bbR^d)} [\bsvarphi]_{\scS_{s,0}(\bbR^d)}\,.
\end{equation*}
Additionally, $L^p(\bbR^d) \subset L^1_{s}(\bbR^d;\bbR^d)$ for $p \in [1,\infty]$.

\section{$\bbL^s$ and H\"older Spaces}\label{sec:HolderSpaces}

The next theorems concerning the mapping properties of $\bbL^s$ are analogues of results from \cite{silvestre2007regularity} for the fractional Laplacian. Just as in those works, the proofs here rely on estimates of the full difference $\bu(\bx)-\bu(\by)$, and so the coupled nature of the operator does not play an essential role. Therefore if the proof for $\bbL^s$ is very similar to the analogous proof for $(-\Delta)^s$ then we state it without proof.

\begin{theorem}[See {\cite[Proposition 2.4]{silvestre2007regularity}} and \cite{garofalo2017fractional}]\label{prop:LuIsContinuousForHolderFxns}
	Let $\Omega \subset \bbR^d$ be a domain. Suppose that $\bu \in L^1_{s}(\bbR^d;\bbR^d)$, and that for some $\veps > 0$
	\begin{equation*}
		\bu \in \begin{cases}
			C^{0,2s+\veps}(\Omega;\bbR^d) &\text{ when } s \in (0,1/2)\,, \\
			C^{1,2s+\veps-1}(\Omega;\bbR^d) &\text{ when } s \in [1/2,1)\,.
		\end{cases}
	\end{equation*}
	Then $\bbL^s \bu$ is a continuous function on $\Omega$, with values given by \eqref{eq:FractionalPowerOfLame}.
\end{theorem}

\begin{theorem}\label{prop:LuCommutesWithDerivatives}
	Let $\Omega \subset \bbR^d$ be a bounded domain. Let $k \in \bbN$. Suppose that $\bu \in L^1_{s}(\bbR^d;\bbR^d)$ and 
	\begin{equation*}
		\bu \in
		\begin{cases}
			C^{k,2s+\veps}(\Omega;\bbR^d)\,, &\qquad s \in (0,1/2)\,, \\
			C^{k+1,2s+\veps-1}(\Omega;\bbR^d)\,, &\qquad s \in [1/2,1)\,, \\
		\end{cases}
	\end{equation*}
	for some $\veps > 0$. Then for any $s \in (0,1)$, $\bbL^s \bu$ is in $C^{k}(\Omega;\bbR^d)$, with values given by \eqref{eq:FractionalPowerOfLame}.
\end{theorem}

\begin{proof}
	The result follows (by induction) in exactly the same way as that of \Cref{prop:LuIsContinuousForHolderFxns} after noting that $\bbL^s$ commutes with derivatives $\p_{x_j}$ for any $j$.
\end{proof}

\begin{theorem}[See {\cite[Proposition 2.5]{silvestre2007regularity}}]\label{prop:HolderSpace0}
	Suppose $\bu \in C^{0,\alpha}(\bbR^d;\bbR^d)$ for $\alpha \in (2s,1]$. Then $\bbL^s \bu \in C^{0,\alpha-2s}(\bbR^d;\bbR^d)$, with the estimate
	\begin{equation}\label{eq:HolderEstimate1}
		[\bbL^s \bu]_{C^{0,\alpha-2s}(\bbR^d)} \leq C [\bu]_{C^{0,\alpha}(\bbR^d)}\,,
	\end{equation}
	where the constant $C$ depends only on $d$, $s$ and $\alpha$.
\end{theorem}

\begin{theorem}[See {\cite[Proposition 2.6]{silvestre2007regularity}}]
	Suppose $\bu \in C^{1,\alpha}(\bbR^d;\bbR^d)$ for $\alpha \in (\max \{0, 2s-1\},1]$.
	
	If $\alpha > 2s$, then $\bbL^s \bu \in C^{1,\alpha-2s}(\bbR^d;\bbR^d)$ and
	\begin{equation}\label{eq:HolderSpace1:Estimate1}
		[\bbL^s \bu]_{C^{1,\alpha-2s}(\bbR^d)} \leq C [\bu]_{C^{1,\alpha}(\bbR^d)}\,.
	\end{equation}
	
	If $\alpha < 2s$, then $\bbL^s \bu \in C^{0,\alpha-2s+1}(\bbR^d;\bbR^d)$ and
	\begin{equation}\label{eq:HolderSpace1:Estimate2}
		[\bbL^s \bu]_{C^{0,\alpha-2s+1}(\bbR^d)} \leq C [\bu]_{C^{1,\alpha}(\bbR^d)}\,.
	\end{equation}
	The constant $C$ depends only on $d$, $s$ and $\alpha$.
\end{theorem}

\begin{theorem}[See {\cite[Proposition 2.7]{silvestre2007regularity}}]\label{thm:MappingProperties:HigherDerivatives}
	Let $k \in \bbN$, $k \neq 1$. Suppose that $\bu \in L^1_{s}(\bbR^d;\bbR^d)$ and $\bu \in C^{k,\alpha}$ for $\alpha \in (0,1]$. Suppose also that $(k + \alpha - 2s) - \lfloor k + \alpha - 2s \rfloor \neq 0$. Then
	\begin{equation*}
		\bbL^s \bu \in C^{\lfloor k + \alpha - 2s \rfloor,(k + \alpha - 2s) - \lfloor k + \alpha - 2s \rfloor}(\bbR^d;\bbR^d)\,,
	\end{equation*}
	with the estimate
	\begin{equation*}
		[\bbL^s \bu]_{C^{\lfloor k + \alpha - 2s \rfloor,(k + \alpha - 2s) - \lfloor k + \alpha - 2s \rfloor}(\bbR^d)} \leq C [\bu]_{C^{k,\alpha}(\bbR^d)}\,.
	\end{equation*}
	The constant $C$ depends on $d$, $k$, $s$, and $\alpha$.
\end{theorem}

With these mapping properties in hand we can now establish Schauder estimates for solutions of $\bbL^s \bu = \bff$.

\begin{theorem}[See {\cite[Proposition 2.8]{silvestre2007regularity}}]\label{thm:Schauder}
	Let $\bff \in C^{0,\alpha}(\bbR^d;\bbR^d)$ for $\alpha \in (0,1]$ and let $\bu \in L^{\infty}(\bbR^d;\bbR^d)$. Suppose that $\bbL^s \bu(\bx) = \bff(\bx)$ for all $\bx \in \bbR^d$.
	
	If $\alpha + 2s < 1$, then $\bu \in C^{0,\alpha+2s}(\bbR^d;\bbR^d)$, with
	\begin{equation}\label{eq:SchauderEstimate1}
		[\bu]_{C^{0,\alpha+2s}(\bbR^d)} \leq C \left( \Vnorm{\bu}_{L^{\infty}(\bbR^d)} + \Vnorm{\bff}_{C^{0,\alpha}(\bbR^d)} \right)\,,
	\end{equation}
	and the constant $C$ depends only on $d$, $s$ and $\alpha$.
	
	If $\alpha + 2s > 1$, then $\bu \in C^{1,\alpha+2s-1}(\bbR^d;\bbR^d)$, with
	\begin{equation}\label{eq:SchauderEstimate2}
		[\bu]_{C^{1,\alpha+2s-1}(\bbR^d)} \leq C \left( \Vnorm{\bu}_{L^{\infty}(\bbR^d)} + \Vnorm{\bff}_{C^{0,\alpha}(\bbR^d)} \right)\,,
	\end{equation}
	and the constant $C$ depends only on $d$, $s$ and $\alpha$.
\end{theorem}

\begin{proof}
	In the case $\alpha + 2s \leq 1$, we will show that
	\begin{equation}\label{eq:SchauderProofA}
		[\bu]_{C^{0,\alpha+2s}(B(0,1/4))} \leq C(d,s,\alpha) \left( \Vnorm{\bu}_{L^{\infty}(\bbR^d)} + \Vnorm{\bff}_{C^{0,\alpha}(\bbR^d)} \right)\,.
	\end{equation}
	By translation invariance of the estimates we can repeat the argument on $B(\bx_0,1/4)$ for any point $\bx_0 \in \bbR^d$, and thus obtain \eqref{eq:SchauderEstimate1}.
	
	Let $\eta$ be a mollifier belonging to $C^{\infty}_c(B(0,2))$ with $\eta \equiv 1$ on $B(0,1)$. Define
	\begin{equation*}
		\wt{\bu}(\bx) := \bbE^{s}[\eta \bff](\bx) = \intdm{\bbR^d}{ \bsPsi^s(\bx-\by) \eta(\by) \bff(\by) }{\by}\,.
	\end{equation*}
	Then
	\begin{equation*}
		|\wt{\bu}(\bx)| \leq C \Vnorm{\bff}_{L^{\infty}(\bbR^d)} \int_{\bbR^d} \frac{\eta(\by)}{|\bx-\by|^{d-2s}} \, \rmd \by = C \Vnorm{\bff}_{L^{\infty}(\bbR^d)} I^s \eta(\bx)\,,
	\end{equation*}
	where $C$ depends on $d$, $s$, $\mu$ and $\lambda$. Therefore $\bbL^s \wt{\bu}$ exists as a distribution in $\scS_s'(\bbR^d;\bbR^d)$ and 
	\begin{equation*}
		\Vint{\bbL^s (\bu-\wt{\bu}), \bsvarphi} = 0 \text{ for all } \bsvarphi \in C^{\infty}_c(B(0,1/2))\,.
	\end{equation*}
	So $\bbL^s (\bu-\wt{\bu}) = {\bf 0}$ in $B(0,1/2)$ and thus $\bbL(\bu-\wt{\bu}) = {\bf 0}$ in $B(0,1/2)$. We can then use the fact that $\bbL$ is a hypoelliptic operator, and use derivative estimates obtained from the mean value formula for $\bbL$ (see \cite[Theorem 10.20]{mitrea} for details) to arrive at the estimate
	\begin{equation}\label{eq:SchauderProof1}
		[\bu-\wt{\bu}]_{C^{0,\alpha+2s}(B(0,1/2))} \leq  C \left( \Vnorm{\bu}_{L^{\infty}(\bbR^d)} + \Vnorm{\bff}_{L^{\infty}(\bbR^d)} \right)\,.
	\end{equation}
	All we need to do now is show the same estimate for $\wt{\bu}$. If $s < 1/2$, then $\bbE^s [\eta \bff] = \bbL^{1-s} \circ  \bbE[\eta \bff]$. From the $C^{2,\alpha}$-estimates for the Poisson equation $\bbL \bv = \eta \bff$ (c.f. \cite{Giaquinta}) we have
	\begin{equation*}
		[\bbE[\eta \bff]]_{C^{2,\alpha}(B(0,1/2))} \leq C \Vnorm{\bff}_{C^{0,\alpha}(B(0,2))}\,.
	\end{equation*}
	Therefore by \Cref{thm:MappingProperties:HigherDerivatives}
	\begin{equation}\label{eq:SchauderProof2}
		[\wt{\bu}]_{C^{0,\alpha+2s}(B(0,1/4))} \leq C [\bbE[\eta \bff]]_{C^{2,\alpha}(B(0,1/2))} \leq C \Vnorm{\bff}_{C^{0,\alpha}(B(0,2))}\,,
	\end{equation}
	which combined with \eqref{eq:SchauderProof1} gives \eqref{eq:SchauderProofA}. So \eqref{eq:SchauderEstimate1} is proved.
	The estimate \eqref{eq:SchauderEstimate2} follows by using exactly the same argument.
\end{proof}

\begin{theorem}[See {\cite[Proposition 2.9]{silvestre2007regularity}}]
	Let $d \geq 3$, $\bff \in L^{\infty}(\bbR^d;\bbR^d)$ and let $\bu \in L^{\infty}(\bbR^d;\bbR^d)$. Suppose that $\bbL^s \bu(\bx) = \bff(\bx)$ for all $\bx \in \bbR^d$.
	
	If $2s < 1$, then $\bu \in C^{0,\alpha}(\bbR^d;\bbR^d)$ for any $\alpha < 2s$, with
	\begin{equation}\label{eq:SchauderEstimate3}
		[\bu]_{C^{0,\alpha}(\bbR^d)} \leq C \left( \Vnorm{\bu}_{L^{\infty}(\bbR^d)} + \Vnorm{\bff}_{L^{\infty}(\bbR^d)} \right)\,,
	\end{equation}
	and the constant $C$ depends only on $d$, $s$ and $\alpha$.
	
	If $2s > 1$, then $\bu \in C^{1,\alpha}(\bbR^d;\bbR^d)$ for any $\alpha < 2s -1$, with
	\begin{equation}\label{eq:SchauderEstimate4}
		[\bu]_{C^{1,\alpha}(\bbR^d)} \leq C \left( \Vnorm{\bu}_{L^{\infty}(\bbR^d)} + \Vnorm{\bff}_{L^{\infty}(\bbR^d)} \right)\,,
	\end{equation}
	and the constant $C$ depends only on $d$, $s$ and $\alpha$.
\end{theorem}

\begin{proof}
	The proof is identical to the proof of \Cref{thm:Schauder}, except we use $C^{1,\alpha}$ estimates for solutions to $\bbL \bv = \eta \bff$ instead of the $C^{2,\alpha}$ estimates.
	We could not find a reference in the literature for this estimate, so we reproduce it in the next lemma.
\end{proof}

\begin{lemma}
	Let $d \geq 3$, and suppose that $\bbL \bv = \bff$ for all $\bx \in B(0,1)$, with $\bff \in L^{\infty}(B(0,2))$ and $\bff = 0$ outside of $B(0,2)$. Then for any $\alpha \in (0,1)$
	\begin{equation}\label{eq:C1AlphaEstimate}
		\Vnorm{\bv}_{C^{1,\alpha}(B(0,1/2))} \leq C \Vnorm{\bff}_{L^{\infty}(B(0,2))}\,,
	\end{equation}
	where the constant $C$ depends only on $d$ and $\alpha$.
\end{lemma}

\begin{proof}
	The solution is given by
	\begin{equation*}
		\bv(\bx) := \int_{B(0,2)} \bsPsi^1(\bx-\bz) \bff(\bz) \, \rmd \bz\,,
	\end{equation*}
	which is clearly in $L^{\infty}(B(0,1/2))$.
	
	For $k \in \{1, \ldots, d\}$ and $\bx$, $\by \in B(0,1/2)$ we have
	\begin{equation*}
		\begin{split}
			|\p_k \bv(\bx) - \p_k \bv(\by)| &\leq C(d,\mu,\lambda) \Vnorm{\bff}_{L^{\infty}(B(0,2))} \\
			&\quad \times \sum_{i,j,k} \int_{B(0,2)} \Bigg(  \left| \frac{z_k-x_k}{|\bz-\bx|^d} - \frac{z_k-y_k}{|\bz-\by|^d} \right| \delta_{ij} \\
				&\qquad +  \left| \frac{(z_i-x_i)(z_j-x_j)(z_k-x_k)}{|\bz-\bx|^{d+2}} - \frac{(z_i-y_i)(z_j-y_j)(z_k-y_k)}{|\bz-\by|^{d+2}} \right| \\
				&\qquad +  \left| \frac{z_j-x_j}{|\bz-\bx|^d} - \frac{z_j-y_j}{|\bz-\by|^d} \right| \delta_{ik}
				+ \left| \frac{z_i-x_i}{|\bz-\bx|^d} - \frac{z_i-y_i}{|\bz-\by|^d} \right| \delta_{jk} \Bigg) \, \rmd \bz\,.
		\end{split}
	\end{equation*}
	We split the integration domain into two regions $B(\bx,2|\bx-\by|)$ and $B(0,2) \setminus B(\bx,2|\bx-\by|)$.
	For the first integral, each term has the bound
	\begin{equation*}
		\frac{C}{|\bz-\bx|^{d-1}} + \frac{C}{|\bz-\by|^{d-1}}\,,
	\end{equation*} 
	where $C$ depends only on the dimension $d$.
	Since $B(\bx,2|\bx-\by|) \subset B(\by,3|\bx-\by|)$, 
	after integrating in polar coordinates the first integral is majorized by 
	$$
	C\int_{B(0,2|\bx-\by|)} \frac{1}{|\bz|^{d-1}} \, \rmd \bz + C\int_{B(0,3|\bx-\by|)} \frac{1}{|\bz|^{d-1}} \, \rmd \bz = C|\bx-\by|\,.
	$$
	As for the second integral, by the mean value theorem each term has the bound
	\begin{equation*}
		\frac{C |\bx-\by|}{|\bw-\bz|^d}\,,
	\end{equation*}
	where $C$ depends only on the dimension $d$ and where $\bw$ is some fixed value in $B(\bx,|\bx-\by|)$ depending on $\bx$ and $\by$. Then since $\bbR^d \setminus B(\bx,2|\bx-\by|) \subset \bbR^d \setminus B(\bw,|\bx-\by|)$ the second integral is majorized by
	$$
		C\int_{B(0,2) \setminus B(\bw,|\bx-\by|)} \frac{|\bx-\by|}{|\bw-\bz|^{d}} \, \rmd \bz = C |\bx-\by| \ln \left( \frac{2}{|\bx-\by|} \right)\,.
	$$
	Putting these two bounds together gives the ``log-Lipschitz'' bound on the partial derivative
	\begin{equation*}
		|\p_k \bv(\bx) - \p_k \bv(\by)| \leq C \Vnorm{\bff}_{L^{\infty}(B(0,2))} |\bx-\by| \ln \left( \frac{2}{|\bx-\by|} \right) \qquad \text{ for all } \bx, \by \in B(0,1/2)\,.
	\end{equation*}
	Then the estimate \eqref{eq:C1AlphaEstimate} follows from the fact that for any $\alpha \in (0,1)$ there exists a $C$ such that $t \ln(2/t) \leq C t^{\alpha}$ for all $t \in [0,1]$.
\end{proof}

\section{$\bbL^s$ and Bessel Potential Spaces}\label{sec:BesselSpaces}

\subsection{Notation}

For $1 < p < \infty$ and for $s > 0$ denote the Bessel potential spaces
\begin{equation*}
	\scL^{s,p}(\bbR^d;\bbR^d) := \{ \bu \in \scS'(\bbR^d;\bbR^d) \, : \, \cF^{-1} \left( (1 + 4 \pi^2 |\bsxi|^2 )^{s/2} \wh{\bu} \right) \in L^p(\bbR^d;\bbR^d)  \}\,,
\end{equation*}
with norm
\begin{equation*}
	\Vnorm{\bu}_{\scL^{s,p}(\bbR^d)} := \Vnorm{ \cF^{-1} \left( (1 + 4 \pi^2 |\bsxi|^2 )^{s/2} \wh{\bu} \right) }_{L^p(\bbR^d)}\,.
\end{equation*}
We denote the dual of $\scL^{s,p}(\bbR^d;\bbR^d)$ as $\scL^{-s,p'}(\bbR^d;\bbR^d)$.
Denote the homogeneous spaces
\begin{equation*}
	\dot{\scL}^{s,p}(\bbR^d;\bbR^d) := \{ \bu \in \scS'(\bbR^d;\bbR^d) \, : \, \cF^{-1} \left( (4 \pi^2 |\bsxi|^2 )^{s/2} \wh{\bu} \right) = (-\Delta)^{\frac{s}{2}} \bu \in L^p(\bbR^d;\bbR^d)  \}\,,
\end{equation*}
and denote the seminorm
\begin{equation*}
	[\bu]_{\dot{\scL}^{s,p}(\bbR^d)} := \Vnorm{(-\Delta)^{\frac{s}{2}} \bu}_{L^p(\bbR^d)}\,.
\end{equation*}
Then
\begin{equation*}
	\Vnorm{\bu}_{\dot{\scL}^{s,p}(\bbR^d)} \approx \Vnorm{\bu}_{L^p(\bbR^d)} + [\bu]_{\dot{\scL}^{s,p}(\bbR^d)}\,.
\end{equation*}
When $p =2$, $\scL^{s,2}$ has an equivalent characterization via the Gagliardo seminorm (see \cite[Proposition 3.6]{Hitchhiker}):
\begin{equation}\label{eq:EquivalenceOfSobolevAndBessel}
	[\bu]_{\scH^s(\bbR^d)}^2 := \iintdm{\bbR^d}{\bbR^d}{\frac{|\bu(\bx)-\bu(\by)|^2}{|\bx-\by|^{d+2s}} }{\by}{\bx} = \frac{2}{c_{d,s}} [\bu]_{\dot{\scL}^{s,2}(\bbR^d)}^2 = \frac{2}{c_{d,s}} \Vnorm{(-\Delta)^{\frac{s}{2}} \bu}_{L^2(\bbR^d)}^2\,.
\end{equation}

\subsection{Negative Powers of $\bbL$}

The Lam\'e-Navier potentials $\bbE^s$ satisfy the following Sobolev inequality:
\begin{theorem}
	Let $s \in (0,\frac{d}{2})$, let $p \in (1,\infty)$ and let $p^* := \frac{dp}{d-2sp}$. Let $\bu \in L^p(\bbR^d;\bbR^d)$. Then the integral defining $\bbE^s \bu$ converges absolutely for almost every $\bx \in \bbR^d$. Furthermore, there exists $C = C(d,s,p,\mu,\lambda)$ such that
	\begin{equation*}
		\Vnorm{\bbE^s \bu}_{L^{p^*}(\bbR^d)} \leq C \Vnorm{\bu}_{L^p(\bbR^d)}\,.
	\end{equation*}
	If $p = 1$, then $\bbE^s \bu$ is still defined by an absolutely convergent integral, and there exists $C = C(d,s,\mu,\lambda)$ such that
	\begin{equation*}
		\big| \{ \bx \, : \, |\bbE^s \bu(\bx)| > \lambda \} \big| \leq C \left( \frac{\Vnorm{\bu}_{L^1(\bbR^d)}}{\lambda} \right)^{\frac{n}{n-2s}} \,, \quad \forall \lambda > 0\,.
	\end{equation*}
\end{theorem}

\begin{proof}
	Using the definition of $\bsPsi^s$, we have
	\begin{equation*}
		|\bbE^s \bu(\bx)| \leq C(d,s,\mu,\lambda) \intdm{\bbR^d}{ \frac{ |\bu(\bx-\by)|}{|\by|^{d-2s}} }{\by} = C \, I^{s} (|\bu|)(\bx)\,.
	\end{equation*}
	The proof then proceeds exactly as that of \cite[Theorem 1.2.3]{grafakos2009modern} or \cite[Chapter V, Theorem 1]{Stein}.
\end{proof}

\subsection{Positive Powers of $\bbL$: Pointwise Definition}

The goal of this subsection is to prove the following:

\begin{theorem}\label{thm:BesselSpaces:MainThm}
	Let $\bu : \bbR^d \to \bbR^d$, $0 < s < 1$ and $1 < p < \infty$.
	Define for $\veps > 0$
	\begin{multline}\label{eq:BesselSpaces:VepsOp}
		\bbL^s_{\veps} \bu(\bx) := \left( \frac{(2s+1)\mu^s - (2\mu+\lambda)^s}{2s}\right) c_{d,s} \intdm{|\bx-\by|\geq \veps}{\frac{\bu(\bx) - \bu(\by)}{|\bx-\by|^{d+2s}}}{\by}  \\
		+ \frac{(2 \mu + \lambda)^s - \mu^s}{2s} \kappa_{d,s} \intdm{|\bx-\by|\geq \veps}{ \left( \frac{(\bx-\by) \otimes (\bx-\by)}{|\bx-\by|^2} \right) \frac{\bu(\bx) - \bu(\by)}{|\bx-\by|^{d+2s}}}{\by}\,.
	\end{multline}
	Then $\bu \in \scL^{2s,p}(\bbR^d;\bbR^d)$ if and only if $\bbL^s_{\veps} \bu(\bx)$ converges in $L^p(\bbR^d;\bbR^d)$ as $\veps \to 0$ to a function $\bu_0$. Moreover, there exists a constant $C(d,s) >0$ such that
	\begin{equation*}
		C^{-1} \Vnorm{\bu}_{\scL^{2s,p}(\bbR^d)} \leq \Vnorm{\bu_0}_{L^p(\bbR^d)} \leq C \Vnorm{\bu}_{\scL^{2s,p}(\bbR^d)}\,.
	\end{equation*}
\end{theorem}

Since any function in $\scL^{2s,p}(\bbR^d;\bbR^d)$ defines a tempered distribution,  the operators $\bbL^s_{\veps} \bu$ converge to the $\bbL^s \bu$ in the topology on $\scS_s'(\bbR^d;\bbR^d)$. By Theorem \ref{thm:BesselSpaces:MainThm}, we can identify the operator $\bbL^s \bu$ with the $L^p(\bbR^d;\bbR^d)$ function $\bu_0(\bx)$ (this is exactly how the fractional Laplacian is defined to act on functions in $\scL^{2s,p}(\bbR^d)$). Therefore we have a new characterization of vector fields in Bessel potential spaces:
\begin{equation*}
	\scL^{2s,p}(\bbR^d;\bbR^d) = \left\{ \bu \in L^p(\bbR^d;\bbR^d) \, : \, \bbL^s \bu \in L^p(\bbR^d;\bbR^d) \right\}\,, \qquad s \in (0,1) \text{ and } p \in (1,\infty)\,.
\end{equation*}

Theorem \ref{thm:BesselSpaces:MainThm} follows from the analogous theorem for $(-\Delta)^s$ (see \cite{Stein, wheeden1968hypersingular}) and the following theorem:
\begin{theorem}\label{thm:BesselSpaces:MainThm2}
	Let $\bu \in \scL^{2s,p}(\bbR^d;\bbR^d)$ with $0 < s < 1$ and $1 < p < \infty$.
	Define for $\veps > 0$
	\begin{equation}\label{eq:BesselSpaces:VepsOp2}
		\bbF^s_{\veps} \bu(\bx) := \kappa_{d,s} \intdm{|\bx-\by|\geq \veps}{ \left( \frac{(\bx-\by) \otimes (\bx-\by)}{|\bx-\by|^2} \right) \frac{\bu(\bx) - \bu(\by)}{|\bx-\by|^{d+2s}}}{\by}\,.
	\end{equation}
	Then $\bbF^s_{\veps} \bu(\bx)$ converges in $L^p(\bbR^d;\bbR^d)$ as $\veps \to 0$ to a function $\bv_s$ that satisfies the estimate
	\begin{equation}\label{eq:MainBesselEstimate}
		\Vnorm{\bv_s}_{L^p(\bbR^d)} \leq C \Vnorm{\bu}_{\cL^{2s,p}(\bbR^d)}\,,
	\end{equation}
	where $C$ depends only on $d$ and $s$.
\end{theorem}

We postpone the proof of \Cref{thm:BesselSpaces:MainThm2} to the next section and focus on several of its consequences here.
First, Theorem \ref{thm:BesselSpaces:MainThm} allows us to state the mapping properties of $\bbL^s$ on Bessel potential spaces:
\begin{theorem}
	Let $s \in (0,1)$, $1 < p < \infty$, and $t \geq 0$. Then the operator $\bbL^s$ is a well-defined map from $\scL^{t+2s,p}(\bbR^d;\bbR^d)$ to $\scL^{t,p}(\bbR^d;\bbR^d)$.
\end{theorem}
Next, the following strong solvability result in Bessel spaces holds for the positive-definite operator $\bbL^s + q \bI$ on all of Euclidean space.

\begin{theorem}
	For $1 < p < \infty$ and $q > 0$, corresponding to any $\bff \in L^p(\bbR^d;\bbR^d)$ the equation
	\begin{equation*}
		\bbL^s \bu(\bx) + q \bu(\bx) = \bff(\bx)\,, \quad \bx \in \bbR^d\,,
	\end{equation*}
	has a unique strong solution $\bu \in \scL^{2s,p}(\bbR^d;\bbR^d)$ satisfying the estimate
	\begin{equation*}
		\Vnorm{\bu}_{\scL^{2s,p}(\bbR^d)} \leq C \Vnorm{\bff}_{L^p(\bbR^d)}
	\end{equation*}
	for $C$ depending only on $d$, $s$, $p$ and $q$.
\end{theorem}

\begin{proof}
	Thanks to \Cref{thm:BesselSpaces:MainThm} the action of $\bbL^s$ on Bessel functions is well-defined.
	Thus the theorem will be proved if we follow the same procedure used to prove \cite[Theorem 2.3]{mengesha2020solvability}. We just need to show that for any $q > 0$ the Fourier matrix symbols
	\begin{equation*}
		\frak{M}^s(\bsxi) := (1 + 4 \pi^2 |\bsxi|^2 )^{-s} ( \bM(\bsxi)^s + q \bI) \text{ and } [\frak{M}^s(\bsxi)]^{-1}
	\end{equation*}
	are both $L^p$-multipliers.
	
	To see this we invoke \cite[Lemma 4.1]{mengesha2020solvability}.
	The symbols $\frak{M}^s(\bsxi)$ and $[\frak{M}^s(\bsxi)]^{-1}$ have the explicit expressions
	\begin{gather*}
		\frak{M}^s(\bsxi) = \left( \frac{4 \pi^2 |\bsxi|^2 }{1 + 4 \pi^2 |\bsxi|^2 } \right)^s \left( \mu^s \bI + \big( (2 \mu + \lambda)^s - \mu^s \big) \frac{\bsxi \otimes \bsxi}{|\bsxi|^2} \right) + \frac{q}{(1 + 4 \pi^2 |\bsxi|^2)^s}\,, \\
		[\frak{M}^s(\bsxi)]^{-1} = \frac{(1 + 4 \pi^2 |\bsxi|^2 )^s}{\mu^s (4 \pi^2 |\bsxi|^2 )^s + q} \left( \bI - \frac{ \big( (2 \mu + \lambda)^s - \mu^s \big) (4 \pi^2 |\bsxi|^2 )^s }{ (2 \mu + \lambda)^s (4 \pi^2 |\bsxi|^2 )^s + q } \frac{\bsxi \otimes \bsxi}{|\bsxi|^2} \right)\,.
	\end{gather*}
	This multiplier and its inverse are exactly of the forms considered in \cite[Lemma 4.1]{mengesha2020solvability}, and so by a line-by-line adaptation of that proof $\frak{M}^s(\bsxi)$ and $[\frak{M}^s(\bsxi)]^{-1}$ are $L^p$-multipliers.
\end{proof}

\subsection{Positive Powers of $\bbL$: Distributional Definition}\label{subsec:NonlocalWeakForm}

We can use nonlocal integration by parts to define a distributional form of $\bbL^s \bu$ for functions belonging to a Bessel space of order less than $2s$. For the moment let $\bu, \bsvarphi \in C^{\infty}_c(\bbR^d;\bbR^d)$.
For $\veps > 0$ and $\bbL^s_{\veps}$ as in \eqref{eq:BesselSpaces:VepsOp},
\begin{equation*}
	\begin{split}
		&\int_{\bbR^d} \bbL^s_{\veps} \bu(\bx) \cdot \bsvarphi(\bx) \, \rmd \bx \\
		&= \left( \frac{(2s+1)\mu^s - (2\mu+\lambda)^s}{2s}\right) c_{d,s} \iint_{|\bx-\by|\geq \veps} \frac{\bu(\bx) - \bu(\by)}{|\bx-\by|^{d+2s}} \bsvarphi(\bx) \, \rmd \by \, \rmd \bx \\
		&\quad + \frac{(2 \mu + \lambda)^s - \mu^s}{2s} \kappa_{d,s}  \iint_{|\bx-\by|\geq \veps} \frac{ \left( \big( \bu(\bx) - \bu(\by) \big) \cdot  \frac{\bx-\by}{|\bx-\by|} \right) }{|\bx-\by|^{d+2s}} \bsvarphi(\bx) \cdot \frac{\bx-\by}{|\bx-\by|} \, \rmd \by \, \rmd \bx\,.
	\end{split}
\end{equation*}
By splitting each integral, switching the roles of $\bx$ and $\by$, and then using Fubini's theorem and recombining the integrals (c.f. \cite{KassmannMengeshaScott, Felsinger}),
\begin{equation*}
	\begin{split}
		&\int_{\bbR^d} \bbL^s_{\veps} \bu(\bx) \cdot \bsvarphi(\bx) \, \rmd \bx \\
		&= \left( \frac{(2s+1)\mu^s - (2\mu+\lambda)^s}{4s}\right) c_{d,s} \iint_{|\bx-\by|\geq \veps} \frac{(\bu(\bx) - \bu(\by)) \cdot (\bsvarphi(\bx) -\bsvarphi(\by)) }{|\bx-\by|^{d+2s}} \, \rmd \by \, \rmd \bx \\
		&\quad + \frac{(2 \mu + \lambda)^s - \mu^s}{4s} \kappa_{d,s}  \iint_{|\bx-\by|\geq \veps} \frac{ \left( \big( \bu(\bx) - \bu(\by) \big) \cdot  \frac{\bx-\by}{|\bx-\by|} \right) \left( \big( \bsvarphi(\bx) - \bsvarphi(\by) \big) \cdot  \frac{\bx-\by}{|\bx-\by|} \right) }{|\bx-\by|^{d+2s}} \, \rmd \by \, \rmd \bx \\
		&:= \cE^s_\veps(\bu,\bsvarphi)\,.
	\end{split}
\end{equation*}
So for each $\veps$, $\bbL^s_{\veps}$ defines the bilinear form $\cE^s_{\veps}$ in a natural way. By H\"older's inequality and the elementary inequality $|\ba \cdot \bfb| \leq |\ba| |\bfb|$
\begin{equation}\label{eq:BiFormCont}
	|\cE^s_{\veps}(\bu,\bsvarphi)| \leq C(\mu,\lambda,d,s) [\bu]_{\dot{\scL}^{s,2}(\bbR^d)} [\bsvarphi]_{\dot{\scL}^{s,2}(\bbR^d)}\,,
\end{equation}
where we have used the identification \eqref{eq:EquivalenceOfSobolevAndBessel}. So by the dominated convergence theorem $\cE^s_{\veps}$ converges as $\veps \to 0$ to the bilinear form
\begin{equation*}
	\begin{split}
		\cE^s(\bu,\bsvarphi)
		&:= \frac{\mu^s - A_{\mu,\lambda,s}}{2} c_{d,s} \int_{\bbR^d} \int_{\bbR^d} \frac{(\bu(\bx) - \bu(\by)) \cdot (\bsvarphi(\bx) -\bsvarphi(\by)) }{|\bx-\by|^{d+2s}} \, \rmd \by \, \rmd \bx \\
		&\quad + \frac{A_{\mu,\lambda,s}}{2} \kappa_{d,s}  \int_{\bbR^d} \int_{\bbR^d} \frac{ \Big( \big( \bu(\bx) - \bu(\by) \big) \cdot  \frac{\bx-\by}{|\bx-\by|} \Big) \Big( \big( \bsvarphi(\bx) - \bsvarphi(\by) \big) \cdot  \frac{\bx-\by}{|\bx-\by|} \Big) }{|\bx-\by|^{d+2s}} \, \rmd \by \, \rmd \bx
	\end{split}
\end{equation*}
with $A_{\mu,\lambda,s} := \frac{(2 \mu + \lambda)^2 - \mu^s}{2}$.
This bilinear form satisfies \eqref{eq:BiFormCont}, so $\cE^s$ is continuous on $\scL^{s,2}(\bbR^d;\bbR^d)$.
We therefore define a distributional form of the operator $\bbL^s$ for any $\bu \in \scL^{s,2}(\bbR^d;\bbR^d)$ via
\begin{equation}\label{eq:NonlocalWeakForm}
	\Vint{ \bbL^s \bu, \bsvarphi } := \cE^s(\bu,\bsvarphi) \qquad \text{ for any } \bsvarphi \in \scL^{s,2}(\bbR^d;\bbR^d)\,.
\end{equation}

\section{Proof of \Cref{thm:BesselSpaces:MainThm2}}\label{sec:BesselSpacesProof}

We start by proving Theorem \ref{thm:BesselSpaces:MainThm2} for smooth functions.

\begin{theorem}\label{thm:BesselSpaces:MainThm3}
	For any $\bu \in C^{\infty}_c(\bbR^d;\bbR^d)$, $\bbF^s_{\veps} \bu$ converges in $L^p(\bbR^d;\bbR^d)$ as $\veps \to 0$, and 
	\begin{equation*}
		\sup_{\veps > 0} \Vnorm{\bbF^s_{\veps} \bu}_{L^p(\bbR^d)} \leq C(d,s) \Vnorm{\bu}_{\cL^{2s,p}(\bbR^d)}\,.
	\end{equation*}
\end{theorem}

To prove Theorem \ref{thm:BesselSpaces:MainThm2} we will need the following auxiliary functions defined through a Poisson-type kernel. For $\veps > 0$ and $\bu \in C^{\infty}_c(\bbR^d;\bbR^d)$, define
\begin{equation*}
	\bbG^s_{\veps} \bu(\bx) := \intdm{\bbR^d}{\Upsilon^{s,\veps}(\bz)  \bu(\bx+\bz) }{\bz}\,,
\end{equation*}
where $\Upsilon^{s,\veps} : \bbR^d \to \bbR^{d \times d}$ is defined as
\begin{equation*}
	\Upsilon^{s,\veps}(\bx) := \cF^{-1} \left[ \left( (2 \pi |\bsxi|)^{2s} \left( 2s \frac{\bsxi \otimes \bsxi}{|\bsxi|^2} + \bI \right) \rme^{-2 \pi \veps |\bsxi|}  \right) \right] (\bx)\,.
\end{equation*}
Since $\wh{\Upsilon}_{s,\veps} \in L^1(\bbR^d)$, $\Upsilon^{s,\veps} \in L^{\infty}(\bbR^d)$, and so $\bbG^s_{\veps} \bu$ is a well-defined function. The functions $\bbG^s_{\veps} \bu$ will serve as an intermediate approximation for the estimate in \Cref{thm:BesselSpaces:MainThm3}, as the next lemma shows:

\begin{lemma}\label{lma:BesselSpaces:MainLemma}
	For any $\bu \in C^{\infty}_c(\bbR^d;\bbR^d)$, $0 < s < 1$, $1 < p < \infty$,
	\begin{equation*}
		\bbF^s_{\veps} \bu - \bbG^s_{\veps} \bu \text{  converges in } L^p(\bbR^d;\bbR^d) \text{ as } \veps \to 0\,,
	\end{equation*}
	and
	\begin{equation*}
		\sup_{\veps > 0} \Vnorm{\bbF^s_{\veps} \bu - \bbG^s_{\veps} \bu}_{L^p(\bbR^d)} \leq C(d,s) \Vnorm{\bu}_{\cL^{2s,p}(\bbR^d)}\,.
	\end{equation*}
\end{lemma}
We postpone the proof of \Cref{lma:BesselSpaces:MainLemma} for now, and first use it to prove \Cref{thm:BesselSpaces:MainThm3}.

\begin{proof}[proof of \Cref{thm:BesselSpaces:MainThm3}]
	by Parseval's relation
	\begin{equation*}
		\begin{split}
			\bbG^s_{\veps} \bu(\bx) &= \intdm{\bbR^d}{ \left( (2 \pi |\bsxi|)^{2s} \left( 2s \frac{\bsxi \otimes \bsxi}{|\bsxi|^2} + \bI \right) \rme^{-2 \pi \veps |\bsxi|}  \right) \rme^{-2 \pi \imath \bx \cdot \bsxi} \wh{\bu}(\bsxi) }{\bsxi} \\
			&= \intdm{\bbR^d}{ \left(  \left( 2s \frac{\bsxi \otimes \bsxi}{|\bsxi|^2} + \bI \right) \rme^{-2 \pi (\veps |\bsxi| +  \imath \bx \cdot \bsxi)}  \right)  \frac{(2 \pi |\bsxi|)^{2s}}{(1+4 \pi^2 |\bsxi|^2)^s} (1+4 \pi^2 |\bsxi|^2)^s \wh{\bu}(\bsxi) }{\bsxi} \\
			&= \intdm{\bbR^d}{  \left( 2s \frac{\bsxi \otimes \bsxi}{|\bsxi|^2} + \bI \right)  \rme^{-2 \pi (\veps |\bsxi| +  \imath \bx \cdot \bsxi)}  \cF[\bff \ast \mu](\bsxi)  }{\bsxi}\,,
		\end{split}
	\end{equation*}
	where $\rmd \mu$ is a finite measure on $\bbR^d$ (see \cite[Chapter V]{Stein}) and $\bff(\bx) := \cF^{-1} \big[ (1+4 \pi^2 |\bsxi|^2)^s \wh{\bu} \big](\bx)$. Since compositions of Riesz transforms are $L^p$-multipliers, we conclude that
	\begin{equation*}
		\wt{\bff}(\bx) := \cF^{-1} \left[ \left( 2s \frac{\bsxi \otimes \bsxi}{|\bsxi|^2} + \bI \right) \cF[\bff \ast \mu] \right] (\bx) \in L^p(\bbR^d;\bbR^d)\,,
	\end{equation*}
	with the estimate
	\begin{equation*}
		\Vnorm{	\wt{\bff}}_{L^p(\bbR^d)} \leq C(d,s) \Vnorm{\bu }_{\cL^{2s,p}(\bbR^d)}\,.
	\end{equation*}
	Therefore
	\begin{equation*}
		\bbG^s_{\veps} \bu(\bx) = \intdm{\bbR^d}{ \cF(\wt{\bff})(\bsxi) \rme^{-2 \pi \imath \bx \cdot \bsxi}  \rme^{-2 \pi \veps |\bsxi|}}{\bsxi} = p_{\veps} \ast \wt{\bff}(-\bx)\,,
	\end{equation*}
	where $p_{\veps}(\bx)$ is the Poisson integral for the upper-half space $\bbR^{d+1}_+$. Thus both desired properties for $\bbG^{\veps}_s \bu$ can be easily obtained using the Poisson integral (see \cite[Chapter III]{Stein}); we have
	\begin{equation*}
		\Vnorm{\bbG^s_{\veps} \bu}_{L^p(\bbR^d)} = \Vnorm{p_{\veps} \ast \wt{\bff}}_{L^p(\bbR^d)} \leq C(d,s) \Vnorm{\wt{\bff}}_{L^p(\bbR^d)} \leq C(d,s)\Vnorm{\bu }_{\cL^{2s,p}(\bbR^d)}\,.
	\end{equation*}
	and $\bbG^s_{\veps} \bu(\bx) = p_{\veps} \ast \wt{\bff}(-\bx)$ converges in $L^p(\bbR^d)$ as $\veps \to 0$.
	
	The result now follows easily by Lemma \ref{lma:BesselSpaces:MainLemma}:
	\begin{equation*}
		\Vnorm{\bbF^s_{\veps} \bu}_{L^p(\bbR^d)} \leq \Vnorm{\bbF^s_{\veps} \bu - \bbG^s_{\veps} \bu}_{L^p(\bbR^d)} + \Vnorm{\bbG^s_{\veps} \bu}_{L^p(\bbR^d)} \leq C(d,s) \Vnorm{\bu}_{\scL^{2s,p}(\bbR^d)}\,,
	\end{equation*}
	and $\bbF^s_{\veps} \bu = (\bbF^s_{\veps} \bu - \bbG^s_{\veps} \bu) + \bbG^s_{\veps} \bu$ converges in $L^p(\bbR^d;\bbR^d)$ as $\veps \to 0$.
\end{proof}

\subsection{Analysis of the Intermediate Approximation}

We now find  $\Upsilon^{s,\veps}$ explicitly. Using polar coordinates
\begin{equation*}
	\begin{split}
		\Upsilon^{s,\veps}(\bx) &= \intdm{\bbR^d}{ \left( (2 \pi |\bsxi|)^{2s} \left( 2s \frac{\bsxi \otimes \bsxi}{|\bsxi|^2} + \bI \right) \rme^{-2 \pi \veps |\bsxi|}  \right) \rme^{2 \pi \imath \bx \cdot \bsxi} }{\bsxi} \\
		&= \int_0^{\infty} \intdm{\bbS^{d-1}}{ (2 \pi r)^{2s} \left( 2s (\bsomega \otimes \bsomega) + \bI \right) \rme^{-2 \pi \veps r}  \rme^{2 \pi r \imath \bx \cdot \bsomega} r^{d-1}}{\sigma(\bsomega)} \, \rmd r\,.
	\end{split}
\end{equation*}
Letting $\varrho = 2 \pi |\bx| r$ and changing variables,
\begin{equation*}
	\begin{split}
		\Upsilon^{s,\veps}(\bx)
		&= \frac{1}{|\bx|^{d+2s}} \frac{1}{(2 \pi)^d} \int_0^{\infty} \varrho^{d+2s-1} \rme^{- \frac{\veps \varrho}{|\bx|}}  \intdm{\bbS^{d-1}}{ \left( 2s (\bsomega \otimes \bsomega) + \bI \right)   \rme^{-\imath \varrho  \frac{\bx}{|\bx|} \cdot \bsomega} }{\sigma(\bsomega)} \, \rmd \varrho\,.
	\end{split}
\end{equation*}
Clearly for any $\bx \in \bbR^d$ and for any $\bR \in \cO(d)$ we have $\Upsilon^{s,\veps} (\bR \bx ) = \bR^T \Upsilon^{s,\veps} (\bx ) \bR$. So it is easy to see that the functions $\mathrm{tr}\Upsilon^{s,\veps}(\bx)$ and $\Upsilon^{s,\veps}(\bx)[\frac{\bx}{|\bx|},\frac{\bx}{|\bx|}]$ are radially symmetric and that the radial functions $\gamma_1$ and $\gamma_2$ defined by
\begin{equation*}
	\gamma_1(|\bx|) := \frac{\mathrm{tr} \Upsilon^{s,\veps}(\bx) - \Upsilon^{s,\veps}(\bx)[\frac{\bx}{|\bx|},\frac{\bx}{|\bx|}] }{(d-1)} \quad \text{ and } \quad \gamma_2(|\bx|) := \frac{d \Upsilon^{s,\veps}(\bx)[\frac{\bx}{|\bx|},\frac{\bx}{|\bx|}] - \mathrm{tr} \Upsilon^{s,\veps}(\bx)}{(d-1)}
\end{equation*}
give the following formula for $\Upsilon^{s,\veps}$:
\begin{equation*}
	\Upsilon^{s,\veps}(\bx) = \gamma_1(|\bx|) \bI + \gamma_2(|\bx|) \frac{\bx \otimes \bx}{|\bx|^2}\,.
\end{equation*}
By Lemma \ref{lma:RadialKernel1}
\begin{equation*}
	\begin{split}
		\mathrm{tr} \Upsilon^{s,\veps}(\bx) &= \frac{1}{|\bx|^{d+2s}} \frac{1}{(2 \pi)^d} \int_0^{\infty} \varrho^{d+2s-1} \rme^{- \frac{\veps \varrho}{|\bx|}}  \intdm{\bbS^{d-1}}{ \left( 2s + d \right)   \rme^{-\imath \varrho  \frac{\bx}{|\bx|} \cdot \bsomega} }{\sigma(\bsomega)} \, \rmd \varrho \\
		&= \frac{1}{|\bx|^{d+2s}} \frac{d+2s}{(2 \pi)^d} \int_0^{\infty} \varrho^{d+2s-1} \rme^{- \frac{\veps \varrho}{|\bx|}} (2 \pi)^{d/2} \frac{J_{\nu}(\varrho)}{\varrho^{\nu}} \, \rmd \varrho
	\end{split}
\end{equation*}
where $\nu := \frac{d-2}{2}$.
By Lemma \ref{lma:RadialKernel1} and Lemma \ref{lma:RadialKernel2}
\begin{equation*}
	\begin{split}
		&\Upsilon^{s,\veps}(\bx) \left[\frac{\bx}{|\bx|}, \frac{\bx}{|\bx|} \right] \\
		&= \frac{1}{|\bx|^{d+2s}} \frac{1}{(2 \pi)^d} \int_0^{\infty} \varrho^{d+2s-1} \rme^{- \frac{\veps \varrho}{|\bx|}}  \intdm{\bbS^{d-1}}{ \left( 2s \left( \bsomega \cdot \frac{\bx}{|\bx|}  \right)^2 + 1 \right)   \rme^{-\imath \varrho  \frac{\bx}{|\bx|} \cdot \bsomega} }{\sigma(\bsomega)} \, \rmd \varrho \\
		&= \frac{1}{|\bx|^{d+2s}} \frac{1}{(2 \pi)^{d/2}} \int_0^{\infty} \varrho^{\frac{d}{2}+2s} \rme^{- \frac{\veps \varrho}{|\bx|}} \left[  (2s+1) J_{\nu}(\varrho) - 2s(d-1) \frac{J_{\nu+1}(\varrho)}{\varrho} \right]  \, \rmd \varrho\,.
	\end{split}
\end{equation*}
Therefore,
\begin{equation*}
	\gamma_1(|\bx|) = \frac{1}{|\bx|^{d+2s}} \frac{1}{(2 \pi)^{d/2}} \int_0^{\infty} \varrho^{\frac{d}{2}+2s} \rme^{- \frac{\veps \varrho}{|\bx|}} \left[ J_{\nu}(\varrho) + 2s \frac{J_{\nu+1}(\varrho)}{\varrho} \right]  \, \rmd \varrho
\end{equation*}
and
\begin{equation*}
	\gamma_2(|\bx|) = \frac{1}{|\bx|^{d+2s}} \frac{1}{(2 \pi)^{d/2}} \int_0^{\infty} \varrho^{\frac{d}{2}+2s} \rme^{- \frac{\veps \varrho}{|\bx|}} \left[ 2s J_{\nu}(\varrho) - 2sd \frac{J_{\nu+1}(\varrho)}{\varrho} \right]  \, \rmd \varrho\,.
\end{equation*}
We arrive at the following theorem:

\begin{theorem}\label{thm:PoissonInt:Formula}
	\begin{equation}\label{eq:PoissonInt:Formula}
		\Upsilon^{s,\veps}(\bx) = \frac{1}{|\bx|^{d+2s}} \left( \psi_1 \left( \frac{\veps}{|\bx|} \right)  \bI + \psi_2 \left( \frac{\veps}{|\bx|} \right) \frac{\bx \otimes \bx}{|\bx|^2} \right)\,,
	\end{equation}
	where the functions $\psi_1$ and $\psi_2$ are defined on $[0,\infty)$ as
	\begin{equation*}
		\begin{gathered}
			\psi_1(r) := \frac{1}{(2 \pi)^{d/2}} \int_0^{\infty} t^{\frac{d}{2}+2s} \rme^{- rt} \left[ J_{\nu}(t) + 2s \frac{J_{\nu+1}(t)}{t} \right]  \, \rmd t\,,  \\
			\psi_2(r) := \frac{1}{(2 \pi)^{d/2}} \int_0^{\infty} t^{\frac{d}{2}+2s} \rme^{- rt} \left[ 2s J_{\nu}(t) - 2sd \frac{J_{\nu+1}(t)}{t} \right]  \, \rmd t\,,
		\end{gathered}
	\end{equation*}
	and where $\nu = \frac{d-2}{2}$. For each fixed $\veps$, $\Upsilon^{s,\veps} \in L^1(\bbR^d) \cap L^{\infty}(\bbR^d)$, and 	$\intdm{\bbR^d}{\Upsilon^{s,\veps}(\bx)}{\bx}$ is the zero matrix.
\end{theorem}

\begin{proof}
	By the bound \eqref{eq:BesselFxn:Bound} the integrals defining $\psi_1$ and $\psi_2$ converge absolutely for every $r > 0$, with the estimate 
	\begin{equation}\label{eq:PoissonInt:Estimate1}
		|\psi_i(r)| \leq C \int_0^{\infty} t^{\frac{d}{2}+2s} t^{\nu} \rme^{- rt} \, \rmd t = C \int_0^{\infty} t^{d+2s} \rme^{- rt} \, \frac{\rmd t}{t} = \frac{C}{r^{d+2s}} \int_0^{\infty} \tau^{d+2s} \rme^{-\tau} \, \frac{\rmd \tau}{\tau} = \frac{C}{r^{d+2s}}\,.
	\end{equation}
	for $i \in \{1,2\}$.
	It follows easily that $\Upsilon^{s,\veps} \in L^{\infty}(\bbR^d)$ (which we already knew). To see that $\Upsilon^{s,\veps} \in L^1(\bbR^d)$, we first note that by \eqref{eq:BesselFxnInt:Bounded} there exists $r_0(d,s) > 0$ such that
	\begin{equation}\label{eq:PoissonIntBoundedAtInfty}
		|\psi_i(r)| \leq C_0(d,s) \qquad \text{ for all } r \in (0,r_0)\,, \qquad i \in \{1,2\}\,,
	\end{equation}
	Then using \eqref{eq:PoissonInt:Estimate1} we have that for any fixed $\veps > 0$ and some $R(\veps,r_0) > 0$ large
	\begin{equation*}
		\Vnorm{\Upsilon^{s,\veps}}_{L^1(\bbR^d)} = \int_{B(0,R)} |\Upsilon^{s,\veps}| \, \rmd \bx + \int_{B(0,R)^c} |\Upsilon^{s,\veps}| \, \rmd \bx \leq C R^d + 2 C_0 \int_{B(0,R)^c} \frac{1}{|\bx|^{d+2s}} \, \rmd \bx < \infty\,.
	\end{equation*}
	Therefore $\Upsilon^{s,\veps} \in L^1(\bbR^d)$, and so
	\begin{equation*}
		\intdm{\bbR^d}{\Upsilon^{s,\veps}(\bx)}{\bx} = \wh{\Upsilon}_{s,\veps}({\bf 0}) = {\bf 0} \otimes {\bf 0}\,.
	\end{equation*}
\end{proof}

By Theorem \ref{thm:PoissonInt:Formula} we can now write
\begin{equation}\label{eq:PoissonApproxForm2}
	\bbG^s_{\veps} \bu(\bx) = \intdm{\bbR^d}{ \Upsilon^{s,\veps}(\bz) (\bu(\bx+\bz) - \bu(\bx)) }{\bz}\,.
\end{equation}

\begin{lemma}
	Let $\alpha \in (0,1)$ be arbitrary. Then for every $r \in [0,\infty)$
	\begin{equation}\label{eq:Decay:psi1:MainEst}
		|\psi_1(r)| \leq C r^{\alpha}
	\end{equation}
	and
	\begin{equation}\label{eq:Decay:psi2:MainEst}
		|\psi_2(r)+\kappa_{d,s}| \leq C r^{\alpha}\,,
	\end{equation}
	where $\kappa_{d,s}$ is defined in \eqref{eq:DefnOfNormalizingConstants}. The constant $C$ depends only on $d$, $s$ and $\alpha$.
\end{lemma}

\begin{proof}
	Using \eqref{eq:PoissonInt:Estimate1} the result follows trivially for $r$ away from $0$.
	Throughout the proof, $\nu = \frac{d-2}{2}$.
	We start with $\psi_1$.
	\begin{equation*}
		\begin{split}
			\psi_1(r) &= \frac{1}{(2 \pi)^{d/2}} \int_0^{\infty} t^{\frac{d}{2}+2s} \rme^{- rt}  J_{\nu}(t)  \, \rmd t \\
			&\qquad + \frac{2s}{(2 \pi)^{d/2}} \int_0^{\infty} t^{\frac{d}{2}+2s-1} \rme^{- rt} J_{\nu+1}(t) \, \rmd t := I(r)+II(r)\,.
		\end{split}
	\end{equation*}
	Preforming an integration by parts and using \eqref{eq:BesselFxn:Derivative} with $m=1$,
	\begin{equation}\label{eq:Expression}
		I(r) = \frac{1}{(2 \pi)^{d/2}} \left( t^{\frac{d}{2}+2s} \rme^{-rt} J_{\nu+1}(t) \Bigg|_{t = 0}^{\infty} - \int_0^{\infty} (2s t^{2s-1} - r t^{2s})t^{\frac{d}{2}} \rme^{- rt} J_{\nu+1}(t) \, \rmd t  \right)\,.
	\end{equation}
	For fixed $r$ we see that from \eqref{eq:BesselFxn:IntRep} the first term is bounded by $C t^{d+2s} \rme^{-rt}$, which converges to $0$ as $t \to 0$ and $t \to \infty$. Thus the boundary term is $0$, and using the expression \eqref{eq:Expression} to simplify $I+II$ gives
	\begin{equation}\label{eq:Psi1Formula}
		\psi_1(r) = \frac{r}{(2 \pi)^{d/2}} \int_0^{\infty} t^{\frac{d}{2}+2s} \rme^{- rt} J_{\nu+1}(t) \, \rmd t\,.
	\end{equation}
	Therefore by \eqref{eq:BesselFxnInt:Bounded}
	\begin{equation}\label{eq:Decay:Psi1}
		|\psi_1(r)| \leq Cr \leq Cr^{\alpha} \qquad \text{ for any } \alpha \in (0,1) \text{ and for any } r \in (0,r_0)\,,
	\end{equation}
	where $r_0$ depends only on $d$ and $s$.
	So \eqref{eq:Decay:psi1:MainEst} is proved.
	
	Now we treat $\psi_2$. Adding and subtracting the proper term,
	\begin{equation*}
		\psi_2(r) = 2s \psi_1(r) - \frac{2s(d+2s)}{(2 \pi)^{d/2}}  \int_0^{\infty} t^{\frac{d}{2}+2s-1} \rme^{- rt} J_{\nu+1}(t) \, \rmd t\,.
	\end{equation*}
	Therefore by \eqref{eq:Decay:Psi1} the only thing we need to show to prove \eqref{eq:Decay:psi2:MainEst} is that there exists $r_0(d,s) > 0$ such that
	\begin{equation}\label{eq:Decay:FinalEstimate}
		\left| \kappa_{d,s} - \frac{2s(d+2s)}{(2 \pi)^{d/2}}  \int_0^{\infty} t^{\frac{d}{2}+2s-1} \rme^{- rt} J_{\nu+1}(t) \, \rmd t \right| \leq C r^{\alpha}\,, \qquad r \in [0,r_0)\,.
	\end{equation}
	We need to treat two different cases. 
	
	\underline{Case 1:} $s < 1/2$. Then we can make use of \eqref{eq:BesselFxnInt:IntRep} to get 
	\begin{equation*}
		\int_0^{\infty} t^{\frac{d}{2}+2s-1} \rme^{- rt} J_{\nu+1}(t) \, \rmd t = \frac{\Gamma(d+2s)}{2^{d/2} \Gamma (\frac{d+2s+1}{2} ) \Gamma(\frac{1-2s}{2}) } \int_0^1 \frac{t^{\frac{d+2s-1}{2} } (1-t)^{\frac{-2s-1}{2} }}{(r^2+t)^{\frac{d+2s}{2}}} \, \rmd t\,.
	\end{equation*}
	The limit as $r \to 0$ clearly exists, and by \eqref{eq:BetaFunction}-\eqref{eq:Item5}
	\begin{equation*}
		\lim\limits_{r \to 0}\int_0^1 \frac{t^{\frac{d+2s-1}{2} } (1-t)^{\frac{-2s-1}{2} }}{(r^2+t)^{\frac{d+2s}{2}}} \, \rmd t = \frac{\sqrt{\pi} \Gamma \left( \frac{1-2s}{2} \right) }{  \Gamma(1-s)}\,.
	\end{equation*}
	Referring back to \eqref{eq:Decay:FinalEstimate}, we now see that the function $\wt{v} : [0,\infty) \to \bbR$ defined as
	\begin{equation*}
		\wt{v}(r) := \frac{2s(d+2s)}{(2 \pi)^{d/2}}  \frac{\Gamma(d+2s)}{2^{d/2} \Gamma (\frac{d+2s+1}{2} ) \Gamma(\frac{1-2s}{2}) } \int_0^1 \frac{t^{\frac{d+2s-1}{2} } (1-t)^{\frac{-2s-1}{2} }}{(r^2+t)^{\frac{d+2s}{2}}} \, \rmd t
	\end{equation*}
	has the property that
	\begin{equation*}
		\wt{v}(0) = \frac{2s \sqrt{\pi} (d+2s) \Gamma(d+2s) }{2^d \pi^{d/2} \Gamma(\frac{d+2s+1}{2}) \Gamma(1-s)}\,.
	\end{equation*}
	We use \eqref{eq:GammaFunction} to arrive at the identity
	\begin{equation*}
		\wt{v}(0) = \kappa_{d,s}\,.
	\end{equation*}
	Therefore \eqref{eq:Decay:FinalEstimate} can be recast as
	\begin{equation*}
		|\wt{v}(r) - \wt{v}(0)| \leq Cr^{\alpha} \qquad \forall\,\, r \in (0,r_0)\,.
	\end{equation*}
	By the mean value theorem, for any $\alpha \in (0,1)$
	\begin{equation*}
		\begin{split}
			|\wt{v}(r) - \wt{v}(0)| &\leq C  \int_0^1 t^{\frac{d+2s-1}{2} } (1-t)^{\frac{-2s-1}{2} } \left| \frac{1}{(r^2+t)^{\frac{d+2s}{2}}} - \frac{1}{t^{\frac{d+2s}{2}}} \right| \, \rmd t \\
			&\leq C \int_0^1 t^{\frac{d+2s-1}{2} } (1-t)^{\frac{-2s-1}{2} }  \int_0^r \frac{\sigma}{(\sigma^2+t)^{\frac{d+2s}{2}+1}} \, \rmd \sigma \, \rmd t \\
			&\leq C \int_0^1 t^{\frac{d+2s-1}{2} } (1-t)^{\frac{-2s-1}{2} }  \int_0^r \frac{\sigma}{(\sigma^2)^{ 1 - \frac{\alpha}{2} }(t)^{\frac{d+2s}{2}+ \frac{\alpha}{2} }} \, \rmd \sigma \, \rmd t \\
			&= C \int_0^1 t^{ -\frac{1+\alpha}{2} } (1-t)^{\frac{-2s-1}{2} }  \, \rmd t  \int_0^r \sigma^{ \alpha-1 } \, \rmd \sigma = Cr^{\alpha }\,.
		\end{split}
	\end{equation*}
	Thus \eqref{eq:Decay:FinalEstimate}, and thus \eqref{eq:Decay:psi2:MainEst}, is proved in the case $0 < s  < 1/2$.
	
	\underline{Case 2:} $1/2 \leq s < 1$. We cannot use \eqref{eq:BesselFxnInt:IntRep} directly since the parameter relation $\nu > \mu - 1$ in \eqref{eq:BesselFxnInt:IntRep} is not satisfied, so we integrate by parts first: 
	\begin{equation*}
		\begin{split}
			\int_0^{\infty} & t^{\frac{d}{2}+2s-1} \rme^{- rt} J_{\nu+1}(t) \, \rmd t  \\
			&= \bigg( t^{\frac{d}{2} + 2s - 1 } \rme^{-rt} J_{\nu+2}(t) \bigg) \Bigg|_{t = 0}^{\infty} - \int_0^{\infty}  \big( (2s-2) t^{-1} - r \big) \rme^{-rt} t^{\frac{d}{2} + 2s - 1} J_{\nu + 2}(t) \, \rmd t\,.
		\end{split}
	\end{equation*}
	From the bound \eqref{eq:BesselFxn:Bound} we see that the boundary term is bounded by $t^{d+2s} \rme^{-rt}$ and thus after evaluation equals $0$. 
	Further, by \eqref{eq:BesselFxnInt:Bounded} there exists $r_0(d,s) > 0$ such that
	\begin{equation*}
		\left| r \int_0^{\infty}  \rme^{-rt} t^{\frac{d}{2} + 2s - 1} J_{\nu + 2}(t) \, \rmd t \right| \leq Cr \qquad \text{ for } r \in (0,r_0)\,,
	\end{equation*}
	so to prove \eqref{eq:Decay:FinalEstimate} it suffices to show that 
	\begin{equation}\label{eq:Decay:FinalEstimate2}
		\left| \kappa_{d,s} - \frac{2s(d+2s) (2-2s)}{(2 \pi)^{d/2}} \int_0^{\infty} t^{\frac{d}{2}+2s-2} \rme^{- rt} J_{\nu+2}(t) \, \rmd t \right| \leq C r^{\alpha}\,, \qquad r \in [0,r_0)\,.
	\end{equation}
	The parameter relation $\nu > \mu - 1$ in \eqref{eq:BesselFxnInt:IntRep} is satisfied by this integral, so we have
	\begin{equation*}
		\int_0^{\infty} t^{\frac{d}{2}+2s-2} \rme^{- rt} J_{\nu+2}(t) \, \rmd t = \frac{\Gamma(d+2s)}{2^{\frac{d}{2}+1 } \Gamma (\frac{d+2s+1}{2} ) \Gamma(\frac{3-2s}{2}) } \int_0^1 \frac{t^{\frac{d+2s-1}{2} } (1-t)^{\frac{1-2s}{2} }}{(r^2+t)^{\frac{d+2s}{2}}} \, \rmd t\,.
	\end{equation*}
	From here the proof follows the same reasoning as in Case 1.
\end{proof}

We will also need the following characterizations of functions in $\scL^{2s,p}$. For $\beta \in \bbR \setminus \{0\}$, Let $G_{2\beta}$ be the Bessel kernel; i.e.
\begin{equation*}
	G_{2\beta}(\bx) = \cF^{-1} \Big( (1 + 4 \pi^2 |\bsxi|^2)^{-\beta} \Big)(\bx)\,.
\end{equation*}
Then it is known \cite{Stein, grafakos2009modern} that $\scL^{2s,p}(\bbR^d;\bbR^d)$ can be written as
\begin{equation*}
	\scL^{2s,p}(\bbR^d;\bbR^d) := \left\{ \bu \in L^p(\bbR^d;\bbR^d) \, : \, G_{-2s} \ast \bu \in L^p(\bbR^d;\bbR^d) \right\}
\end{equation*}
with norm $\Vnorm{\bu}_{\scL^{2s,p}(\bbR^d)} = \Vnorm{G_{-2s} \ast \bu}_{L^{p}(\bbR^d)}$.

\begin{lemma}\label{lma:ModContOfBesselSpaces}
	Let $0 < \beta \leq 1$, $\bz \in \bbR^d$. Then for every $\bu \in C^{\infty}_c(\bbR^d;\bbR^d)$
	\begin{equation*}
		\left( \intdm{\bbR^d}{ |\bu(\bx+\bz) + \bu(\bx-\bz) - 2\bu(\bx)|^p }{\bx} \right)^{1/p} \leq C(d,\beta) \Vnorm{\bu}_{\scL^{2\beta,p}(\bbR^d)} |\bz|^{2\beta}\,.
	\end{equation*}
\end{lemma}
\begin{proof}
	Write $\bu(\bx) = G_{2\beta} \ast (G_{-2\beta} \ast \bu)(\bx)$. Then by Young's inequality
	\begin{equation*}
		\begin{split}
			\left( \intdm{\bbR^d}{ |\bu(\bx+\bz) + \bu(\bx-\bz) - 2\bu(\bx)|^p }{\bx} \right)^{1/p} &\leq \intdm{\bbR^d}{ |G_{2\beta}(\bx+\bz) + G_{2\beta}(\bx-\bz) - 2G_{2\beta}(\bx)| }{\bx} \\
			&\qquad \times \Vnorm{ (G_{-2\beta} \ast \bu)}_{L^p(\bbR^d)}\,.
		\end{split}
	\end{equation*}
	To conclude the proof we just need the inequality
	\begin{equation*}
		\intdm{\bbR^d}{ |G_{2\beta}(\bx+\bz) + G_{2\beta}(\bx-\bz) - 2G_{2\beta}(\bx)| }{\bx} \leq C |\bz|^{2\beta}\,.
	\end{equation*}
	This is shown for all $\beta > 0$ in \cite[Chapter V, Section 5.4]{Stein}.
	The proof is complete.
\end{proof}

\begin{proof}[proof of Lemma \ref{lma:BesselSpaces:MainLemma}]
	
	We use \eqref{eq:PoissonInt:Formula}, \eqref{eq:PoissonApproxForm2}, and a change of variables in \eqref{eq:BesselSpaces:VepsOp2} to write
	\begin{equation*}
		\begin{split}
			\bbG^s_{\veps} \bu(\bx) - \bbF^s_{\veps} \bu(\bx)
			&= \intdm{|\bz| < \veps}{\Upsilon^{s,\veps}(\bz) (\bu(\bx+\bz)-\bu(\bx)) }{\bz} \\
			&\quad + \intdm{|\bz| \geq \veps}{\left( \Upsilon^{s,\veps}(\bz) + \frac{\kappa_{d,s}}{|\bz|^{d+2s}} \frac{\bz \otimes \bz}{|\bz|^2} \right) (\bu(\bx+\bz)-\bu(\bx)) }{\bz} \\
			&= \intdm{|\bz| < \veps}{\left[ \psi_1 \left( \frac{\veps}{|\bz|} \right) \bI + \psi_2 \left( \frac{\veps}{|\bz|} \right)  \frac{\bz \otimes \bz}{|\bz|^2} \right] \frac{(\bu(\bx+\bz)-\bu(\bx))}{|\bz|^{d+2s}} }{\bz} \\
			&\quad + \intdm{|\bz| \geq \veps}{\left[ \psi_1 \left( \frac{\veps}{|\bz|} \right) \bI + \left( \psi_2 \left( \frac{\veps}{|\bz|} \right) + \kappa_{d,s} \right) \frac{\bz \otimes \bz}{|\bz|^2} \right] \frac{(\bu(\bx+\bz)-\bu(\bx))}{|\bz|^{d+2s}} }{\bz}\,.
		\end{split}
	\end{equation*}
	Note that by splitting the integrals and change of variables we can replace the difference $\bu(\bx+\bz)-\bu(\bx)$ with $\frac{1}{2} ( \bu(\bx+\bz)+\bu(\bx-\bz)-2\bu(\bx) )$. So
	\begin{equation*}
		\begin{split}
			\bbG^s_{\veps} &\bu(\bx) - \bbF^s_{\veps} \bu(\bx) \\
			&= \intdm{|\bz| < \veps}{\left[ \psi_1 \left( \frac{\veps}{|\bz|} \right) \bI + \psi_2 \left( \frac{\veps}{|\bz|} \right)  \frac{\bz \otimes \bz}{|\bz|^2} \right] \frac{( \bu(\bx+\bz)+\bu(\bx-\bz)-2\bu(\bx) )}{2|\bz|^{d+2s}} }{\bz} \\
			&\quad + \intdm{|\bz| \geq \veps}{\left[ \psi_1 \left( \frac{\veps}{|\bz|} \right) \bI + \left( \psi_2 \left( \frac{\veps}{|\bz|} \right) + \kappa_{d,s} \right) \frac{\bz \otimes \bz}{|\bz|^2} \right] \frac{( \bu(\bx+\bz)+\bu(\bx-\bz)-2\bu(\bx) )}{2|\bz|^{d+2s}} }{\bz} \\
			&:= I + II\,.
		\end{split}
	\end{equation*}
	By \eqref{eq:PoissonInt:Estimate1} and Lemma \ref{lma:ModContOfBesselSpaces}, the $L^p(\bbR^d)$ norm of $I$ is majorized by
	\begin{equation*}
		C \Vnorm{\bu}_{\scL^{2\beta,p}(\bbR^d)}  \intdm{|\bz|<\veps}{\frac{ |\bz|^{2\beta} }{\veps^{d+2s}}}{\bz}\,, \qquad \beta \in (0,1)\,.
	\end{equation*}
	Choosing $\beta = s$, we see that the $L^p(\bbR^d;\bbR^d)$ norm of $I$ is bounded by $C \Vnorm{\bu}_{\scL^{2s,p}(\bbR^d)} $, and choosing $\beta > s$ reveals that $I$ converges in $L^p(\bbR^d)$ norm as $\veps \to 0$ (Recall $\bu$ is smooth).
	
	From the estimates \eqref{eq:Decay:psi1:MainEst} and \eqref{eq:Decay:psi2:MainEst} and Lemma \ref{lma:ModContOfBesselSpaces}, we see that for any $\alpha \in (0,1)$ the $L^p(\bbR^d;\bbR^d)$ norm of $II$ is majorized by
	\begin{equation*}
		\begin{split}
			&C \intdm{|\bz| \geq \veps}{ \frac{\veps^{\alpha}}{ |\bz|^{d+2s+\alpha } } \Vnorm{\bu(\cdot + \bz) + \bu(\cdot - \bz) - 2\bu(\cdot)}_{L^p(\bbR^d)} }{\bz} \\
			&\leq C  \Vnorm{\bu}_{\scL^{2s,p}(\bbR^d)} \intdm{|\bz| \geq \veps}{ \frac{\veps^{\alpha} |\bz|^{2s} }{ |\bz|^{d+2s+\alpha} }  }{\bz} = C \Vnorm{\bu}_{\scL^{2s,p}(\bbR^d)}\,.
		\end{split}
	\end{equation*}
	Thus the first part of \Cref{lma:BesselSpaces:MainLemma} follows. On the other hand, for any $\delta > \veps $ fixed,
	\begin{equation*}
		\begin{split}
			\Vnorm{II}_{L^p(\bbR^d)} &\leq C  \Vnorm{\bu}_{\scL^{2s,p}(\bbR^d)} \intdm{|\bz| \geq \delta}{ \frac{\veps^{\alpha} |\bz|^{2s} }{ |\bz|^{d+2s+\alpha} }  }{\bz} + C \intdm{\veps \leq |\bz| \leq \delta}{ \Vnorm{\bu}_{\scL^{2,p}(\bbR^d)} \frac{\veps^{\alpha}|\bz|^2}{|\bz|^{d+2s+\alpha }} }{\bz} \\
			&= C(\delta) \veps^{\alpha} \Vnorm{\bu}_{\scL^{2s,p}(\bbR^d)} + C \Vnorm{\bu}_{\scL^{2,p}(\bbR^d)} \int_{\veps}^{\delta} \frac{\veps^{\alpha}}{r^{ 2s+ \alpha -1 } } \, \rmd r\,.
		\end{split}
	\end{equation*}
	Setting $\alpha = 1-s$ implies that the right-hand side is $O(\veps^{1-s})$, and therefore that the $L^p(\bbR^d)$ norm of $II$ converges to $0$ as $\veps \to 0$ whenever $s \in (0,1)$. Thus the second part of \Cref{lma:BesselSpaces:MainLemma} follows.
\end{proof}

\begin{proof}[proof of \Cref{thm:BesselSpaces:MainThm2}]
	The results of Theorem \ref{thm:BesselSpaces:MainThm2} were shown to hold for any $\bv \in C^{\infty}_c(\bbR^d;\bbR^d)$ in Theorem \ref{thm:BesselSpaces:MainThm3}. Let $\{\bu_n \} \subset C^{\infty}_c(\bbR^d;\bbR^d)$ be a sequence that converges to $\bu$ in $\scL^{2s,p}(\bbR^d;\bbR^d)$. Then for each fixed $\veps > 0$ $\bbF^s_{\veps} \bu_n \to \bbF^s_{\veps} \bu$ in $L^p(\bbR^d;\bbR^d)$ as $n \to \infty$ by Young's inequality.
	Since the estimate
	\begin{equation*}
		\sup_{\veps > 0} \Vnorm{\bbF^s_{\veps} \bu_n}_{L^p(\bbR^d)} \leq C(d,s) \Vnorm{\bu_n}_{\cL^{2s,p}(\bbR^d)}
	\end{equation*}
	holds for every $n$, we see by fixing each $\veps$ and then taking $n \to \infty$ that $\bu$ satisfies the same estimate.
	Further, for any $\veps_1$ and $\veps_2$
	\begin{equation*}
		\begin{split}
			\Vnorm{\bbF^s_{\veps_1} \bu - \bbF^s_{\veps_2} \bu}_{L^p(\bbR^d)} &\leq \Vnorm{\bbF^s_{\veps_1} \bu - \bbF^s_{\veps_1} \bu_n}_{L^p(\bbR^d)} + \Vnorm{\bbF^s_{\veps_2} \bu - \bbF^s_{\veps_2} \bu_n}_{L^p(\bbR^d)} + \Vnorm{\bbF^s_{\veps_1} \bu_n - \bbF^s_{\veps_2} \bu_n}_{L^p(\bbR^d)} \\
			&\leq 2C \Vnorm{\bu_n-\bu}_{\cL^{2s,p}(\bbR^d)} + \Vnorm{\bbF^s_{\veps_1} \bu_n - \bbF^s_{\veps_2} \bu_n}_{L^p(\bbR^d)}\,.
		\end{split}
	\end{equation*}
	The right-hand side can be made arbitrarily small by taking $n$ large and then taking $\veps_1$, $\veps_2$ to $0$. Thus $\bbF^s_{\veps} \bu$ converges in $L^p(\bbR^d;\bbR^d)$ as $\veps \to 0$ to a limit $\bv_s$ and $\bv_s$ satisfies \eqref{eq:MainBesselEstimate}.
\end{proof}

\section{An $\bbL$-Harmonic Extension System}\label{sec:ExtensionProblem}

In this section we identify $\bbL^s$ as the Dirichlet-to-Neumann map for a degenerate elliptic system.
Throughout this section differential operators like $\grad$ and $\div$ will be taken only with respect to the $\bx$ variable unless otherwise stated. 

\subsection{Extension Problem}

For $\bU : \bbR^d \times [0,\infty) \to \bbR^d$ and for $\bu : \bbR^d \to \bbR^d$ consider the problem
\begin{equation}\label{eq:Approach1:MainExtProb}
	\begin{cases}
		\p_{tt} \bU(\bx,t) + \frac{1-2s}{t} \p_t \bU(\bx,t) - \bbL \bU(\bx,t) = 0\,, \\
		\bU(\bx,0) = \bu(\bx)\,.
	\end{cases}
\end{equation}
Note here that the derivatives in $\bbL$ are applied only in $\bx$.
It will turn out that for a suitable class of boundary data, \eqref{eq:Approach1:MainExtProb} has a unique solution satisfying
$$
\lim\limits_{t \to 0} (-t^{1-2s} \p_t \bU(\bx,t)) =  \left[ \frac{2 \Gamma(1-s)}{2^{2s} \Gamma(s)} \right] \bbL^s \bu(\bx)\,.
$$

Following the treatment of fractional powers of more general scalar-valued operators (see for instance \cite{stinga2010extension, arendt2018fractional}) our candidate for a solution is the Poisson integral
\begin{equation}\label{eq:PoissonIntegralExtProb}
	\bU(\bx,t) = (\bP(\cdot,t) \ast \bu)(\bx)\,,
\end{equation}
where
\begin{equation*}
	\bP(\bx,t) := \frac{t^{2s}}{2^{2s} \Gamma(s)} \int_0^{\infty} \bW(\bx,r) \rme^{- \frac{t^2}{4r}} r^{-s} \frac{\rmd r}{r}\,,
\end{equation*}
and where the matrix field $\bW : \bbR^d \times (0,\infty) \to \bbR^{d \times d}$ is the heat kernel associated to $\bbL$; that is, formally $\p_t \bW(\bx-\bx_0,t-t_0) + \bbL \bW(\bx-\bx_0,t-t_0) = \delta_{ \{ \bx=\bx_0 \} } \delta_{ \{ t=t_0 \} }$. 
We will see in the sequel that $\bP$ exists and has an explicit expression, and in turn that the Poisson integral \eqref{eq:PoissonIntegralExtProb} solves \eqref{eq:Approach1:MainExtProb} for suitable functions $\bu$.

The heat kernel $\bW$ has an explicit formula
\begin{equation}\label{eq:HeatKernelLameDefn}
	\bW(\bx,t) := H(\bx,\mu t) \bI + \int_{\mu t}^{(2 \mu + \lambda) t} \grad^2 H(\bx,\sigma) \, \rmd \sigma\,,
\end{equation}
where $H$ is the classical heat kernel associated to $-\Delta$ given by
\begin{equation}\label{eq:HeatKernelDefn}
	H(\bx,t) := \frac{1}{(4 \pi t)^{d/2}} \rme^{-|\bx|^2 / 4t}\,, \qquad \bx \in \bbR^d\,, \quad t > 0\,.
\end{equation}
The desired properties of $\bW$ can be seen either by a direct computation or by the Fourier transform formula in \Cref{lma:FourierHeatKernel}; see also \cite{shen1991boundary}.

\subsection{Poisson Kernel}

Thanks to the formula for $\bW$ we can compute the exact form of the Poisson kernel $\bP$.

\begin{theorem}
	The matrix-valued function $\bP : \bbR^d \times (0,\infty) \to \bbR^{d \times d}$ given by
	\begin{equation}\label{eq:PoissonKernel1}
		\bP(\bx,t) = \frac{t^{2s}}{2^{2s} \Gamma(s)} \int_0^{\infty} \bW(\bx,r) \rme^{- \frac{t^2}{4r}} r^{-s} \frac{\rmd r}{r}
	\end{equation}
	has the closed-form expression
	\begin{equation}\label{eq:PoissonKernel2}
		\begin{split}
			\bP(\bx,t) &:= \mu^s \frac{\Gamma(\frac{d}{2}+s)}{\pi^{\frac{d}{2}} \Gamma(s) } \frac{t^{2s}}{(|\bx|^2 + \mu t^2)^{\frac{d+2s}{2}}} \bI \\
			&\quad -   \frac{\Gamma(\frac{d}{2}+s)}{2 \pi^{\frac{d}{2}} \Gamma(s) } t^{2s} \int_{\mu}^{2 \mu + \lambda} \frac{\sigma^{s-1}}{(|\bx|^2 + \sigma t^2)^{\frac{d+2s}{2}}} \, \rmd \sigma \bI \\
			&\quad + (d+2s) \frac{\Gamma(\frac{d}{2}+s)}{2 \pi^{\frac{d}{2}} \Gamma(s) } t^{2s} \int_{\mu}^{2 \mu + \lambda} \frac{\sigma^{s-1}}{(|\bx|^2 + \sigma t^2)^{\frac{d+2s+2}{2}}} \, \rmd \sigma (\bx \otimes \bx)\,.
		\end{split}
	\end{equation}
\end{theorem}

\begin{proof}
	We begin with finding the explicit expression for $\bW(\bx,r)$, which is
	\begin{equation*}
		\begin{split}
			\bW(\bx,r) &= H(\bx, \mu r ) \bI + \int_{\mu r}^{(2 \mu + \lambda) r} \grad^2 H(\bx,\sigma) \, \rmd \sigma \\
			&= \frac{1}{(4 \pi \mu r)^{d/2}} \rme^{-|\bx|^2/4 \mu r} \bI - \int_{\mu r}^{(2 \mu + \lambda) r} \frac{1}{(4 \pi \sigma)^{d/2}}  \frac{1}{2 \sigma} \rme^{-|\bx|^2/4 \sigma} \, \rmd \sigma \bI \\
			&\qquad + \int_{\mu r}^{(2 \mu + \lambda) r} \frac{1}{(4 \pi \sigma)^{d/2}}  \frac{1}{4 \sigma^2} \rme^{-|\bx|^2/4 \sigma} \, \rmd \sigma (\bx \otimes \bx) \\
			&:= A_1(\bx,r) \bI - A_2(\bx,r) \bI + A_3(\bx,r) (\bx \otimes \bx)\,. 
		\end{split}
	\end{equation*}
	By the coordinate change $\rho = \frac{\sigma}{r}$ we can write
	\begin{equation}\label{eq:HeatKernelCoordChange}
		\begin{split}
			A_2(\bx,r) &= \int_{\mu}^{2 \mu + \lambda} \frac{r}{(4 \pi r \rho)^{d/2}}  \frac{1}{2 \rho r} \rme^{-|\bx|^2/(4 r \rho)} \, \rmd \rho\,, \\
			A_3(\bx,r) &= \int_{\mu}^{2 \mu + \lambda} \frac{r}{(4 \pi r \rho)^{d/2}}  \frac{1}{4 \rho^2 r^2} \rme^{-|\bx|^2/(4 r \rho)} \, \rmd \rho\,.
		\end{split}
	\end{equation}
	We use linearity of the integral defining $\bP$ to separate and simplify each term separately, and combining each will give us the result.
	We define each piece of $\bP$ as follows:
	\begin{equation*}
		\begin{split}
			\bP(\bx,t) &= \frac{t^{2s}}{2^{2s} \Gamma(s)} \int_0^{\infty} A_1(\bx,r) \rme^{- \frac{t^2}{4r}} r^{-s} \frac{\rmd r}{r} \bI -  \frac{t^{2s}}{2^{2s} \Gamma(s)} \int_0^{\infty} A_2(\bx,r) \rme^{- \frac{t^2}{4r}} r^{-s} \frac{\rmd r}{r} \bI \\
			& \qquad + \frac{t^{2s}}{2^{2s} \Gamma(s)} \int_0^{\infty} A_3(\bx,r) \rme^{- \frac{t^2}{4r}} r^{-s} \frac{\rmd r}{r} (\bx \otimes \bx) \\
			&:= p_1(\bx,t) \bI + p_2(\bx,t)\bI + p_3(\bx,t) (\bx \otimes \bx)\,.
		\end{split}
	\end{equation*}
	First, by the coordinate change $\rho = \frac{|\bx|^2 + \mu t^2}{4 \mu r}$
	\begin{equation*}
		\begin{split}
			p_1(\bx,t) &= \frac{t^{2s}}{2^{2s} \Gamma(s)} \int_0^{\infty} \frac{1}{(4 \pi \mu r)^{d/2}} \rme^{- \frac{|\bx|^2 + \mu t^2}{4 \mu r}} r^{-s} \frac{\rmd r}{r} \\
			&= \frac{\mu^s }{ \pi^{d/2} \Gamma(s)} \frac{t^{2s}}{(|\bx|^2 + \mu t^2)^{\frac{d+2s}{2}}} \int_0^{\infty} \rho^{d/2 + s} \rme^{-\rho} \frac{\rmd \rho}{\rho} = \frac{\mu^s \Gamma(\frac{d}{2} + s )}{ \pi^{d/2} \Gamma(s)} \frac{t^{2s}}{(|\bx|^2 + \mu t^2)^{\frac{d+2s}{2}}}\,.
		\end{split}
	\end{equation*}
	The integrand defining $p_2$ is nonnegative,
	so by Fubini's theorem
	\begin{equation*}
		\begin{split}
			p_2(\bx,t) &= \frac{t^{2s}}{2^{2s} \Gamma(s)} \int_0^{\infty} \int_{\mu}^{2 \mu + \lambda} \frac{r}{(4 \pi r \rho)^{d/2}} \frac{1}{2 \rho r} \rme^{- \frac{|\bx|^2 + \rho t^2}{4 \rho r}} r^{-s} \, \rmd \rho  \, \frac{\rmd r}{r} \\
			&= \frac{t^{2s}}{2^{2s} \pi^{d/2} \Gamma(s)} \int_{\mu}^{2 \mu + \lambda} \frac{1}{2 \rho }  \int_0^{\infty}  \frac{1}{(4 \rho r)^{d/2}} \rme^{- \frac{|\bx|^2 + \rho t^2}{4 \rho r}} r^{-s}  \, \frac{\rmd r}{r} \, \rmd \rho \,.
		\end{split}
	\end{equation*}
	Using the coordinate change $\sigma = \frac{|\bx|^2+ \rho t^2}{4 \rho r}$ in the inner integral in the last expression,
	\begin{equation*}
		\begin{split}
			p_2(\bx,t) &= \frac{t^{2s}}{2 \pi^{d/2} \Gamma(s)} \int_{\mu}^{2 \mu + \lambda} \rho^{s-1} \frac{1}{(|\bx|^2 + \rho t^2)^{\frac{d+2s}{2}}} \int_0^{\infty} \sigma^{d/2+s} \rme^{-\sigma} \, \frac{\rmd \sigma}{\sigma} \, \rmd \rho \\
			&= \frac{\Gamma(\frac{d}{2} + s)}{2 \pi^{d/2} \Gamma(s)} \int_{\mu}^{2 \mu + \lambda} \frac{t^{2s} \rho^{s-1}}{(|\bx|^2 + \rho t^2)^{\frac{d+2s}{2}}} \, \rmd \rho\,.
		\end{split}
	\end{equation*}
	The computation for $p_3$ is very similar. Use Fubini's theorem and the coordinate change $\sigma = \frac{|\bx|^2+ \rho t^2}{4 \rho r}$ in the inner integral to get
	\begin{equation*}
		p_3(\bx,t) = \frac{\Gamma(\frac{d}{2} + s+1)}{\pi^{d/2} \Gamma(s)} \int_{\mu}^{2 \mu + \lambda} \frac{t^{2s} \rho^{s-1}}{(|\bx|^2 + \rho t^2)^{\frac{d+2s+2}{2}}} \, \rmd \rho\,.
	\end{equation*}
	We conclude the calculation for $p_3$ with the identity $\Gamma(\frac{d}{2} +s + 1) = \frac{d+2s}{2} \Gamma(\frac{d}{2}+s)$.
\end{proof}

\subsection{Properties of the Poisson Kernel}
The Fourier transform of $\bP$ can be stated in terms of the formula \eqref{eq:PoissonKernel1}. By \Cref{lma:FourierHeatKernel}
\begin{equation}\label{eq:PoissonKernelFourierTransform}
	\begin{split}
		\widehat{\bP}(\bsxi,t) &:= \cF(\bP(\cdot,t))(\bsxi)  \\
		&= \frac{t^{2s}}{2^{2s} \Gamma(s)} \int_0^{\infty} \cF(\bW)(\bsxi,r) \rme^{-\frac{t^2}{4r}} r^{-s} \frac{\rmd r}{r} \\
		&= \frac{t^{2s}}{2^{2s} \Gamma(s)} \int_0^{\infty} \rme^{-r \bM(\bsxi)} \rme^{-\frac{t^2}{4r}} r^{-s} \frac{\rmd r}{r} \\
		&=
		\frac{t^{2s}}{2^{2s} \Gamma(s)} \int_0^{\infty} \left[ \rme^{-4 \pi^2 \mu |\bsxi|^2 r} \bI + \big( \rme^{-4 \pi^2 (2\mu + \lambda) |\bsxi|^2 r} - \rme^{-4 \pi^2 \mu |\bsxi|^2 r} \big) \frac{\bsxi \otimes \bsxi}{|\bsxi|^2} \right] \rme^{-\frac{t^2}{4r}} r^{-s} \frac{\rmd r}{r}\,.
	\end{split}
\end{equation}

Various properties satisfied for classical Possion kernels can now be easily seen using the Fourier transform formula.
We collect these into a theorem.
\begin{theorem}\label{thm:PoissonKernelProperties}
	The Poisson kernel \eqref{eq:PoissonKernel2} satisfies the following:
	\begin{enumerate}
		\item[i)] There exists $C > 0$ depending only on $d$, $s$, $\mu$ and $\lambda$ such that
		\begin{equation}\label{eq:NaturalBoundOnPoisson}
			|\bP(\bx,1)| \leq \frac{C}{(1+|\bx|^2)^{\frac{d+2s}{2}}} \qquad \text{ for } \bx \in \bbR^d\,.
		\end{equation}
		\item[ii)] For all $t > 0$
		\begin{equation}\label{eq:PoissonKernel:Dilation}
			\bP(\bx,t) = t^{-d} \, \bP \left(\frac{\bx}{t} , 1 \right)\,.
		\end{equation}
		\item[iii)] For all $t > 0$
		\begin{equation*}
			\int_{\bbR^d} \bP(\bx,t) \, \rmd \bx = \bI\,. 
		\end{equation*}

		\item[iv)] $\bP(\bx,t) \in C^{\infty} ( \bbR^d \times (0,\infty) )$, and $\bP(\bx,t)$ satisfies \eqref{eq:Approach1:MainExtProb} pointwise for every $\bx \in \bbR^d$ and $t > 0$.
		
		\item[v)] Defining $\wt{\bP}(\bx,t) := t^{1-2s} \bP(\bx,t)$, we have $\wt{\bP} \in C^{\infty}( \bbR^d \times [0,\infty) \setminus B_{d+1}(0,\veps) )$ for every $\veps > 0$, where $B_{d+1}(0,\veps) := \{ (\bx,t) \in \bbR^d \times [0,\infty) \, : \, |\bx|^2 + t^2 < \veps^2 \}$.
		
		\item[vi)] For every $\beta > 0$, $\wt{\bP}(\beta \bx,\beta t) = \beta^{-d-2s+1} \wt{\bP}(\bx,t)$. Consequently, for every multi-index $\alpha \in \bbN_0^{d+1}$ (that is, $\alpha$ ranges over both $\bx$ and $t$) there exists $C = C(d,s,\mu,\lambda,\alpha)$ such that
		$$
		|\p^{\alpha} \wt{\bP}(\bx,t)| \leq C (|\bx|^2 + t^2)^{-d-2s+1-|\alpha|} \text{ for all } (\bx,t) \in \big( \bbR^d \times [0,\infty) \big) \setminus \{ ({\bf 0},0) \} \,.
		$$
		
		\item [vii)] Suppose $\bu$ belongs to the weighted Lebesgue space $L^1_{s} (\bbR^d)$ defined in \eqref{eq:DefnOfWeightedLebSpace}. Then $(\bP(\cdot,t) \ast \bu)(\bx) \in C^{\infty}(\bbR^d \times (0,\infty))$, and for any $\beta > 0$ there exists a constant $C_{\beta} > 0$ depending only on $d$, $s$, $\mu$, $\lambda$ and $\beta$ such that for any $\bx_0 \in \bbR^d$
		\begin{equation}\label{eq:NontanConv1}
			\sup_{\substack{|\bx-\bx_0| < \beta t \\ t > 0}} |(\bP(\cdot,t) \ast \bu)(\bx)|  \leq C_{\beta} \cM  \big( |\bu| \big) (\bx_0)\,,
		\end{equation}
		where $\cM$ is the Hardy-Littlewood maximal operator.
		Further, for every Lebesgue point $\bx_0 \in \bbR^d$ of $\bu$ and for any $\beta > 0$
		\begin{equation}\label{eq:NontanConv2}
			\lim\limits_{\substack{(\bx,t) \to (\bx_0,0) \\ |\bx-\bx_0| < \beta t}} (\bP(\cdot,t) \ast \bu)(\bx) = \bu(\bx_0)\,.
		\end{equation}
	\end{enumerate}
\end{theorem}

\begin{proof}
	Items i) and ii) follow by direct inspection of \eqref{eq:PoissonKernel2}. 
	Item iii) follows from noting that $\rme^{-t \bM(\bsxi)}|_{t = 0} = \bI$, and so $\intdm{\bbR^d}{\bP(\bx,t)}{\bx} = \widehat{\bP}({\bf 0},t) = \frac{t^{2s}}{2^{2s} \Gamma(s)} \int_0^{\infty} \rme^{-\frac{t^2}{4r}} r^{-s} \frac{\rmd r}{r} \, \bI = \bI \,.$ (Alternatively, item iii) can be proved by using the formula for $\bW$).
	
	For item iv), since $\wh{\bP}(\bsxi,1)$ is rapidly decreasing in $\bsxi$, it follows that $\bP(\bx,1) \in C^{\infty}(\bbR^d)$. Then $\bP(\bx,t) \in C^{\infty} ( \bbR^d \times (0,\infty) )$ since \eqref{eq:PoissonKernel:Dilation} holds for all $t > 0$. That $\bP$ is a solution of the PDE in \eqref{eq:Approach1:MainExtProb} follows by direct computation.
	
	To prove item v) we just need to show that for any $|\bx| > 0$, all derivatives of $\wt{\bP}$ in $t$ extend continuously to the point $(\bx,0)$. In this context, $\wt{\bP}$ is the product of $t$ and a function with $t$-profile comparable to $(1+t^2)^{-\beta}$ for fixed $\beta > 0$. Therefore the result follows by induction using the Leibniz product rule.
	Then item vi) follows easily from the formula for $\wt{\bP}$ and item v).
	
	Finally, to prove the nontangential convergence in item viii) we first note that for any $\by \in \bbR^d$ with $|\by| < \beta t$
	\begin{equation}\label{eq:PoissonComparison}
		\bP(\bx-\by,t) \leq C \frac{t^{2s}}{ (|\bx-\by|^2 +t^2)^{\frac{d+2s}{2}} } \leq C'_{\beta} \frac{t^{2s}}{ (|\bx|^2 +t^2)^{\frac{d+2s}{2}} }\,,
	\end{equation}
	where $C_{\beta}'$ depends only on $d$, $s$, $\mu$, $\lambda$ and $\beta$.
	Therefore, \eqref{eq:NontanConv1} is established via
	\begin{equation*}
		\begin{split}
		\sup_{\substack{|\bx-\bx_0| < \beta t \\ t > 0}} |(\bP(\cdot,t) \ast \bu)(\bx)| &= \sup_{\substack{|\bx-\bx_0| < \beta t \\ t > 0}} \int_{\bbR^d} \bP(\bx-\by,t) \bu(\by) \, \rmd \by \\
		&= \sup_{\substack{|\bx| < \beta t \\ t > 0}} \int_{\bbR^d} \bP(\bx_0-\bx-\by,t) \bu(\by) \, \rmd \by \\
		&\leq  C_{\beta}' \sup_{t > 0} \int_{\bbR^d} \frac{t^{2s}}{(|\bx_0-\by|^2 +t^2)^{\frac{d+2s}{2}} } |\bu(\by)| \, \rmd \by \\
		&\leq C_{\beta} \cM \big( |\bu| \big)(\bx_0)\,,
		\end{split}
	\end{equation*}
	where in the last line we used \cite[Chapter III, Theorem 2]{Stein}. 
	
	To prove \eqref{eq:NontanConv2} let $\bx_0$ be a Lebesgue point of $\bu$, i.e. for any $\veps > 0$ there exists a $\delta > 0$ such that
	\begin{equation*}
		\fint_{B(0,r)} |\bu(\bx_0-\by) - \bu(\bx_0) | \, \rmd \by < \veps
	\end{equation*}
	for all $r < \delta$ ($\fint$ denotes the integral average). 
	Therefore
	\begin{equation}\label{eq:NontanConv2:Pf1}
		\sup_{r < \delta} \fint_{B(0,r)} |\bu(\bx_0-\by) - \bu(\bx_0) | \, \rmd \by = \cM \left( |\bu(\bx_0-\cdot) - \bu(\bx_0)| \chi_{B(0,\delta)}(\cdot) \right)(0) < \veps\,.
	\end{equation}
	Now assume $|\bx-\bx_0| + t \to 0$, and assume $|\bx| < \beta t$. So by item iii) and \eqref{eq:PoissonComparison} 
	\begin{equation*}
		\begin{split}
		|(\bP(\cdot,t) \ast \bu)(\bx_0-\bx) - \bu(\bx_0)| &= \left| \int_{\bbR^d} \bP(\by-\bx,t) (\bu(\bx_0-\by)-\bu(\bx_0)) \, \rmd \by \right| \\
		&\leq C_{\beta}' \int_{\bbR^d} \frac{t^{2s}}{(|\by|^2 +t^2)^{\frac{d+2s}{2}} } |\bu(\bx_0-\by)-\bu(\bx_0)| \, \rmd \by \\
		&= C_{\beta}' \left( \int_{B(0,\delta)} \cdots + \int_{\bbR^d \setminus B(0,\delta) } \cdots \right)
		\end{split}
	\end{equation*}
	An application of \cite[Chapter III, Theorem 2]{Stein} and \eqref{eq:NontanConv2:Pf1} reveals that the first integral is majorized by $C_{\beta} \veps$. The second integral is majorized by
	\begin{equation*}
		\int_{\bbR^d \setminus B(0,\delta)} t^{2s} (\delta^{-2}+1)^{\frac{d+2s}{2}} \frac{|\bu(\bx_0-\by)-\bu(\bx_0)|}{(1+|\by|^2)^{\frac{d+2s}{2}}} \, \rmd \by\,, 
	\end{equation*}
	which converges to $0$ as $t \to 0$ since $\bu \in L^1_{s}(\bbR^d)$. Therefore \eqref{eq:NontanConv2} is proved since $\veps >0$ is arbitrary.
\end{proof}

We can characterize $\wh{\bP}(\bsxi,t)$ more precisely by using special functions. 
\begin{theorem}
	For $s \in (0,1)$, for all $\bsxi \in \bbR^d$ and $t > 0$
	\begin{equation}\label{eq:PoissonKernelFourierBessel}
		\wh{\bP}(\bsxi,t) =  \cK_s(2 \pi \sqrt{\mu} |\bsxi| t) \bI 
		+ \Big( \cK_s(2 \pi \sqrt{2\mu+\lambda} |\bsxi| t) - \cK_s(2 \pi \sqrt{\mu} |\bsxi| t) \Big) 
		\frac{\bsxi \otimes \bsxi}{|\bsxi|^2}\,.
	\end{equation}
	Here, $\cK_s(a) := \frac{2^{1-s}}{\Gamma(s)}a^s K_s(a)$ for $a > 0$,
	where $K_s$ denotes the modified Bessel function of the second kind of order $s$.
	When $s=1/2$, the Poisson kernel satisfies the semigroup property
	\begin{equation*}
		\bP(\cdot,t_1) \ast \bP(\cdot,t_2) = \bP(\cdot,t_1 + t_2)\,, \qquad \text{ for } t_1, t_2 > 0\,.
	\end{equation*}
\end{theorem}

\begin{proof}
	Using \cite[Equation 3.478.4]{gradshteyn2014table} and \cite[Equation 9.6.6]{abramowitz1988handbook}, $K_s$ can be expressed as
	\begin{equation*}
		K_s(a) = \frac{1}{2} \left( \frac{a}{2} \right)^s \int_0^{\infty} \rme^{-r} \rme^{-\frac{a^2}{4r}} r^{-s} \frac{\rmd r}{r}\,, \qquad a > 0\,.
	\end{equation*}
	Therefore setting $a = 2 \pi b |\bsxi| t$ for $b > 0$, using the change of variables $\rho = \frac{a^2 r}{t^2}$ gives
	\begin{equation*}
		\frac{t^{2s}}{2^{2s} \Gamma(s)} \int_0^{\infty} \rme^{-4 \pi^2 b^2 |\bsxi|^2 r} \rme^{-\frac{t^2}{4r}} r^{-s} \frac{\rmd r}{r} = \frac{2^{1-s}}{\Gamma(s)} (2 \pi b |\bsxi| t)^s \, K_s(2 \pi b |\bsxi| t)\,.
	\end{equation*}
	By applying this formula in \eqref{eq:PoissonKernelFourierTransform} with $b = \sqrt{\mu}$ and $b = \sqrt{2 \mu + \lambda}$, we obtain \eqref{eq:PoissonKernelFourierBessel}.
	
	When $s = 1/2$, the function $\cK_s$ reduces to an exponential
	\begin{equation*}
		\cK_{1/2}(a) = \sqrt{\frac{2}{\pi}} a^{1/2} \left( \sqrt{\frac{\pi}{2}} a^{-1/2} \rme^{-a} \right) = \rme^{-a}\,.
	\end{equation*}
	Thus the semigroup property of $\bP$ can be seen directly by multiplying the Fourier transforms:
	\begin{equation*}
		\wh{\bP}(\bsxi,t_1) \wh{\bP}(\bsxi,t_2) = \wh{\bP}(\bsxi,t_1+t_2)\,.
	\end{equation*}
\end{proof}

\subsection{Existence and Uniqueness}

We state the well-posedness of a weak formulation of the problem \eqref{eq:Approach1:MainExtProb}. We denote the upper-half plane $\bbR^d \times (0,\infty)$ as $\bbR^{d+1}_+$, and define the weighted Lebesgue and weighted Sobolev spaces
\begin{equation*}
	\begin{split}
		L^2_s(\bbR^{d+1}_+; \bbR^d) &:= \left\{ \bU : \bbR^{d+1}_+ \to \bbR^d \, : \, \Vnorm{\bU}_{L^2_s(\bbR^{d+1}_+)} := \int_0^{\infty} \int_{\bbR^d} |\bU(\bx,t)|^2 t^{1-2s} \, \rmd \bx \, \rmd t < \infty \right\}\,, \\
		H^1_{s}(\bbR^{d+1}_+; \bbR^d) &:= \left\{ \bU \in L^2_{s}(\bbR^{d+1}_+; \bbR^d) \, : \, |\grad_{(\bx,t)} \bU| \in L^2_{s}(\bbR^{d+1}_+; \bbR^d) \right\}
	\end{split}
\end{equation*}
with the natural norms.
Consider the weighted Dirichlet-type energy
\begin{equation*}
	\cD(\bU) := \frac{1}{2} \int_0^{\infty} \intdm{\bbR^d}{ t^{1-2s} \big( |\p_t \bU|^2  + \mu |\grad \bU|^2 + (\mu+\lambda) |\div \bU|^2 \big)  }{\bx} \, \rmd t\,.
\end{equation*}
It is clear that
\begin{equation}\label{eq:Extprob:Coercivity}
	\frac{1}{2} \min \{1,\mu, 2 \mu + \lambda \} \Vnorm{\grad_{(\bx,t)} \bU}_{L^2_{s}(\bbR^{d+1}_+)} \leq \cD(\bU) \leq C(\mu,\lambda) \Vnorm{\grad_{(\bx,t)}\bU}_{L^2_{s}(\bbR^{d+1}_+)}\,.
\end{equation}
It is well-known \cite{nekvinda1993characterization} that every function $\bu \in \scL^{s,2}(\bbR^d;\bbR^d)$ can be represented as the trace of a function $\bU \in L^2_{s,loc}(\bbR^{d+1}_+; \bbR^d)$ with  $\grad_{(\bx,t)} \bU \in L^2_s(\bbR^{d+1}_+; \bbR^d)$. 
This gives a natural setting to the inhomogeneous Dirichlet problem, since any $\bU$ with $\cD(\bU) < \infty$ will belong to $H^1_s(\bbR^{d+1}_+;\bbR^d)$ and thus will have a well-defined trace $\bu$ in $\scL^{s,2}(\bbR^d;\bbR^d)$.
Define the homogeneous space
\begin{equation*}
	H^1_{s,0}(\bbR^{d+1}_+;\bbR^d) := \{ \bU \in H^1_s(\bbR^{d+1}_+;\bbR^d) \, : \, \bU(\bx,0) ={\bf 0} \text{ in the trace sense} \}\,.
\end{equation*}

\begin{theorem}
	For any $\bu \in \scL^{s,2}(\bbR^d;\bbR^d)$ there exists a unique weak solution $\bU \in H^1_s(\bbR^{d+1}_+;\bbR^d)$ of \eqref{eq:Approach1:MainExtProb}.
	To be precise, $\bU$ satisfies
	\begin{equation}\label{eq:ExtProb:WeakForm1}
		\begin{split}
			\cB(\bU,\bsPhi) := \int_0^{\infty} \intdm{\bbR^d}{ t^{1-2s} \big( \Vint{\p_t \bU, \p_t \bsPhi }  + \mu \Vint{\grad \bU, \grad \bsPhi } + (\mu+\lambda) (\div \bU) (\div \bsPhi) \big)  }{\bx} \, \rmd t = 0
		\end{split}
	\end{equation}
	for all $\bsPhi \in H^1_{s,0}(\bbR^{d+1}_+;\bbR^d)$
	and
	\begin{equation}\label{eq:ExtProb:WeakForm2}
		\bU \big|_{ \{  t = 0 \} } = \bu \text{ in the trace sense. }
	\end{equation}
	In addition, $\bU \in C^{\infty}(\bbR^{d+1}_+)$, and so $\bU$ satisfies \eqref{eq:Approach1:MainExtProb} pointwise in $\bbR^{d+1}_+$. The boundary condition in \eqref{eq:Approach1:MainExtProb} is satisfied for all Lebesgue points of $\bu$; that is, for almost every $\bx \in \bbR^d$.
\end{theorem}

\begin{proof}
	For existence: The Poisson integral
	\begin{equation*}
		\bU(\bx,t) := \int_{\bbR^d} \bP(\bx-\by,t) \bu(\by) \, \rmd \by
	\end{equation*}
	is well-defined by \eqref{eq:NaturalBoundOnPoisson}-\eqref{eq:PoissonKernel:Dilation}and satisfies \eqref{eq:ExtProb:WeakForm1}-\eqref{eq:ExtProb:WeakForm2}. Regularity and pointwise properties follows from the properties of $\bP$ described in \Cref{thm:PoissonKernelProperties}.
	
	For uniqueness: if $\bU$ solves \eqref{eq:Approach1:MainExtProb} with homogeneous Dirichlet data, then we can use $\bsPhi = \bU$ in \eqref{eq:ExtProb:WeakForm1}. Then using \eqref{eq:Extprob:Coercivity}
	\begin{equation*}
		C \int_0^{\infty} \intdm{\bbR^d}{ t^{1-2s} |\grad_{(\bx,t)} \bU|^2 }{\bx} \, \rmd t \leq \cD(\bU) = 0\,.
	\end{equation*}
	Then a weighted Hardy inequality \cite{nevcas1962methode} (see also \cite{lehrback2014weighted}) implies that
	\begin{equation*}
		\int_0^{\infty} \int_{\bbR^d} t^{-1-2s} |\bU(\bx,t)|^2 \, \rmd \bx \, \rmd t \leq C \int_0^{\infty} \intdm{\bbR^d}{ t^{1-2s} |\grad_{(\bx,t)} \bU(\bx,t)|^2 }{\bx} = 0\,.
	\end{equation*}
	Therefore the only solution to \eqref{eq:ExtProb:WeakForm1}-\eqref{eq:ExtProb:WeakForm2} with $\bu = {\bf 0}$ is $\bU = { \bf 0 }$.
\end{proof}

\subsection{Dirichlet-to-Neumann Map}

Consider the bilinear form $\cB$ defined in \eqref{eq:ExtProb:WeakForm1} on the weighted space $H^1_s(\bbR^{d+1}_+;\bbR^d)$.
For $\bU$ and $\bsPhi \in C^{\infty}(\bbR^{d+1}_+;\bbR^d)$, an application of the divergence theorem on the set $\bbR^d \times [\veps,\infty)$ and then using dominated convergence as $\veps \to 0$ gives
\begin{equation}\label{eq:EnergyEquivalence1}
	\cB(\bU,\bsPhi) = \int_{\bbR^d} \lim\limits_{t \to 0} \left( - t^{1-2s} \p_t \bU(\bx,t) \right) \cdot \bsPhi(\bx,0) \, \rmd \bx\,.
\end{equation}
Therefore, the Dirichlet-to-Neumann map $\bsLambda : \scL^{s,2}(\bbR^d;\bbR^d) \to \scL^{-s,2}(\bbR^d;\bbR^d)$ for the equation \eqref{eq:Approach1:MainExtProb} is
\begin{equation*}
	\bsLambda \bu(\bx) := -\lim\limits_{t \to 0} t^{1-2s} \p_t \bU(\bx,t)\,.
\end{equation*}
The next theorems reveal that $\bsLambda$ is given by a constant multiple of $\bbL^s$.

\begin{theorem}\label{thm:DtNMap} 
Let $\bu \in \scL^{s,2}(\bbR^d;\bbR^d)$ and let $\bU$ be the unique smooth solution of  \eqref{eq:ExtProb:WeakForm1}-\eqref{eq:ExtProb:WeakForm2}. Then
	\begin{equation}\label{eq:DtNMap1}
		- \lim\limits_{t \to 0} t^{1-2s} \p_t \bU(\bx,t)  = -2s \lim\limits_{t \to 0} \frac{\bU(\bx,t)-\bU(\bx,0)}{t^{2s}} = \frac{2 \Gamma(1-s)}{2^{2s} \Gamma(s)} \bbL^s \bu(\bx)\,,
	\end{equation}
	where the convergence is in $\scL^{-s,2}(\bbR^d;\bbR^d)$, i.e.\ in the distributional sense.
\end{theorem}

\begin{proof}
	Note that \eqref{eq:Approach1:MainExtProb} holds pointwise for the solution $\bU$ defined via the Poisson kernel, so $\bU(\bx,0) = \bu(\bx)$ for almost every $\bx \in \bbR^d$. 
	Since $\intdm{\bbR^d}{\bP(\bx-\by,t)}{\by} = \bI$ for all $t$,
	\begin{equation}\label{eq:DtNMap:Proof1}
		2s \frac{\bU(\bx,t)-\bU(\bx,0)}{t^{2s}} = \frac{2s}{t^{2s}} \intdm{\bbR^d}{\bP(\bx-\by,t)(\bu(\by) - \bu(\bx)) }{\by}\,, \quad t > 0\,.
	\end{equation}
	Therefore, by the definitions of $\bP$, $c_{d,s}$ and $\kappa_{d,s}$,
	\begin{equation*}
		\begin{split}
			2s \frac{\bU(\bx,t)-\bU(\bx,0)}{t^{2s}} 
			&= \frac{2 \Gamma(1-s)}{2^{2s} \Gamma(s)} \Bigg[ \mu^s c_{d,s} \int_{\bbR^d} \frac{\bu(\by)-\bu(\bx)}{(|\bx-\by|^2 + \mu t^2)^{\frac{d+2s}{2}}} \, \rmd \by \\
			&\quad -  c_{d,s} \int_{\bbR^d} \int_{\mu}^{2 \mu + \lambda} \frac{\sigma^{s-1}}{2 (|\bx-\by|^2 + \sigma t^2)^{\frac{d+2s}{2}}} \, \rmd \sigma (\bu(\by) - \bu(\bx)) \, \rmd \by \\
			&\quad + \kappa_{d,s} \int_{\bbR^d} \int_{\mu}^{2 \mu + \lambda} \frac{\sigma^{s-1}}{2 (|\bx-\by|^2 + \sigma t^2)^{\frac{d+2s+2}{2}}} \, \rmd \sigma \\
			&\qquad \quad \times \left( (\bx-\by) \otimes (\bx-\by) \right) (\bu(\by) - \bu(\bx)) \, \rmd \by \Bigg]\,, \quad t > 0\,.
		\end{split}
	\end{equation*}
	Note that for each $t > 0$ all of these integrals are finite for $\bu \in \scL^{s,2}(\bbR^d;\bbR^d)$.
	
	Now, let $\bsvarphi \in \scL^{s,2}(\bbR^d;\bbR^d)$.
	Just as in \Cref{subsec:NonlocalWeakForm}, we obtain via nonlocal integration by parts
	\begin{equation*}
		\begin{split}
			\int_{\bbR^d} & \Vint{-2s \frac{\bU(\bx,t)-\bU(\bx,0)}{t^{2s}} , \bsvarphi(\bx) } \, \rmd \bx \\
			&= \frac{2 \Gamma(1-s)}{2^{2s} \Gamma(s)} \Bigg[ \frac{\mu^s c_{d,s}}{2} \int_{\bbR^d} \frac{(\bu(\bx)-\bu(\by)) \cdot ( \bsvarphi(\bx)-\bsvarphi(\by)) }{(|\bx-\by|^2 + \mu t^2)^{\frac{d+2s}{2}}} \, \rmd \by \\
			&\quad -  c_{d,s} \int_{\bbR^d} \int_{\mu}^{2 \mu + \lambda} \frac{\sigma^{s-1}}{4} \frac{(\bu(\bx) - \bu(\by)) \cdot ( \bsvarphi(\bx)-\bsvarphi(\by))}{ (|\bx-\by|^2 + \sigma t^2)^{\frac{d+2s}{2}}} \, \rmd \sigma \, \rmd \by \\
			&\quad + \kappa_{d,s} \int_{\bbR^d} \int_{\mu}^{2 \mu + \lambda} \frac{\sigma^{s-1}}{4} \frac{\left( \big( \bu(\bx) - \bu(\by) \big) \cdot (\bx-\by) \right) \left( \big( \bsvarphi(\bx) - \bsvarphi(\by) \big) \cdot (\bx-\by) \right) }{(|\bx-\by|^2 + \sigma t^2)^{\frac{d+2s+2}{2}}} \, \rmd \sigma \, \rmd \by \Bigg]
		\end{split}
	\end{equation*}
	for all $t > 0$.
	By H\"older's inequality the right-hand side is majorized by
	\begin{equation*}
		C(\mu,\lambda,d,s) [\bu]_{\dot{\scL}^{s,2}(\bbR^d)} [\bsvarphi]_{\dot{\scL}^{s,2}(\bbR^d)}\,,
	\end{equation*}
	where we have used the identification \eqref{eq:EquivalenceOfSobolevAndBessel}. We now see that by the dominated convergence theorem the right-hand side converges as $t \to 0$ to
	\begin{equation*}
		\begin{gathered}
			\frac{2 \Gamma(1-s)}{2^{2s} \Gamma(s)} \Bigg[ \frac{\mu^s c_{d,s}}{2} \int_{\bbR^d} \frac{(\bu(\bx)-\bu(\by)) \cdot ( \bsvarphi(\bx)-\bsvarphi(\by)) }{|\bx-\by|^{d+2s}} \, \rmd \by \\
			-  c_{d,s} \int_{\bbR^d} \left( \int_{\mu}^{2 \mu + \lambda} \frac{\sigma^{s-1}}{4} \, \rmd \sigma  \right) \frac{(\bu(\bx) - \bu(\by)) \cdot ( \bsvarphi(\bx)-\bsvarphi(\by))}{|\bx-\by|^{d+2s}} \, \rmd \by \\
			+ \kappa_{d,s} \int_{\bbR^d} \left( \int_{\mu}^{2 \mu + \lambda} \frac{\sigma^{s-1}}{4} \, \rmd \sigma  \right) \frac{\left( \big( \bu(\bx) - \bu(\by) \big) \cdot \frac{\bx-\by}{|\bx-\by|} \right) \left( \big( \bsvarphi(\bx) - \bsvarphi(\by) \big) \cdot \frac{\bx-\by}{|\bx-\by|}  \right) }{|\bx-\by|^{d+2s}} \, \rmd \by \Bigg]\,.
		\end{gathered}
	\end{equation*}
	Therefore by integrating both $\rmd \sigma$ integrals and then using \eqref{eq:NonlocalWeakForm}
	\begin{equation*}
		\lim\limits_{t \to 0} \int_{\bbR^d} \Vint{-2s \frac{\bU(\bx,t)-\bU(\bx,0)}{t^{2s}} , \bsvarphi(\bx) } \, \rmd \bx = \frac{2 \Gamma(1-s)}{2^{2s} \Gamma(s)} \Vint{\bbL^s \bu, \bsvarphi}\,,
	\end{equation*}
	which is the second equality in \eqref{eq:DtNMap1}.
	
	To prove the first equality, we note that $\intdm{\bbR^d}{\bP(\bx,t)}{\bx} = \bI$, and write
	\begin{equation*}
		\p_t \bU(\bx,t) = \intdm{\bbR^d}{\p_t \bP(\bx-\by,t) \bu(\by)}{\by} = \intdm{\bbR^d}{\p_t \bP(\bx-\by,t) (\bu(\by)-\bu(\bx))}{\by}\,.
	\end{equation*}
	By a direct computation,
	\begin{equation*}
		\p_t \bP(\bx,t) = 2s t^{2s-1} \left( \frac{1}{t^{2s}} \bP(\bx,t) \right) + t^{2s-1} \bR(\bx,t)\,,
	\end{equation*}
	where
	\begin{equation*}
		\begin{split}
			\bR(\bx,t) &:= \frac{\Gamma(\frac{d}{2}+s)}{\pi^{\frac{d}{2}} \Gamma(s) } \Bigg[ \mu^s  \frac{(d+2s)}{(|\bx|^2 + \mu t^2)^{\frac{d+2s}{2}}} \left( \frac{|\bx|^2}{|\bx|^2 + \sigma t^2} - 1 \right) \bI \\
			&\quad -  \frac{d+2s}{2} \int_{\mu}^{2 \mu + \lambda} \frac{\sigma^{s-1}}{(|\bx|^2 + \sigma t^2)^{\frac{d+2s}{2}}} \left( \frac{|\bx|^2}{|\bx|^2 + \sigma t^2} - 1\right) \, \rmd \sigma \bI \\
			&\quad + \frac{d+2s}{2} (d+2s+2) \int_{\mu}^{2 \mu + \lambda} \frac{\sigma^{s-1}}{(|\bx|^2 + \sigma t^2)^{\frac{d+2s+2}{2}}} \left( \frac{|\bx|^2}{|\bx|^2+\sigma t^2} - 1\right) \, \rmd \sigma (\bx \otimes \bx) \Bigg]\,.
		\end{split}
	\end{equation*}
	Therefore, 
	\begin{equation*}
		\begin{split}
		t^{1-2s} \p_t \bU(\bx,t) &= \frac{2s}{t^{2s}} \intdm{\bbR^d}{\bP(\bx-\by,t)(\bu(\by) - \bu(\bx)) }{\by} \\
			&\qquad + \intdm{\bbR^d}{\bR(\bx-\by,t)(\bu(\by) - \bu(\bx)) }{\by}\,.
		\end{split}
	\end{equation*}
	By \eqref{eq:DtNMap:Proof1} and the first part of the proof, it suffices to show that
	\begin{equation*}
		\int_{\bbR^d} \intdm{\bbR^d}{\bR(\bx-\by,t)(\bu(\by) - \bu(\bx)) }{\by} \, \bsvarphi(\bx) \, \rmd \bx  \to 0 \text{ as } t \to 0
	\end{equation*}
	for all $\bsvarphi \in \scL^{s,2}(\bbR^d;\bbR^d)$. But this follows by dominated convergence after writing this quantity in the same weak form used in the first part of the proof. Thus the second inequality in \eqref{eq:DtNMap1} is proved.
\end{proof}

If we assume that $\bu \in \scL^{2s,p}(\bbR^d;\bbR^d)$, then we can strengthen the result to norm convergence:
\begin{theorem}
	 Let $1< p < \infty$, let $\bu \in \scL^{2s,p}(\bbR^d;\bbR^d)$ and let $\bU$ be the unique smooth solution of  \eqref{eq:ExtProb:WeakForm1}-\eqref{eq:ExtProb:WeakForm2}. Then
	\begin{equation}\label{eq:DtNMap2}
		- \lim\limits_{t \to 0} t^{1-2s} \p_t \bU(\bx,t)  = -2s \lim\limits_{t \to 0} \frac{\bU(\bx,t)-\bU(\bx,0)}{t^{2s}} = \frac{2 \Gamma(1-s)}{2^{2s} \Gamma(s)} \bbL^s \bu(\bx)\,,
	\end{equation}
	where the convergence is in the strong topology of $L^p(\bbR^d;\bbR^d)$.
\end{theorem}
The proof strategy is very similar to the strategy used to prove 
\Cref{thm:BesselSpaces:MainThm}, replacing $\bbL^s_{\veps}$ in the proof with with $-\frac{2^{2s}\Gamma(s)}{2 \Gamma(1-s)} \veps^{1-2s} \p_t \bU(\bx,\veps)$.

\section{A Variational Problem Associated to $\bbL^s$}\label{sec:VariationalProblem}
As an application, we will use extension problem from the previous section to obtain the following equivalence of seminorms:

\begin{theorem}\label{thm:KornInequality}
	Let $s \in (0,1)$ and let $\bu \in \scL^{s,2}(\bbR^d;\bbR^d)$. Then 
	\begin{equation}\label{eq:KornInequality}
		\theta c_{d,s} [\bu]_{\scH^s(\bbR^d)}^2 \leq \cE^s(\bu,\bu) \leq C [\bu]_{\scH^s(\bbR^d)}^2\,,
	\end{equation}
	where the bilinear from $\cE^{s}(\bu,\bsvarphi)$ is defined in \Cref{subsec:NonlocalWeakForm}. The constant $C > 0$ depends only $d$, $s$, $\mu$ and $\lambda$, and the constant $\theta > 0$ depends only on $\mu$ and $\lambda$.
\end{theorem}

\begin{proof}
	The second inequality follows easily from \eqref{eq:BiFormCont} and \eqref{eq:EquivalenceOfSobolevAndBessel}, so it remains to prove the first.
	Define $\bU \in H^1_s(\bbR^{d+1}_+; \bbR^d)$ to be the unique smooth solution of \eqref{eq:ExtProb:WeakForm1}-\eqref{eq:ExtProb:WeakForm2} with boundary data $\bu$. Then by \eqref{eq:NonlocalWeakForm}, \Cref{thm:DtNMap}, and \eqref{eq:EnergyEquivalence1}, 
	\begin{equation*}
		\begin{split}
		\cE^s(\bu,\bu) &= \frac{2^{2s} \Gamma(s)}{2 \Gamma(1-s)} \cB(\bU,\bU) \\
		&= \frac{2^{2s} \Gamma(s)}{2 \Gamma(1-s)} \int_0^{\infty} \intdm{\bbR^d}{ t^{1-2s} \big( |\p_t \bU|^2  + \mu |\grad \bU|^2 + (\mu+\lambda) |\div \bU|^2 \big)  }{\bx} \, \rmd t\,.
		\end{split}
	\end{equation*}
	Then by the coercivity inequality \eqref{eq:Extprob:Coercivity}
	\begin{equation*}
		\cE^s(\bu,\bu) \geq C(\mu,\lambda) \frac{2^{2s} \Gamma(s)}{2 \Gamma(1-s)} \int_0^{\infty} \intdm{\bbR^d}{ t^{1-2s} |\grad_{(\bx,t)} \bU|^2 }{\bx} \, \rmd t\,.
	\end{equation*}
	Now, let $\bV(\bx,t)$ satisfy
	\begin{equation*}
		\bV(\bx,t) = \mathrm{argmin} \left( \int_0^{\infty} \intdm{\bbR^d}{ t^{1-2s} |\grad_{(\bx,t)} \wt{\bV}|^2 }{\bx} \, \rmd t \right)
	\end{equation*}
	over all functions $\wt{\bV} \in H^1_s(\bbR^{d+1}_+; \bbR^d)$ with boundary data $\bu$. Then by \cite[Section 3.2]{caffarelli2007extension} (see also \cite[Section 10]{garofalo2017fractional})
	\begin{equation*}
		 \int_0^{\infty} \intdm{\bbR^d}{ t^{1-2s} |\grad_{(\bx,t)} \bV|^2 }{\bx} \, \rmd t = \frac{2 \Gamma(1-s)}{2^{2s} \Gamma(s)} \Vnorm{(-\Delta)^{\frac{s}{2}} \bu }_{L^2(\bbR^d)}^2\,.
	\end{equation*}
	Therefore, since $\bV$ is a minimum, we have
	\begin{equation*}
		\cE^{s}(\bu,\bu) \geq C(\mu,\lambda) \Vnorm{(-\Delta)^{\frac{s}{2}} \bu }_{L^2(\bbR^d)}^2 = \theta c_{d,s} [\bu]_{\scH^s(\bbR^d)}^2\,,
	\end{equation*}
	which is the first inequality.
\end{proof}

The inequality just established is similar to fractional Korn-type inequalities investigated in \cite{KassmannMengeshaScott, Mengesha-HalfSpace}. An $L^p$ version is proved in \cite{MengeshaScott2018Korn}. In the context of the extension problem, this says that the Korn inequality continues to hold even after taking traces.
Thanks to the asymptotics for $c_{d,s}$ and $\kappa_{d,s}$ investigated in \cite{DNPV12} and to limiting theorems for nonlocal seminorms \cite{mengesha2015VariationalLimit}, we readily see that the classical Korn inequality \cite{nitsche1981korn} is recovered from \eqref{eq:KornInequality} as $s \to 1^-$.

We will use \Cref{thm:KornInequality} to analyze a nonlocal Dirichlet problem associated to $\bbL^s$:
\begin{equation}\label{eq:NonlocalDirichletProb}
	\begin{cases}
		\bbL^s \bu = \bff\,, & \text{ in } \Omega\,, \\
		\bu = {\bf 0}\,, & \text{ in } \bbR^d \setminus \Omega\,.
	\end{cases}
\end{equation}
The natural energy space for this problem is
\begin{equation*}
	\scL^{s,2}_{\Omega}(\bbR^d;\bbR^d) := \left\{ \bu \in \scL^{s,2}(\bbR^d;\bbR^d) \, : \, \bu \equiv {\bf 0} \text{ on } \bbR^d \setminus \Omega \right\}\,.
\end{equation*}
It is easily seen that $\scL^{s,2}_{\Omega}(\bbR^d;\bbR^d)$ is a Hilbert space with inner product inherited from the potential space $\scL^{s,2}(\bbR^d;\bbR^d)$. We denote its functional dual by $\scL^{-s,2}_{\Omega}(\bbR^d;\bbR^d)$.

Recall the bilinear from $\cE^{s}(\bu,\bsvarphi)$ defined in \Cref{subsec:NonlocalWeakForm}. For a given $\bff \in \scL^{-s,2}_{\Omega}(\bbR^d;\bbR^d)$, we say that $\bu \in \scL^{s,2}_{\Omega}(\bbR^d;\bbR^d)$ is a \textit{weak solution} of \eqref{eq:NonlocalDirichletProb} if 
\begin{equation}\label{eq:NonlocalDirichletProb:WeakForm}
	\cE^s(\bu,\bsvarphi) = \Vint{\bff, \bsvarphi} \text{ for all } \bsvarphi \in \scL^{s,2}_{\Omega}(\bbR^d;\bbR^d)\,.
\end{equation}

\begin{theorem}
	Let $s \in (0,1)$ and let $\Omega$ be a bounded domain in $\bbR^d$. For any $\bff \in \scL^{-s,2}_{\Omega}(\bbR^d;\bbR^d)$ there exists a unique $\bu \in \scL^{s,2}_{\Omega}(\bbR^d;\bbR^d)$ satisfying \eqref{eq:NonlocalDirichletProb:WeakForm} with the energy estimate 
	$$
	\Vnorm{\bu}_{\scL^{s,2}(\bbR^d)} \leq \Vnorm{\bff}_{\scL^{-s,2}_{\Omega}(\bbR^d)}\,.
	$$
\end{theorem}

\begin{proof}
	We follow the strategy for similar nonlocal Dirichlet problems treated in \cite{KassmannMengeshaScott, Felsinger}. The result will follow by the Lax-Milgram theorem if we show that the bilinear form $\cE^s$ is both continuous and coercive on $\scL^{s,2}_{\Omega}(\bbR^d;\bbR^d)$.
	By \eqref{eq:BiFormCont} the bilinear form $\cE^s$ is continuous on $\scL^{s,2}_{\Omega}(\bbR^d;\bbR^d)$.
	
	To establish coercivity we use the inequality of \Cref{thm:KornInequality} and \eqref{eq:EquivalenceOfSobolevAndBessel} to get that
	\begin{equation*}
		\cE^s(\bu,\bu) \geq \theta [\bu]_{\dot{\scL}^{s,2}(\bbR^d)}^2\,,
	\end{equation*}
	where $\theta$ depends only on $\mu$ and $\lambda$. From here, we use the nonlocal Poincar\'e inequality that follows easily from the fractional Sobolev inequality \cite[Theorem 1]{maz2002bourgain} and obtain
	\begin{equation*}
		\Vnorm{\bu}_{L^2(\bbR^d)}^2 \leq C(d,s) [\bu]_{\scH^s(\bbR^d)}^2 = C'(d,s) [\bu]_{\dot{\scL}^{s,2}(\bbR^d)}^2
	\end{equation*}
	for all $\bu \in \scL^{s,2}_{\Omega}(\bbR^d)$.
	Therefore,
	\begin{equation*}
		C(d,s,\mu,\lambda) \Vnorm{\bu}_{\scL^{s,2}(\bbR^d)} \leq  \cE^s(\bu,\bu)
	\end{equation*}
	for all $\bu \in \scL^{s,2}_{\Omega}(\bbR^d)$ and so coercivity is proved. Existence and uniqueness then follows by the Lax-Milgram theorem, and the energy estimate follows by choosing $\bsvarphi = \bu$ in \eqref{eq:NonlocalWeakForm}.
\end{proof}

We remark that the coercivity inequality can also be shown using the Fourier transform, and does not require \Cref{thm:KornInequality} at all. We use the extension system here because purely variational techniques may be useful in more general settings, for example, in situations when $\bu \in L^p$ for $p \neq 2$.

\appendix

\section{Fourier Transform Formulas}

\begin{lemma}\label{lma:FourierHeatKernel}
	Let $\bW$ be the heat kernel associated to the operator $\bbL$ defined in \eqref{eq:HeatKernelLameDefn}. Then for all $t > 0$,
	\begin{equation*}
		\cF(\bW)(\bsxi,t) = \rme^{-4 \pi^2 \mu |\bsxi|^2 t} \bI + \big( \rme^{-4 \pi^2 (2\mu + \lambda) |\bsxi|^2 t} - \rme^{-4 \pi^2 \mu |\bsxi|^2 t} \big) \frac{\bsxi \otimes \bsxi}{|\bsxi|^2} = \rme^{-t \bM(\bsxi)}\,.
	\end{equation*}
\end{lemma}
\begin{proof}
	Since $\cF(H)(\bsxi,t) = \rme^{- 4 \pi^2 |\bsxi|^2 t}$ the first equality is straightforward:
	\begin{equation}\label{eq:FourierHeatKernelPf1}
		\begin{split}
			\cF(\bW)(\bsxi,t) &= \cF(H)(\bsxi, \mu t) \bI + \int_{\mu t}^{(2 \mu + \lambda)t } \cF( \grad^2 H ) (\bsxi,\sigma) \, \rmd \sigma \\
			&= \rme^{- 4 \pi^2 \mu |\bsxi|^2 t} \bI - 4 \pi^2 \int_{\mu t}^{(2 \mu + \lambda)t }   (\bsxi \otimes \bsxi ) \cF( H ) (\bsxi,\sigma) \, \rmd \sigma \\
			&= \rme^{- 4 \pi^2 \mu |\bsxi|^2 t} \bI - 4 \pi^2  (\bsxi \otimes \bsxi ) \int_{\mu t}^{(2 \mu + \lambda)t }  \rme^{- 4 \pi^2 |\bsxi|^2 \sigma} \, \rmd \sigma \\
			&= \rme^{-4 \pi^2 \mu |\bsxi|^2 t} \bI + \big( \rme^{-4 \pi^2 (2\mu + \lambda) |\bsxi|^2 t} - \rme^{-4 \pi^2 \mu |\bsxi|^2 t} \big) \frac{\bsxi \otimes \bsxi}{|\bsxi|^2}\,.
		\end{split}
	\end{equation}
	For the second equality, use the binomial theorem to obtain that for each $n \in \bbN$ and any real numbers $a$ and $b$
	\begin{equation*}
		\left( a \bI + b \frac{\bsxi \otimes \bsxi}{|\bsxi|^2} \right)^n = a^n \left( \bI - \frac{\bsxi \otimes \bsxi}{|\bsxi|^2} \right) +   (a+b)^n \frac{\bsxi \otimes \bsxi}{|\bsxi|^2}\,.
	\end{equation*}
Therefore setting $a = -4 \pi^2 |\bsxi|^2 t \mu$ and $b = -4 \pi^2 |\bsxi|^2 t (\mu+\lambda)$
\begin{equation*}
	\begin{split}
		\rme^{-t \bM(\bsxi)} &= \sum_{n=0}^{\infty} \frac{\left( a \bI + b \frac{\bsxi \otimes \bsxi}{|\bsxi|^2} \right)^n }{n!} \\
		&=\sum_{n=0}^{\infty} \frac{a^n}{n!} \left( \bI - \frac{\bsxi \otimes \bsxi}{|\bsxi|^2} \right) +  \sum_{n=0}^{\infty} \frac{(a+b)^n}{n!} \frac{\bsxi \otimes \bsxi}{|\bsxi|^2} \\
		&= \rme^{a} \bI + \big( \rme^{a+b} - \rme^a \big) \frac{\bsxi \otimes \bsxi}{|\bsxi|^2} \\
		&= \rme^{-4 \pi^2 \mu |\bsxi|^2 t} \bI + \big( \rme^{-4 \pi^2 (2\mu + \lambda) |\bsxi|^2 t} - \rme^{-4 \pi^2 \mu |\bsxi|^2 t} \big) \frac{\bsxi \otimes \bsxi}{|\bsxi|^2}\,.
	\end{split}
\end{equation*}
\end{proof}

\begin{lemma}
	Let $d \geq 2$, $s \in (0,\frac{d}{2})$. Then
	\begin{equation}\label{eq:NegativePowers:FourierTransformPiece1}
		\cF^{-1} \left( - \frac{1}{(2 \pi |\bsxi|)^{2s}} \frac{\bsxi \otimes \bsxi}{|\bsxi|^2} \right) = \frac{\gamma_{d,s}}{|\bx|^{d-2s}} \left( \frac{\bx \otimes \bx}{|\bx|^2} - \frac{1}{d-2s} \bI \right) \qquad \text{ in } \scS'(\bbR^d;\bbR^{d \times d})\,,
	\end{equation}
	where $\gamma_{d,s}$ is the constant defined in \eqref{eq:Definition:LameNavierPotentialConstant}.
\end{lemma}

\begin{proof}
	Let $H$ be the heat kernel associated to $-\Delta$ defined in \eqref{eq:HeatKernelDefn}.
	For $0 < a < b$, define the matrix field
	\begin{equation*}
		\bF(\bx,t) := \int_{at}^{bt} \grad^2 H(\bx,\sigma) \, \rmd \sigma\,, \quad \bx \in \bbR^d\,,\, t > 0\,.
	\end{equation*}
	It is easy to verify using \eqref{eq:FourierHeatKernelPf1} that for every $t>0$ $\bF(\cdot,t) \in \scS(\bbR^d;\bbR^{d \times d})$, and that
	\begin{equation*}
		\wh{\bF}(\bsxi,t) := \cF(\bF(\cdot,t))(\bsxi) = (\rme^{- 4 \pi^2 |\bsxi|^2 bt} - \rme^{- 4 \pi^2 |\bsxi|^2 at}) \frac{\bsxi \otimes \bsxi}{|\bsxi|^2}\,, \quad t > 0\,.
	\end{equation*}
	Let $\bV \in \scS(\bbR^d;\bbR^{d \times d})$. Our starting point for proving \eqref{eq:NegativePowers:FourierTransformPiece1} is Parseval's relation for distributions; since $\bF(-\bx,t) = \bF(\bx,t)$
	\begin{equation*}
		\intdm{\bbR^d}{ \wh{\bF}(\bx,t) : \bV(\bx) }{\bx} = \intdm{\bbR^d}{ \bF(\bx,t) : \wh{\bV}(\bx) }{\bx}\,.
	\end{equation*}
	Multiply both sides by $t^{\frac{d}{2}-s-1}$ and integrate in $t$ from $0$ to $\infty$:
	\begin{equation}\label{eq:NegativePowers:FTransformPiece1:Pf1}
		\int_0^{\infty} t^{\frac{d}{2}-s-1} \intdm{\bbR^d}{ \wh{\bF}(\bx,t) : \bV(\bx) }{\bx} \, \rmd t = \int_0^{\infty} t^{\frac{d}{2}-s-1} \intdm{\bbR^d}{ \bF(\bx,t) : \wh{\bV}(\bx) }{\bx} \, \rmd t\,.
	\end{equation}
	The integral on the left-hand side of \eqref{eq:NegativePowers:FTransformPiece1:Pf1} converges absolutely, and by Fubini's theorem is equal to
	\begin{equation}\label{eq:NegativePowers:FTransformPiece1:Pf2}
		\intdm{\bbR^d}{ \left[ \int_0^{\infty} t^{\frac{d}{2}-s-1} (\rme^{- 4 \pi^2 |\bx|^2 bt} - \rme^{- 4 \pi^2 |\bx|^2 at})  \, \rmd t \right] \frac{\bx \otimes \bx}{|\bx|^2} : \bV(\bx) }{\bx}\,.
	\end{equation}
	Now, for $\delta > 0$, the coordinate change $r = 4 \pi^2 |\bx|^2$ gives
	\begin{equation*}
		 \int_0^{\infty} t^{\frac{d}{2}-s-1} \rme^{- 4 \pi^2 |\bx|^2 \delta t} \, \rmd t = \frac{1}{(2 \pi |\bx|)^{d-2s}} \frac{1}{\delta^{\frac{d}{2}-s}} \int_0^{\infty} r^{\frac{d}{2}-s} \rme^{-r} \frac{\rmd r}{r} = \frac{1}{(2 \pi |\bx|)^{d-2s}} \frac{\Gamma( \frac{d-2s}{2} )}{\delta^{\frac{d}{2}-s}}\,.
	\end{equation*}
	Inserting this into \eqref{eq:NegativePowers:FTransformPiece1:Pf2} with $\delta \in \{a, b\}$ leads to equality of \eqref{eq:NegativePowers:FTransformPiece1:Pf2} with the quantity
	\begin{equation}\label{eq:NegativePowers:FTransformPiece1:LHS}
		\frac{\Gamma( \frac{d-2s}{2} )}{b^{\frac{d}{2}-s}-a^{\frac{d}{2}-s}} \intdm{\bbR^d}{ \frac{1}{(2 \pi |\bx|)^{d-2s}} \frac{\bx \otimes \bx}{|\bx|^2} : \bV(\bx) }{\bx}\,.
	\end{equation}
	The right-hand side of \eqref{eq:NegativePowers:FTransformPiece1:Pf1} is equal to
	\begin{equation*}
		\intdm{\bbR^d}{ \left[ \int_0^{\infty} t^{\frac{d}{2}-s-1} 
			\int_{at}^{bt} \frac{1}{(4 \pi \sigma)^{d/2}}
				\left( \frac{\rme^{\frac{-|\bx|^2}{4\sigma}} }{4 \sigma^2 } \bx \otimes \bx - \frac{ \rme^{\frac{-|\bx|^2}{4\sigma}} }{2 \sigma} \bI \right)
			\, \rmd \sigma
			\rmd t \right] : \wh{\bV}(\bx) }{\bx}\,.
	\end{equation*}
	The integrand in $\sigma$ and $t$ is nonnegative, so by Tonelli's theorem we get
	\begin{equation*}
		\intdm{\bbR^d}{ \left[
			\int_0^{\infty} \left( \int_{\sigma/b}^{\sigma/a} t^{\frac{d}{2}-s-1} \, \rmd t \right)
			\frac{1}{(4 \pi \sigma)^{d/2}}
			\left( \frac{\rme^{\frac{-|\bx|^2}{4\sigma}} }{4 \sigma^2 } \bx \otimes \bx - \frac{ \rme^{\frac{-|\bx|^2}{4\sigma}} }{2 \sigma} \bI \right)
			\, \rmd \sigma 
		\right] : \wh{\bV}(\bx) }{\bx}\,.
	\end{equation*}
	We compute the $t$ integral and split the $\sigma$ integral:
	\begin{equation}\label{eq:NegativePowers:FTransformPiece1:Pf3}
		\begin{split}
		\frac{1}{(4 \pi)^{d/2}} \frac{1}{b^{\frac{d}{2}-s}-a^{\frac{d}{2}-s}} \frac{2}{d-2s}
		\intdm{\bbR^d}{ \left[
			\int_0^{\infty}
			\sigma^{-s}
			\left( \frac{1}{2 \sigma} \bI - \frac{1}{4 \sigma^2 } \bx \otimes \bx \right)  \rme^{\frac{-|\bx|^2}{4\sigma}}
			\, \rmd \sigma 
			\right] : \wh{\bV}(\bx) }{\bx}\,.
		\end{split}
	\end{equation}
	Now, for $\alpha > 0$, the coordinate change $r = \frac{|\bx|^2}{4 \sigma}$ gives the equality
	\begin{equation*}
		\int_0^{\infty} \sigma^{-\alpha} \rme^{-\frac{|\bx|^2}{4 \sigma}} \frac{\rmd \sigma}{\sigma} = \frac{4^{\alpha}}{|\bx|^{2\alpha}} 	\int_0^{\infty} r^{\alpha} \rme^{-r} \frac{\rmd r}{r} = \frac{4^{\alpha} \Gamma(\alpha)}{|\bx|^{2\alpha}}\,.
	\end{equation*}
	Applying this to the $\sigma$ integral in \eqref{eq:NegativePowers:FTransformPiece1:Pf3} with $\alpha \in \{ s, s+1 \}$ leads to the equality of \eqref{eq:NegativePowers:FTransformPiece1:Pf3} with
	\begin{equation}\label{eq:NegativePowers:FTransformPiece1:RHS}
		\begin{split}
			\frac{1}{(4 \pi)^{d/2}} \frac{1}{b^{\frac{d}{2}-s}-a^{\frac{d}{2}-s}} \frac{2}{d-2s}
			\intdm{\bbR^d}{ \left(
				\frac{4^{s} \Gamma(s)}{2 |\bx|^{2s}}\bI - \frac{4^{1+s} \Gamma(1+s)}{4 |\bx|^{2+2s}} \bx \otimes \bx \right)
			: \wh{\bV}(\bx) }{\bx}\,.
		\end{split}
	\end{equation}
	Now, the identity \eqref{eq:NegativePowers:FTransformPiece1:Pf1} means that the quantities \eqref{eq:NegativePowers:FTransformPiece1:LHS}  and \eqref{eq:NegativePowers:FTransformPiece1:RHS} are equal. After some algebraic simplification, this identity becomes
	\begin{multline}
			\frac{d-2s}{2} \frac{\Gamma( \frac{d-2s}{2} )}{\pi^{d/2}} \intdm{\bbR^d}{ \frac{1}{|\bx|^{d-2s}} \frac{\bx \otimes \bx}{|\bx|^2} : \bV(\bx) }{\bx} \\
			= 
			\intdm{\bbR^d}{ \left(
				\frac{4^{s} \Gamma(s)}{2 (2 \pi |\bx|)^{2s}}\bI - \frac{4^{s} \Gamma(1+s)}{(2 \pi |\bx|)^{2s}} \frac{\bx \otimes \bx}{|\bx|^2} \right)
				: \wh{\bV}(\bx) }{\bx}\,.
	\end{multline}
	Writing the Fourier transform integrals in $\bsxi$ and using the identity $t\Gamma(t)=\Gamma(t+1)$,
	\begin{multline}
		\frac{\Gamma(s)}{2 \Gamma(1+s) } \intdm{\bbR^d}{ (2 \pi |\bsxi|)^{-2s} \wh{\bV}(\bsxi) }{\bsxi} - \frac{\Gamma( \frac{d+2-2s}{2} )}{2^{2s} \pi^{d/2} \Gamma(1+s)} \intdm{\bbR^d}{ \frac{1}{|\bx|^{d-2s}} \frac{\bx \otimes \bx}{|\bx|^2} : \bV(\bx) }{\bx} \\
		= 
		\intdm{\bbR^d}{ \left(
			\frac{1}{(2 \pi |\bsxi|)^{2s}} \frac{\bsxi \otimes \bsxi}{|\bsxi|^2} \right)
			: \wh{\bV}(\bsxi) }{\bsxi}\,.
	\end{multline}
	Now we use \eqref{eq:RieszTransform:Fourier} on the first integral to get
	\begin{multline}
		\frac{\Gamma(s)}{2 \Gamma(1+s) } g_{d,s} \intdm{\bbR^d}{  \frac{1}{|\bx|^{d-2s}} \bV(\bx) }{\bx} - \frac{\Gamma( \frac{d+2-2s}{2} )}{2^{2s} \pi^{d/2} \Gamma(1+s)} \intdm{\bbR^d}{ \frac{1}{|\bx|^{d-2s}} \frac{\bx \otimes \bx}{|\bx|^2} : \bV(\bx) }{\bx} \\
		= 
		\intdm{\bbR^d}{ \left(
			\frac{1}{(2 \pi |\bsxi|)^{2s}} \frac{\bsxi \otimes \bsxi}{|\bsxi|^2} \right)
			: \wh{\bV}(\bsxi) }{\bsxi}\,.
	\end{multline}
	The proof is finished by using the identity $a\Gamma(a)=\Gamma(a+1)$ for $a > 0$ to obtain that $\frac{\Gamma(s)}{2 \Gamma(1+s)} g_{d,s} = \frac{\gamma_{d,s}}{d-2s}$.
\end{proof}

\section{Formulae for Classical Functions}

For completeness we report the following classical information for special functions along with references. Some identities hold in a more general context, but we only report what we need.

The Gamma function is extended to $\bbR \setminus \{ 0, -1, -2, \ldots \}$ via the formula $\Gamma(a+1) = a \Gamma(a)$.
We also use the Legendre duplication formula for the Gamma function
\begin{equation}\label{eq:GammaFunction}
	\Gamma(a) \Gamma \left( a + \frac{1}{2} \right) = 2^{1-2a} \sqrt{\pi} \Gamma(2a)\,.
\end{equation}

Let $\rmB : (0,\infty) \times (0,\infty) \to \bbR$ denote the Euler beta function
\begin{equation}\label{eq:BetaFunction}
	\rmB(a,b) = \int_0^1 t^{a-1} (1-t)^{b-1} \, \rmd t\,.
\end{equation}
We use the following three well-known identities:
\begin{equation}\label{eq:Item5}
	\rmB(a,b) = \rmB(b,a)\,, \qquad \rmB(a,1-a) = \frac{\pi}{\sin(a \pi)} \text{ for } a < 1\,, \qquad \rmB(a,b) = \frac{\Gamma(a) \Gamma(b)}{\Gamma(a+b)}\,.
\end{equation}

\subsection{Hypergeometric Functions}
Suppose $a$, $b$, $c \in \bbR$ with $-c \notin \bbN_0$. Let $z \in \bbC$. We denote the complex-valued \textit{hypergeometric function} defined for all $z \in \bbC \setminus \{ 1, \infty \}$ as $F(a,b;c;z)$.
For  $|z| < 1$ the series defining $F(a,b;c;z)$ converges absolutely, and for $z \in \overline{\bbC} \setminus \{1, \infty\}$ $F$ is defined via analytic continuation. 
See the discussion in \cite[Chapter 15]{abramowitz1988handbook}.

When $c > b > 0$, the hypergeometric function has the following integral representation, valid for all $z \in \bbC \setminus \{ [1,\infty) \}$:
\begin{equation}\label{eq:HyperGeo:IntRep}
	F(a,b;c;z) = \frac{\Gamma(c)}{\Gamma(b) \Gamma(c-b)} \int_0^1 t^{b-1} (1-t)^{c-b-1} (1-tz)^{-a} \, \rmd t\,;
\end{equation}
see \cite[Equation 15.3.1]{abramowitz1988handbook}. Note that in this case when $a>0$ we have the uniform bound $F(a,b;c;z) \leq F(a,b;c;0) = \frac{\Gamma(c)}{\Gamma(b)\Gamma(c-b)} \rmB(b,c-b) = 1$ for all $z < 0$.

From \cite[Equation 15.3.7]{abramowitz1988handbook} and the accompanying discussion, we see that 
for $|z| \gg 1$ with $z \notin (0,\infty)$ and whenever $a - b \notin \bbZ$, $F$
has the asymptotic expression
\begin{equation}\label{eq:HyperGeo:Asymp}
	\begin{split}
		F(a,b;c;z) &= \frac{\Gamma(c) \Gamma(b-a)}{\Gamma(b)\Gamma(c-a)} \frac{1}{(-z)^{a}} 
		+  \frac{\Gamma(c) \Gamma(a-b)}{\Gamma(a)\Gamma(c-b)} \frac{1}{(-z)^{b}} \\
		&\quad + O \left( \frac{1}{(-z)^{a+1}} + \frac{1}{(-z)^{b+1}} \right)\,.
	\end{split}
\end{equation}

\subsection{Bessel Functions}
For $\nu > -1/2$ and $z > 0$, we denote by $J_{\nu}(z)$ the \textit{Bessel functions of the first kind.} We will make use of the following integral representation
\begin{equation}\label{eq:BesselFxn:IntRep}
	J_{\nu}(z) = \frac{z^{\nu}}{2^{\nu} \sqrt{\pi} \Gamma \left( \nu + \frac{1}{2} \right) } \int_{-1}^1 (1-t^2)^{\nu - 1/2} \cos(zt) \, \rmd t\,;
\end{equation}
see \cite[Equation 9.1.20]{abramowitz1988handbook}.
With this we obtain the elementary bound
\begin{equation}\label{eq:BesselFxn:Bound}
	|J_{\nu}(z)| \leq C(\nu) |z|^{\nu} \quad \text{ for all z > 0. }
\end{equation}

Bessel functions satisfy the formula (c.f. \cite[Equation 9.1.30]{abramowitz1988handbook})
\begin{equation}\label{eq:BesselFxn:Derivative}
	\left( \frac{1}{z} \frac{d}{d z} \right)^m (z^{\nu} J_{\nu}(z) ) = z^{\nu-m} J_{\nu-m}(z)\,, \qquad m \in \bbN_0\,,\,\, \nu - m > -1/2\,.
\end{equation}

Integrals of Bessel functions against polynomials and exponential functions
can be connected to hypergeometric functions in the following way: for any $a$, $b \in (0,\infty)$, $\mu \in \bbR$ and $\nu > -1/2$ with $\mu + \nu > 0$, then
\begin{equation}\label{eq:BesselFxnInt}
	\int_0^{\infty} \rme^{-a x} J_{\nu}(bx) x^{\mu-1} \, \rmd x = \frac{b^{\nu}}{a^{\mu+\nu}} \frac{\Gamma(\nu+\mu)}{2^{\nu}\Gamma(\nu+1)} F \left( \frac{\nu+\mu}{2}, \frac{\nu+\mu+1}{2}; \nu+1 ; \frac{-b^2}{a^2} \right)\,.
\end{equation}
See \cite[6.621]{gradshteyn2014table}.
In this case we can use the formula \eqref{eq:HyperGeo:Asymp} to obtain
\begin{equation}\label{eq:BesselFxnInt:Bounded}
	\left| \int_0^{\infty} \rme^{-a x} J_{\nu}(bx) x^{\mu-1} \, \rmd x \right| \leq C(\mu,\nu) \left( \frac{1}{b^{\mu}} + \frac{a}{b^{\mu+1}} \right) + O \left(  \frac{a^{2}}{b^{\mu+2}} + \frac{a^3}{b^{\mu+3}}  \right)
\end{equation}
whenever $\frac{b^2}{a^2} \gg 1$.
If the inequality $\nu > \mu - 1$ also holds, then we can combine \eqref{eq:HyperGeo:IntRep} and \eqref{eq:BesselFxnInt} to get
\begin{equation}\label{eq:BesselFxnInt:IntRep}
	\int_0^{\infty} \rme^{-a x} J_{\nu}(bx) x^{\mu-1} \, \rmd x = \frac{1}{b^{\mu}} \frac{\Gamma(\nu+\mu)}{2^{\nu} \Gamma( \frac{\nu+\mu+1}{2}) \Gamma (\frac{\nu-\mu+1}{2} )} \int_0^1 \frac{ t^{\frac{\nu+\mu-1}{2}} (1-t)^{\frac{\nu-\mu-1}{2}} }{ (\frac{a^2}{b^2} + t)^{\frac{\nu+\mu}{2}} }\, \rmd t\,.
\end{equation}

\subsection{Two Important Formulae}

\begin{lemma}\label{lma:RadialKernel1}
	For $r > 0$ and for any unit vector $\bseta$,
	\begin{equation*}
		\intdm{\bbS^{d-1}}{ \rme^{- \imath r \bseta \cdot \bsomega} }{\sigma(\bsomega)} = (2 \pi)^{d/2} \frac{J_{\nu}(r)}{r^{\nu}}\,, \qquad \text{ where } \nu = \frac{d-2}{2}\,.
	\end{equation*}
\end{lemma}

\begin{proof}
	Since $\bsomega \mapsto \sin(r \bseta \cdot \bsomega)$ is odd, the complex part of the left-hand side integral is $0$. Using explicit integration in spherical coordinates 
	\begin{equation*}
		\intdm{\bbS^{d-1}}{ \rme^{- \imath r \bseta \cdot \bsomega} }{\sigma(\bsomega)} = \frac{2 \pi^{\frac{d-1}{2}}}{\Gamma(\frac{d-1}{2})} \int_{-1}^1 \cos(rt) (1-t^2)^{\frac{d-3}{2}} \, \rmd t\,.
	\end{equation*}
	The result then follows by rewriting the right-hand side integral using the formula \eqref{eq:BesselFxn:IntRep} for $\nu = \frac{d-2}{2}$.
\end{proof}

\begin{lemma}\label{lma:RadialKernel2}
	For $r > 0$ and for any unit vector $\bseta$,
	\begin{equation*}
		\intdm{\bbS^{d-1}}{ (\bseta \cdot \bsomega )^2 \rme^{- \imath r \bseta \cdot \bsomega} }{\sigma(\bsomega)} = (2 \pi)^{d/2} \left(  \frac{J_{\nu}(r)}{r^{\nu}} - (d-1) \frac{J_{\nu+1}(r)}{r^{\nu+1}} \right) \,, \qquad \text{ where } \nu = \frac{d-2}{2}\,.
	\end{equation*}
\end{lemma}

\begin{proof}
	We proceed similarly to Lemma \ref{lma:RadialKernel1}. Since $\bsomega \mapsto (\bseta \cdot \bsomega)^2 \sin(r \bseta \cdot \bsomega)$ is odd, the complex part of the left-hand side integral is $0$. Using explicit integration in spherical coordinates,
	\begin{equation*}
		\intdm{\bbS^{d-1}}{ \rme^{- \imath r \bseta \cdot \bsomega} }{\sigma(\bsomega)} = \frac{2 \pi^{\frac{d-1}{2}}}{\Gamma(\frac{d-1}{2})} \int_{-1}^1 t^2 \cos(rt) (1-t^2)^{\frac{d-3}{2}} \, \rmd t\,.
	\end{equation*}
	Add and subtract $\cos(rt)(1-t^2)^{\frac{d-3}{2}}$ in the integral and use \eqref{eq:BesselFxn:IntRep} for $\nu = \frac{d-2}{2}$ to get
	\begin{equation*}
		\begin{split}
			\intdm{\bbS^{d-1}}{ \rme^{- \imath r \bseta \cdot \bsomega} }{\sigma(\bsomega)} &= \frac{2 \pi^{\frac{d-1}{2}}}{\Gamma(\frac{d-1}{2})} \int_{-1}^1 (t^2 -1) \cos(rt) (1-t^2)^{\frac{d-3}{2}} \, \rmd t + (2 \pi)^{d/2} \frac{J_{\nu}(r)}{r^{\nu}} \\
			&= -\frac{2 \pi^{\frac{d-1}{2}}}{\Gamma(\frac{d-1}{2})} \int_{-1}^1 \cos(rt) (1-t^2)^{\frac{d-1}{2}} \, \rmd t + (2 \pi)^{d/2} \frac{J_{\nu}(r)}{r^{\nu}}\,.
		\end{split}
	\end{equation*}
	The result then follows by rewriting the right-hand side integral using the formula \eqref{eq:BesselFxn:IntRep} for $\nu = \frac{d}{2}$ and using the identity $\Gamma(a+1) = a \Gamma(a)$ for $a = \frac{d-1}{2}$.
\end{proof}

\bibliography{References}
\bibliographystyle{plain}

\end{document}